\documentclass[10pt]{amsart}

\usepackage[lmargin=1in,rmargin=1in,tmargin=1in,bmargin=1in]{geometry}
\RequirePackage{amsmath} 
\RequirePackage{amssymb}
\usepackage{amscd,latexsym,amsthm,amsfonts,amssymb,amsmath,amsxtra}

\usepackage[colorlinks=true,urlcolor=blue,citecolor=blue, hypertexnames=false]{hyperref}
\usepackage{color}
\usepackage[all]{xy}
\usepackage[OT2,T1]{fontenc}
\usepackage{bm}
\usepackage{mathtools}
\usepackage{ mathrsfs }
\usepackage{xcolor}
\usepackage{comment}
\usepackage{marginnote}
\usepackage{enumitem}
\usepackage{thmtools, thm-restate}

\DeclareSymbolFont{cyrletters}{OT2}{wncyr}{m}{n}
\DeclareMathSymbol{\Sha}{\mathalpha}{cyrletters}{"58}

\let\Re\undefined
\let\Im\undefined

\DeclareMathOperator{\Re}{Re}
\DeclareMathOperator{\Im}{Im}

\DeclareMathOperator{\ord}{ord}

\DeclareMathOperator{\SL}{SL}

\DeclareMathOperator{\rad}{rad}

\newcommand{\bZ}{\mathbb{Z}}

\newcommand{\bR}{\mathbb{R}}
\newcommand{\bC}{\mathbb{C}}

\newcommand{\cS}{\mathcal{S}}

\newcommand{\cP}{\mathcal{P}}

\newcommand{\cJ}{\mathcal{J}}

\newcommand{\bH}{\mathbb{H}}

\newcommand{\tmod}[1]{\ \left(\text{mod }#1\right)}
\newcommand{\legendre}[2]{\left(\frac{#1}{#2}\right)}

\newcommand{\tlegendre}[2]{\textstyle{\left(\frac{#1}{#2}\right)}}

\def\Re{\operatorname{Re}}

	\newcommand{\Mod}[1]{\ (\mathrm{mod}\ #1)}
        \newcommand{\sym}{\operatorname{sym}}

	\newcommand{\Geo}{\operatorname{Geo}}

	\newcommand{\du}{\operatorname{Dual}}

	\newcommand{\Reg}{\operatorname{Reg}}

	\newcommand{\si}{\operatorname{Sing}}

	\newcommand{\sm}{\operatorname{Small}}

	\newcommand{\RNum}[1]{\uppercase\expandafter{\romannumeral #1\relax}}

\begin{document}
\theoremstyle{plain}
	\newtheorem{thm}{Theorem}[section]
	\newtheorem{cor}[thm]{Corollary}
	\newtheorem{thmy}{Theorem}
        \newtheorem{cory}{Corollary}
	\renewcommand{\thethmy}{\Alph{thmy}}
	\newenvironment{thmx}{\stepcounter{thm}\begin{thmy}}{\end{thmy}}
        \newenvironment{corx}{\stepcounter{cor}\begin{cory}}{\end{cory}}

	\renewcommand{\thecory}{\Alph{cory}}
	\newtheorem{hy}[thm]{Hypothesis}
	\newtheorem*{thma}{Theorem A}
	\newtheorem*{corb}{Corollary B}
	\newtheorem*{thmc}{Theorem C}
        \newtheorem*{thmd}{Theorem D}
	\newtheorem{lemma}[thm]{Lemma}  
	\newtheorem{prop}[thm]{Proposition}
	\newtheorem{conj}[thm]{Conjecture}
	\newtheorem{fact}[thm]{Fact}
	\newtheorem{claim}[thm]{Claim}
	\newtheorem{question}[thm]{Question}
	
	\theoremstyle{definition}
	\newtheorem{defn}[thm]{Definition}
	\newtheorem{example}[thm]{Example}
	\theoremstyle{remark}
	
	\newtheorem{remark}[thm]{Remark}	
	\numberwithin{equation}{section}

\title[]{Second moment of central values of half-integral weight modular forms and subconvexity}
\author{Steven Creech, Henry Twiss, Zhining Wei, Peter Zenz}

\address{Kassar House, 151 Thayer St, Providence, RI 02912 USA}
\email{steven$\_$creech@brown.edu}

\address{Kassar House, 151 Thayer St, Providence, RI 02912 USA}
\email{henry$\_$twiss@brown.edu}

\address{Kassar House, 151 Thayer St, Providence, RI 02912 USA}
\email{zhining$\_$wei@brown.edu}

\address{Kassar House, 151 Thayer St, Providence, RI 02912 USA}
\email{peter$\_$zenz@brown.edu}

\begin{abstract}
We let $f$ be a half-integral weight modular form of weight $\kappa>4$ on $\Gamma_0(4)$ that is an eigenfunction of all Hecke operators $T_n$, so that $T_nf = \Lambda_f(n)n^{\frac{\kappa-1}{2}}f$. Let $\|f\|$ denote the Petersson norm of $f$. We study a weighted second moment of the central value of the $L$-function associated to $f$ over an orthogonal basis $H_\kappa(4)$ of $\cS_{\kappa}(\Gamma_0(4))$. This corresponds to studying the following sum:
\[
\sum_{f\in H_\kappa(4)}\frac{\Lambda_f(n)\vert L(1/2,f)\vert^2}{\|f\|^2}.
\]
Using the relative trace formula, we obtain an asymptotic formula for the second moment. We then use the method of amplification to get the subconvexity bound 
\[
 L(1/2,f)\ll_{\varepsilon} (\kappa^2)^{\frac{1}{4}-\frac{1}{40}+\varepsilon}.
\]
This is the first subconvexity result for half-integral weight modular forms in the weight aspect. We also apply our second moment result to get a quantitative simultaneous non-vanishing result for central values of $L$-functions.
\end{abstract}

\date{\today}
\maketitle

\tableofcontents

\section{Introduction}

The study of central values of $L$-functions is a classical theme in number theory. For many problems in number theory, the size of the central values is of particular importance and the general theme of bounding these $L$-values is commonly referred to as the \textit{subconvexity} problem. More precisely, given an $L$-function $L(s, \pi)$ satisfying analytic continuation and a functional equation, one expects:
$$L(1/2, \pi) \ll_{\varepsilon} \mathcal{C}(\pi)^{1/4-\delta+\varepsilon},$$
for some $\delta > 0.$ Here $\mathcal{C}(\pi)$ is called the analytic conductor of the $L$-function (see \cite[Chapter 5]{AnalyticNumberTheory} for a thorough treatment). 
The trivial bound corresponding to $\delta=0$ is called the convexity bound and can be achieved by applying the Phragm\'en-Lindelöf principle. 
A \textit{subconvexity} bound is an improvement of the trivial bound corresponding to some $\delta>0$. The best possible bound corresponds to $\delta=\frac{1}{4}$; this is known as the Lindel\"of Hypothesis and follows in the cases where $L(s,\pi)$ satisfies the Riemann Hypothesis. Subconvexity bounds have far-reaching applications to deep arithmetic problems, including the Quantum Unique Ergodicity conjecture and the equidistribution of Heegener points (see e.g., \cite{RudnickSarnak1994}, \cite{Duke1988} and \cite[Lecture 5]{Michel2007}).
 
This paper investigates the subconvexity problem for $L$-functions associated to modular forms of half-integral weight. These $L$-functions do not admit Euler products and are known to not satisfy the Riemann Hypothesis (see \cite{yoshida1995calculations}). 

There are examples of Dirichlet series without an Euler product that fail to be subconvex at the central point (see \cite{ConreyGhosh2006}). This would suggest that the Euler product may be essential for subconvexity. However, for automorphic $L$-functions attached to modular forms of integral weight, the Euler product arises from the fact that the underlying form is a ``simultaneous eigenfunction'' for the Hecke operators. Jeffrey Hoffstein informally conjectured at Oberwolfach in 2011 that such a ``simultaneous eigenfunction'' property is crucial for implying a Lindel\"of Hypothesis. Although half-integral weight forms lack an Euler product, they still have Hecke operators; thus, one might expect that they satisfy the Lindel\"of Hypothesis.  In \cite{Kiral2015}, Kiral studied the subconvexity problems for $L$-functions associated to half-integral weight eigenforms twisted by a character in the conductor aspect, and a corresponding second moment was later derived in \cite{DunnZaharescu2025}. Subconvexity for other $L$-functions without Euler products has also been successfully investigated, notably by Blomer in \cite{Blomer2011} and \cite{BlomerEpstein}. 
In this paper, we study the subconvexity problem for $L$-functions associated with half-integral weight modular forms in the weight aspect. In the elliptic modular form case, the subconvexity problem in the weight aspect has been extensively studied  (see, e.g., \cite{Peng2001} and \cite{Young2017}). In contrast, no subconvexity bound in the weight aspect is known for $L$-functions associated with half-integral weight modular forms as the classical method requires the Euler product. 

The goal of this paper is to establish the first such result. Instead of the classical method, we utilize the relative trace formula method to derive a second moment formula for central values of half-integral weight modular forms. This can be regarded as an average Lindel\"of result in the weight aspect. As applications, we establish the desired subconvexity bound as well as a simultaneous non-vanishing result for the central values in the weight aspect. 

\subsection{Main Results} Let $\kappa$ be a half-integer. Denote by $\cS_{\kappa}(\Gamma_0(4))$ the space of holomorphic cusp forms of weight $\kappa.$ By the Hecke theory \cite{Shimura1977}, we can find an orthogonal basis, denoted by $H_{\kappa}(4)$, whose elements are eigenfunctions for Hecke operators. Let $n$ be an odd square integer and assume that $T_nf=\Lambda_f(n)n^{\frac{\kappa-1}{2}}f.$ Denote by $\|f\|$ the Petersson norm of $f.$ 
Let $\mathbf{s}=(s_1,s_2)\in\bC^2$ and set $\mathbf{0}=(0,0).$ The first main result is the eigenvalue-weighted second moment formula for central values of half-integral weight modular forms, which is a half-integral weight version of Kuznetsov's identity. 

\begin{thmx}\label{thm. second moment closed formula}
	 Let $\kappa>4$ be a half-integer, and $n$ an odd square integer. Then
	\[n^{\frac{\kappa-1}{2}}\frac{\Gamma(\kappa/2)^2}{(2\pi)^{\kappa}}\sum_{f\in H_{\kappa}(4)}\frac{\Lambda_f(n)|L(1/2,f)|^2}{\|f\|^2}=J_{\si}(\textbf{0},n)+J_{\Reg}(\textbf{0},n),\]
	where
	\begin{flalign*}
		J_{\si}(\mathbf{s},n)&=\frac{2n^{\frac{\kappa}{2}-1}}{i^{\kappa}C_{\kappa}2^{s_1+s_2}\Gamma(\kappa)}\left(M(s_1,s_2;n)+M(-s_1,-s_2;n)\right)\\
		&\hspace{20mm}-\frac{2n^{\frac{\kappa}{2}-1}}{i^{\kappa}C_{\kappa}2^{s_1+s_2}\Gamma(\kappa)}\frac{1}{(2\pi i)^2}\oint\displaylimits_{|z_1|=1}\oint\displaylimits_{|z_2|=1}\frac{(z_1^2-z_2^2)^2(z_1z_2+s_1s_2)K(z_1,z_2;n)}{\prod_{i=1}^2\prod_{j=1}^2(z_i-s_j)(z_i+s_j)}\,dz_1\,dz_2,	
	\end{flalign*}
	\[K(s_1,s_2;n)=\frac{\Gamma(s_1+\kappa/2)\Gamma(s_2+\kappa/2)\zeta(1+s_1+s_2)}{\pi^{s_1+s_2}}\sum_{\substack{g^4ad=n\\\gcd(a,d)=1}}\frac{1}{a^{s_1}d^{s_2}g^{2s_1+2s_2}}\prod_{p|g}(1-p^{s_1+s_2}),\]
and
\[M(s_1,s_2;n)=\frac{\Gamma(s_1+\kappa/2)\Gamma(s_2+\kappa/2)}{\pi^{s_1+s_2}}\sum_{\substack{g^2ad=n\\\gcd(a,d)=1\\g\neq\square}}L(1+s_1+s_2,\chi_{g_*})g_*^{\frac{1}{2}+s_1+s_2}\frac{\varepsilon_g^{2\kappa+1}\legendre{-1}{g}}{a^{s_1}d^{s_2}g^{s_1+s_2}}\prod_{p|g_0}\left(1-\chi_{g_*}(p)p^{s_1+s_2}\right),\]
where $g_*$ is the squarefree part of $g,$ $g_0=\frac{\rad(g)}{g_*}$ are divisors of $g$, and $C_{\kappa}$ is defined in \eqref{eq. kappa constant}.

$J_{\Reg}(\textbf{0},n)$ is given in Proposition \ref{prop, the error term in RTF} and satisfies
\[J_{\Reg}(\textbf{0},n)\ll_{\varepsilon} \frac{n^{\frac{\kappa}{2}+\varepsilon}\Gamma\left(\frac{\kappa}{2}\right)^2}{|C_{\kappa}|\kappa^{\frac{1}{2}}\Gamma(\kappa)}.\]  
\end{thmx}

\begin{remark}\label{rem. main term size}
In Corollary \ref{cor. estimation of the singular orbital integral}, we will show: for any square integer $n$ and any $\varepsilon>0,$
	\[J_{\si}(\textbf{0},n)\ll_{\varepsilon} n^{\frac{\kappa}{2}-\frac{3}{4}+\varepsilon}\frac{\kappa^{\varepsilon}\Gamma\left(\frac{\kappa}{2}\right)^2}{|C_{\kappa}|\Gamma(\kappa)}.\]	
\end{remark}

\begin{remark}
    When $n=1,$ we obtain a stronger asymptotic formula for the second moment, which is Proposition \ref{prop. second moment n=1}.
\end{remark}

Let $f\in H_{\kappa}^+(4)$ be a Hecke eigenform in the Kohnen plus space. Then it has a Fourier expansion
\begin{equation}\label{eq. fourier expansion of half-integral form}
    f(z) = \sum_{n \ge 1}a_f(n)n^{\frac{\kappa-1}{2}}e(nz).
\end{equation}
We can normalize $f$ such that $a_f(n)$ are always real. Let $\Lambda_{f}(n)$ be the $n$-th Hecke eigenvalue of $f$ so that $T_{n}f = \Lambda_{f}(n)n^{\frac{\kappa-1}{2}}f$. We normalize $f$ such that, for any fundamental discriminant $D$ satisfying $(-1)^{\kappa-\frac{1}{2}}D>0,$ 
\begin{equation}\label{eq. normalization of forms}
a_f(|D|)^2=\pi^2L\left(\frac{1}{2},F\times\chi_D\right),
\end{equation}
where $F:=F_f\in H_{2\kappa-1}(1)$ is the Shimura correspondent of $f$, which is an integral-weight $2\kappa-1$ eigenform on $\SL_2(\mathbb{Z})$. The normalization is well-defined by the Kohnen-Zagier formula \cite[Theorem 1]{KohnenZagier1981}.

Let $f\in H_{\kappa}^+(4)$ be a Hecke eigenform in the Kohnen plus space. The analytic conductor of $f$ is of size $ \kappa^2.$ By the amplification method, we prove the following subconvexity result:
\begin{thmx}\label{thm. subconvexity theorem}
	Let $f\in H_{\kappa}^+(4)$ satisfying the normalization \eqref{eq. normalization of forms}. Then for any $\varepsilon>0,$ 
	\[L\left(\frac{1}{2},f\right)\ll_{\varepsilon} (\kappa^{2})^{\frac{1}{4}-\frac{1}{40}+\varepsilon}.\]
\end{thmx}

This result establishes the first subconvexity result for half-integral weight modular forms in the weight aspect. From the proof of Theorem \ref{thm. subconvexity theorem}, this subconvexity bound is true for all eigenforms $f\in H_{\kappa}(4).$

\bigskip

Applying H\"older's inequality, we establish the following simultaneous nonvanishing result for the central values:
\begin{thmx}\label{thm. nonvanishing theorem}
	 Let $D$ be a fixed fundamental discriminant satisfying $(-1)^{\kappa-\frac{1}{2}}D>0.$ Then for any $\varepsilon>0,$
\[\#\{f\in H_{\kappa}^+(4):L(1/2,f)L(1/2,F\times\chi_D)\neq0\}\gg\frac{\kappa}{(\log\kappa)^{3+\varepsilon}}.\]
\end{thmx}
Prior to Theorem \ref{thm. nonvanishing theorem}, only the existence of half integral Hecke eigenforms with nonvanishing central values was known (see \cite{RamakrishnanShankhadhar2014,KohnenRaji2017}). Our result provides the first quantitative lower bound. Additionally, this implies that, for $D$ a fundamental discriminant satisfying $(-1)^{\kappa-\frac{1}{2}}D>0,$ the number of integral weight eigenforms $F\in H_{2\kappa-1}(1)$ for which $L(1/2,F\times\chi_D)\neq0$ is at least of  order $\frac{\kappa}{(\log  \kappa)^{3+\varepsilon}}$.

\subsection{Strategy of the Proof}

In the integral weight set up, the eigenvalue-weighted second moment, which is known as Kuznetsov's identity (see, e.g., \cite[Theorem 4.2]{BalkanovaFrolenkov21}), has been extensively studied. The main idea is to apply the Hecke relations to produce divisor functions, and subsequently apply the Petersson trace formula and the divisor-type Voronoi summation formula. For the details in this classical set up, one can refer to \cite[Section 4]{BalkanovaFrolenkov21}.

This method is difficult to adapt to the half-integral weight setting because the associated $L$-functions lack Euler products and the corresponding multiplicative relations. Consequently, divisor functions do not emerge naturally, and the classical method fails. In this paper, we instead employ the "relative trace formula" (RTF) to establish a half-integral weight Kuznetsov's identity, which is Theorem \ref{thm. second moment closed formula}.

The RTF method has been extensively studied in analytic number theory (see, e.g., \cite{RR05,FeigonWhitehouse2009,Yan23a}). It provides an effective tool for bypassing certain intermediate steps in the classical derivation of Kuznetsov's identity. The main idea of the RTF method is to integrate the pre-trace formula (see, e.g., \cite[Theorem 1]{Zagier1977}) along “proper subgroups.” In our setting, the proper subgroup corresponds to the diagonal group in $\operatorname{PGL}_2$, or, classically, the positive imaginary axis. Integrating the spectral side of the pre-trace formula--often referred to as the \textit{spectral side} of the RTF--yields the desired second moment. Consequently, it suffices to analyze the integral of the geometric side, commonly called the \textit{geometric side} of the RTF. By applying the Bruhat decomposition, one finds that the geometric side decomposes into a sum of integrals, known as orbital integrals. A key difference between the representation-theoretic RTF and our approach arises in the half-integral weight setting: here, the orbital integrals cannot generally be expressed as products of local orbital integrals. Nevertheless, one can employ a classical version of the RTF, previously studied in the integral weight case in \cite{Wei2025}.

We give a brief overview of the proof focusing on the essential details. Notice that the $L$-functions of half-integral weight modular forms can be defined by the Mellin transformation \eqref{eq. integral representation for L functions}. This motivates us to consider the following double integral:
\[J(\mathbf{s},n)=\int_0^{\infty}\int_0^{\infty}\sum_{f\in H_{\kappa}(4)}\frac{(T_nf)(iy_1)\overline{f(iy_2)}}{\|f\|^2}y_1^{s_1+\kappa/2}y_2^{s_2+\kappa/2}\frac{\,dy_1}{y_1}\frac{\,dy_2}{y_2},\]
where $n$ is an odd square integer, and $T_n$ is the half-integral weight Hecke operator. Then by \eqref{eq. integral representation for L functions}, $J(\textbf{0},n)$ is the left hand side of Theorem \ref{thm. second moment closed formula}. Similar to \cite{Zagier1977}, we can deduce the reproducing kernel for the space $\cS_{\kappa}(\Gamma_0(4))$, which implies:
\[J(\mathbf{s},n)=C_{\kappa}^{-1}n^{\kappa-1}\int_0^{\infty}\int_0^{\infty}\sum_{\gamma\in G_4(n)}\frac{y_1^{s_1+\kappa/2}y_2^{s_2+\kappa/2}}{R_{\gamma}(iy_1,iy_2)}\frac{\,dy_1}{y_1}\frac{\,dy_2}{y_2}.\]
Here $G_4(n)=\Gamma_0(4)\begin{psmallmatrix}
	1&\\&n
\end{psmallmatrix}\Gamma_0(4)$ and $C_{\kappa}$ and $R_{\gamma}(z,z')$ are defined in \S \ref{subsec. kernel function}.  Depending on if $\gamma\in G_4(n)$ is a triangular matrix or not, we write $G_4(n)$ as a disjoint union and hence express $J(\mathbf{s},n)$ as the sum of singular orbital integrals $J_{\si}(\textbf{s},n)$ and regular orbital integrals $J_{\Reg}(\mathbf{s},n)$. 

The singular orbital integral can be computed directly and will contribute the main term (Proposition \ref{prop. main term for the general case}). For the regular orbital integral, we write them as sums of hypergeometric functions, which will later be transformed to Legendre functions. The upper bounds of the Legendre functions in different regions will imply that the regular orbital integrals contribute the error term (Proposition \ref{prop. est of error general}).

\bigskip 

For Theorem \ref{thm. subconvexity theorem}, we apply the amplification method, which was established in \cite{IwaniecSarnak1995}. Here we use the version in \cite[Section 11.2.2]{WeiYangZhao2024}. Instead of using the Fourier coefficients of half-integral weight modular forms, we utilized the Fourier coefficients of their Shimura correspondents as the amplifier. This new amplifier works as the Shimura correspondent is Hecke equivariant, which can also be regarded as an evidence for Hoffstein's ``simultaneous eigenfunction'' philosophy.

\bigskip

For Theorem \ref{thm. nonvanishing theorem}, we apply H\"older's inequality for the first moment. Then Theorem \ref{thm. second moment closed formula} will yield the desired result. Notice that for Theorem \ref{thm. nonvanishing theorem}, we only need $n=1,$ and Proposition \ref{prop. second moment n=1} provides a better error term in this case.

\bigskip

Finally, we outline the organization of the paper:
\begin{itemize}
	\item In \S \ref{sec. pre}, we do the preparation work.  In \S \ref{subsec. Hecke}, we briefly introduce the Hecke theory of half-integral weight modular forms. In \S \ref{subsec. kernel function}, we prove Proposition \ref{prop. kernel function}, which can be regarded as a Zagier's kernel function in the half-integral weight set up. Then we set up the integral $J(\mathbf{s},n)$ and express it as the summation of singular orbital integrals $J_{\si}(\mathbf{s},n)$ and the regular orbital integrals $J_{\Reg}(\mathbf{s},n).$
	\item In \S \ref{sec. singular}, we study the singular orbital integrals $J_{\si}(\mathbf{s},n)$, which will be written as the summation of small cell orbital integral $J_{\sm}(\mathbf{s},n)$ and the dual orbital integral $J_{\du}(\mathbf{s},n).$ They will be treated separately in \S \ref{subsec. small cell} and \S \ref{subsec. dual cell}. In \S \ref{subsec. summary of singular}, we combine the small cell orbital integral and the dual orbital integral to deduce Proposition \ref{prop. main term for the general case}. We also give the estimation of the singular orbital integral, which is Corollary \ref{cor. estimation of the singular orbital integral}.
	\item In \S \ref{sec. regular}, we study the regular orbital integrals $J_{\Reg}(\mathbf{s},n),$ which is a summation of three types of regular orbital integrals. We study them separately in \S \ref{subsec. reg 1}, \S \ref{subsec. reg 2} and \S \ref{subsec. reg 3}. The explicit regular orbital integral $J_{\Reg}(\mathbf{0},n)$ is stated in Proposition \ref{prop, the error term in RTF}. In \S \ref{subsec. majorization}, we show that $J_{\Reg}(\mathbf{0},n)$ only contribute the error term, which is Proposition \ref{prop. est of error general}.
	\item In \S \ref{sec. proof}, we prove the main results. In \S \ref{subsec. proof A}, we prove Theorem \ref{thm. second moment closed formula}. In \S \ref{subsec. proof C}, we prove Theorem \ref{thm. subconvexity theorem}. In \S \ref{subsec. proof B}, we prove Theorem \ref{thm. nonvanishing theorem}.
\end{itemize}

\subsection{Notation Guide}\label{subsec. notations}

\subsubsection{The half-integral weight modular forms} Denote by $\Gamma_0(4)$ the Hecke congruence subgroup of $\SL_2(\bZ)$ of level $4.$  Here we introduce the half-integral weight automorphic forms on $\Gamma_{0}(4)\backslash\mathbb{H}$. For any $\gamma=\begin{psmallmatrix}
	a&b\\c&d
    \end{psmallmatrix} \in \Gamma_0(4)$, we set
    \[
        j_{\gamma}(z)=\legendre{c}{d}\varepsilon_{d}^{-1}\sqrt{cz+d},
    \]
    where we take the principal branch of the square-root, $\tlegendre{c}{d}$ is the Jacobi symbol with the stipulations $\tlegendre{c}{d} = \tlegendre{-c}{-d}$ and $\tlegendre{0}{d} = 1$, and $\varepsilon_{d} = 1,i$ depending on whether $d \equiv 1,3 \tmod{4}$, respectively. Moreover, we set $j_{\gamma}(z)=j_{-\gamma}(z)$ if $d<0.$ This makes $j_{\gamma}(z)$ into the theta multiplier. We also set
    \[j(\gamma,z)=cz+d.\]
    Then $|j_{\gamma}(z)|^2=|j(\gamma,z)|$. Let $\kappa$ be a positive half-integer. A holomorphic form of half-integral weight $\kappa$ is a holomorphic function $f:\bH\to\bC$ satisfying 
    \begin{equation}\label{eq. automorphy of half-integral weight form}
        f(\gamma z)=j_{\gamma}(z)^{2\kappa}f(z),
    \end{equation}
    for all $\gamma \in \Gamma_{0}(4)$, and such that $f(\alpha z) \ll 1$ as $y \to \infty$ for all $\alpha \in \mathrm{GL}_{2}^{+}(\mathbb{Q})$. We say $f(z)$ is a cusp form if the last condition can be replaced by $f(\alpha z) \to 0$. Denote by $\cS_{\kappa}(\Gamma_0(4))$ the vector space of holomorphic cusp forms of weight $\kappa.$ This is a Euclidean space endowed with the Petersson inner product. That is, for 
    $f,g\in\cS_{\kappa}(\Gamma_0(4))$ their inner product is given by
    \[\langle f,g\rangle=\int_{\Gamma_0(4)\backslash\bH}f(z)\overline{g(z)}y^{\kappa}\,d\mu(z).\]
\subsubsection{The $L$-functions} For any $f\in\cS_{\kappa}(\Gamma_0(4))$, we associate the $L$-function
    \[L(s,f)=\sum_{n=1}^{\infty}\frac{a_f(n)}{n^s}.\]
    This Dirichlet series is absolutely convergent when $\Re(s)>1$.  Moreover, $L(s,f)$ admits analytic continuation to $\mathbb{C}$ and can be expressed via the following Mellin transform:
    \begin{equation}\label{eq. integral representation for L functions}
    \int_0^{\infty}f(iy)y^{\frac{\kappa}{2}+s}\frac{\,dy}{y}=(2\pi)^{-(s+\frac{\kappa}{2})}\Gamma\left(s+\frac{\kappa}{2}\right)L\left(\frac{1}{2}+s,f\right).
    \end{equation}
\subsubsection{The transformations of matrices}\label{subsubsec: transformation-matrices} Let $\gamma=\begin{psmallmatrix}
    a&b\\c&d
    \end{psmallmatrix}$ be an integral $2\times 2$ matrix. We define
    \begin{equation}\label{eq. matrix notations}
    \widehat{\gamma}=\begin{pmatrix}
    	-a&b\\c&-d
    \end{pmatrix},\quad \overline{\gamma}=\begin{pmatrix}
    	-a&-b\\c&d
    \end{pmatrix},\quad \underline{\gamma}=\begin{pmatrix}
    	a&-b\\c&-d
    \end{pmatrix},\quad {}^{\iota}\gamma=\begin{pmatrix}
    	a&-b\\c&d
    \end{pmatrix},\quad \gamma^{\iota}=\begin{pmatrix}
    	-a&-b\\c&-d
    \end{pmatrix}.
    \end{equation}
    These matrices will be used in \S \ref{sec. regular}, see Remark \ref{rem: matrix-discussion-sec4} for a discussion on the motivation behind these matrices.

\subsubsection{Hypergeometric functions and the beta function}\label{subsubsec.hypergeo} Denote by $F(\alpha,\beta;\gamma;z):={}_2F_1(\alpha,\beta;\gamma;z)$ the hypergeometric function, and by ${}_1F_1(\alpha,\gamma;z)$ the confluent hypergeometric function. Let $B(x,y)=\frac{\Gamma(x)\Gamma(y)}{\Gamma(x+y)}$ be the beta function. Let $s$ be a complex number, and we introduce:
$$\mathcal{B}(s):=\mathcal{B}(s;n,\kappa)=\frac{2n^{\kappa-1}B(s+\kappa/2,-s+\kappa/2)}{i^{s+\kappa/2}C_{\kappa}}.$$
\subsubsection{The upper and lower triangular matrices}\label{subsubsec.upper and lower triangualr matrices} Fix an odd square integer $n$. Throughout we will consider upper and lower triangular matrices of the form
    \[
        \gamma^{a,d,m} = \begin{pmatrix} a & m \\ & d \end{pmatrix} \quad \text{and} \quad \gamma_{a,d,m} = \quad \begin{pmatrix} a & \\ m & d \end{pmatrix}
    \]
    subject to the conditions $ad = n$, $m\Mod{d}$, and $\gcd(a,d,m) = 1$. In addition, for the lower triangular matrices we assume $m \neq 0$ and $m \equiv 0 \tmod{4}$. With these matrices, it will be prudent to define some associated notation: 
     \[
        g = \gcd(a,d), \quad g_{1} = \gcd(m,d), \quad h=\gcd(a,m), \quad \text{and} \quad h_1=\gcd\left(h,\frac{ad}{h}\right).
    \]
   Let $g$ be any positive integer (not necessarily the greatest common divisor of $a,d$). We can write: 
    \[
     g = g_{\ast}(g')^{2}, \quad \text{and} \quad g_{\ast} = p_{1} \cdots p_{s},
    \]
    so that $g_{\ast}$ is the square-free part of $g$ made up of unique primes $p_{1},\ldots,p_{s}$ and $g'$ is the square part of $g$. Then $\chi_g=\legendre{\cdot}{g}$ is induced from the primitive character $\chi_{g_*}=\chi_{p_1\cdots p_s}$ with conductor $g_{\ast}=p_1\cdots p_s.$ Note that $g_{\ast}$ and $g'$ may not be relatively prime. Moreover, we also set
    \[
        \rad(g)=\prod_{p|g}p, \quad g_{0} = \frac{\rad(g)}{g_{\ast}},  \quad \text{and} \quad g_{2} = \frac{g}{\rad(g)}.
    \]

\textbf{Acknowledgments} We would like to express our gratitude to Alexander Dunn, Jeffrey Hoffstein, Wenzhi Luo, and Liyang Yang for taking a look at the first version of this document and giving comments, encouragement, and valuable suggestions.

\section{Preliminaries}\label{sec. pre}
\subsection{The Complex Logarithm and Complex Power Functions} In what follows, we always take the principal branch of the logarithm. That is, in polar coordinates $z=re^{i\theta}\in\bC$ with $\theta\in(-\pi,\pi]$ and $r>0.$ Denote by $\log$ the complex logarithm function, and its domain is $\bC-(-\infty,0].$ When the complex logarithm function is defined, we can also define the complex power function with the same domain by
\[z^{\alpha}:=e^{\alpha\log z},\]
where $\alpha\in\bC.$ We have the following results for the complex power function:
\begin{lemma}\label{lemma. complex power st}
    Let $z_1=r_1e^{i\theta_1}$ and $z_2=r_2e^{i\theta_2}$  be two complex numbers with $\theta_1,\theta_2\in(-\pi,\pi)$. Then
   \[z_1^{\alpha}z_2^{\alpha}=\begin{cases}
   	(z_1z_2)^{\alpha}&\mbox{if $\theta_1+\theta_2\in(-\pi,\pi)$},\\
   	(z_1z_2)^{\alpha}e^{2\alpha\pi i}&\mbox{if $\theta_1+\theta_2>\pi$},\\
   	(z_1z_2)^{\alpha}e^{-2\alpha\pi i}&\mbox{if $\theta_1+\theta_2<-\pi$}.\\
   \end{cases}\]   
\end{lemma}
Finally, for $z = re^{i\theta} \in\bC-(-\infty,0]$ (so that $\theta \in (-\pi,\pi)$), we let $\arg z = \theta$ denote the argument of $z$.
\subsection{The Hecke Operators} \label{subsec. Hecke}
In this section, we introduce the Hecke operators on $\cS_{\kappa}(\Gamma_0(4)).$ We define
\[M_4^+:=\left\{\gamma=\begin{pmatrix}
	a&b\\c&d
\end{pmatrix}: a,b,c,d\in\bZ,4|c,ad-bc>0\right\}.\]
Then we have the decomposition
\[M_4^+=\bigsqcup_{n=1}^{\infty}M_4^+(n),\]
where
\[M_4^+(n):=\left\{\gamma=\begin{pmatrix}
	a&b\\c&d
\end{pmatrix}: a,b,c,d\in\bZ,4|c,ad-bc=n\right\}.\]
Notice that $M_4^+(1)=\Gamma_0(4).$ The following proposition allows us to extend the definition of $j_\gamma(z)$ from $\Gamma_0(4)$ to all of $M_4^+$. It can be deduced from \cite[Section 1]{Shimura1977}.
\begin{prop}\label{prop:t_gamma_for_any_matrix}
For any $\gamma\in M_4^+$, there exists a function $j_{\gamma}(z)$ such that the following properties hold:
\begin{itemize}
	\item Cocycle condition:
	\begin{equation}\label{eq. cocyle condition}
		j_{\gamma_1\gamma_2}(z)=j_{\gamma_1}(\gamma_2 z)j_{\gamma_2}(z).
	\end{equation}

	\item Compatibility condition: if $\gamma=\begin{psmallmatrix}
	a&b\\c&d
\end{psmallmatrix}\in\Gamma_0(4),$ then
\begin{equation}\label{eq. compatibility condition}
		j_{\gamma}(z)=\legendre{c}{d}\varepsilon_d^{-1}\sqrt{cz+d}.
	\end{equation}

\item Modulus condition:  for $\gamma\in M_4^+,$
\begin{equation}\label{eq. modulus condition}
		j_{\gamma}(z)=t_{\gamma}\sqrt{cz+d},
	\end{equation}
with $t_{\gamma}\in\{\pm 1,\pm i\}$ so that $|t_{\gamma}|=1$. Moreover, $t_{\gamma}$ is independent of the choice of $z.$
\end{itemize}
\end{prop}
\begin{remark}
	Our choice of $j_{\gamma}(z)$ differs by a factor of $\det(\gamma)^{1/4}$ from that in  \cite{Shimura1977}.
\end{remark}

Let $n$ be an odd integer. The following two lemmas will be used to prove a coset decomposition of the form
\begin{equation}\label{equ:coset_decomp_for_Hecke}
G_4(n):=\Gamma_0(4)\begin{pmatrix}
	1&\\&n
\end{pmatrix}\Gamma_0(4)=\bigsqcup_{j=1}^r\Gamma_0(4)\gamma_j,\end{equation}
for some matrices $\gamma_j \in M_{4}^{+}$ and $r \ge 1$. This coset decomposition is essential as it allows us to define the Hecke operators. Notice that $G_4(n)$ is not the same as $M_4^+(n)$ and $G_{4}(1) = \Gamma_{0}(4)$. We first prove the following useful lemma:
\begin{lemma}\label{lem. decom lem in linear algebra}
Let $n$ be an odd square integer. Let $\begin{psmallmatrix}
a&m\\&d	
\end{psmallmatrix}\in M_4^+(n)$ be such that $\gcd(a,d,m)=1.$ Then we can find $\gamma_1,\gamma_2\in \Gamma_0(4)$ satisfying
\[\begin{pmatrix}
a&m\\&d	
\end{pmatrix}=\gamma_1\begin{pmatrix}
1&\\&ad	
\end{pmatrix}\gamma_2.\]
Moreover, $\gamma_1$ and $\gamma_2$ satisfy \eqref{eq. explicit smith normal form}.
\end{lemma}

\begin{proof}
	Recall from \S \ref{subsubsec.upper and lower triangualr matrices} that $g_1=\gcd(m,d)$, so we can find $x,y\in\bZ$ such that $mx+dy=g_1.$ We claim that $x$ may be assumed to be a positive prime satisfying $\gcd(x,d)=1$ and $x\equiv1\Mod{4}.$ We will assume this condition on $x$ for now and delay its proof to the end. Since $\gcd(a,d,m)=1$ and $ad=n$ is odd, we have $\gcd(ax,4g_1)=1$ (notice that $\gcd(x,d)=1$ implies $(x,g_1)=1$). Then we can find integers $s$ and $r$ such that $axs-4g_1r=1.$ Now let $\gamma_{1},\gamma_{2} \in \Gamma_{0}(4)$ be the matrices
    \[
    \gamma_1=\begin{pmatrix}
		as-4rm&-y\\-4dr&x
	\end{pmatrix} \quad \text{and} \quad \gamma_2=\begin{pmatrix}
		ax&g_1\\4r&s
	\end{pmatrix},
    \]
    and note that
	\begin{equation}\label{eq. explicit smith normal form}
	\begin{pmatrix}
    a&m\\&d	
    \end{pmatrix}=\gamma_{1}\begin{pmatrix}
    1&\\&ad	
    \end{pmatrix}\gamma_{2} = \begin{pmatrix}
    		as-4rm&-y\\-4dr&x
    	\end{pmatrix}\begin{pmatrix}
    1&\\&ad	
    \end{pmatrix}\begin{pmatrix}
    		ax&g_1\\4r&s
	\end{pmatrix},
	\end{equation}
as desired.

Finally, we prove the claim on the conditions we assumed that $x$ satisfied. Our $x$ and $y$ from above satisfy $\frac{m}{g_1}x+\frac{d}{g_1}y=1.$ Let $(x_0,y_0)$ be a solution to $\frac{m}{g_1}x+\frac{d}{g_1}y=1.$ Then $\gcd\left(x_0,\frac{d}{g_1}\right)=1$. By elementary number theory, $x=x_0+t\frac{d}{g_1}$ and $y=y_0-t\frac{m}{g_1}$ for $t \in \mathbb{Z}$ are all the solutions of the linear equation $\frac{m}{g_1}x+\frac{d}{g_1}y=1.$ As $\gcd\left(x_0,\frac{d}{g_1}\right)=1$, Dirichlet's theorem on primes in arithmetic progressions implies that there are infinitely many $t$ such that $x=x_0+t\frac{d}{g_1}$ is prime. Then we can choose large enough $t$ such that $x>0$, $\gcd(x,d)=1$, and $x\equiv1\Mod{4}$.
\end{proof}

We can use Lemma \ref{lem. decom lem in linear algebra} to furnish different decompositions of $G_4(n):$

\begin{lemma}\label{lem. total set of the Hecke operator}
Let $n$ be an odd square integer. Then the following sets of matrices are the same:  
\[
G_4(n)=\Gamma_0(4)\begin{pmatrix}
	1&\\&n
\end{pmatrix}\Gamma_0(4),\]
\[
G_4'(n)=M_4^+(n)\bigcap \left\{\gamma=\begin{pmatrix}
	a&b\\c&d
\end{pmatrix}\in M_4^+: \gcd(a,b,c,d)=1\right\},\]
and
\[G_4''(n)=\bigsqcup\limits_{\substack{ad=n\\ m\Mod{d}, \\ \gcd(a,d,m)=1}}\Gamma_0(4)\begin{pmatrix}
		a&m\\&d
	\end{pmatrix}.\]
\end{lemma}

\begin{proof}
To prove this, it suffices to show the following chain of inclusions $G_4(n)\subseteq G_4'(n)\subseteq G_4''(n)\subseteq G_4(n)$. 

We begin by showing that $G_4(n)\subseteq G_4'(n)$. 	Take $\gamma\in G_4(n).$ Then we can find $\begin{psmallmatrix}
	a&b\\c& d\end{psmallmatrix}, \begin{psmallmatrix}
	x&y\\z& w\end{psmallmatrix}\in\Gamma_0(4)$ such that
    \[\gamma=\begin{pmatrix}a&b\\c& d\end{pmatrix}\begin{pmatrix}1&\\& n\end{pmatrix}\begin{pmatrix}
    	x&y\\z& w\end{pmatrix}=\begin{pmatrix}
    	ax+bzn&ay+bwn\\cx+dzn&cy+dwn\end{pmatrix}\in M_4^+(n).\]
    Assume that $p|\gcd(ax+bzn,ay+bwn,cx+dzn,cy+dwn).$ Then $p$ divides the determinant of $\gamma$ and hence $p|n$, which implies $p|\gcd(ax,ay,cx,cy).$ Without loss of generality, we can assume that $p|a$ and $p\nmid c$, as $\begin{psmallmatrix}
	a&b\\c& d\end{psmallmatrix}\in\Gamma_0(4)$ has determinant $1$. This forces $p|\gcd(x,y)$ which contradicts that $\begin{psmallmatrix}
	x&y\\z& w\end{psmallmatrix}\in\Gamma_0(4)$ having determinant $1$. Therefore $\gcd(ax+bzn,ay+bwn,cx+dzn,cy+dwn) = 1$ as desired.
    
Now we show that $G_4'(n)\subseteq G_4''(n)$. This is obvious upon noting that
    \[M_4^+(n)=\bigsqcup\limits_{\substack{ad=m\\ m\Mod{d}}}\Gamma_0(4)\begin{pmatrix}
		a&m\\&d
	\end{pmatrix}.\]
Thus, adding the $\gcd$ condition makes it clear that $G_4'(n)\subseteq G_4''(n)$.

Finally, we shall show that $G_4''(n)\subseteq G_4(n)$. It suffices to show that if $\begin{psmallmatrix}
		a&m\\&d
	\end{psmallmatrix}$ satisfies $\gcd(a,d,m)=1$ then we can find $\gamma_1,\gamma_2\in\Gamma_0(4)$ such that $\begin{psmallmatrix}
		a&m\\&d
	\end{psmallmatrix}=\gamma_1\begin{psmallmatrix}
		1&\\&ad
	\end{psmallmatrix}\gamma_2.$ This is exactly what was proved in Lemma \ref{lem. decom lem in linear algebra}. 
\end{proof}

With the decomposition in \eqref{equ:coset_decomp_for_Hecke} in hand, we define the Hecke operator $T_{n}$ by
\[T_{n}f(z)=n^{\frac{\kappa}{2}}\sum_{j=1}^rj_{\gamma_j}(z)^{-2\kappa}f(\gamma_j z).\]
When $n$ is not a square number, Shimura \cite[Page 450]{Shimura1977} showed that the Hecke operator vanishes.

Similar to the integral weight Hecke theory, we can find a Hecke orthogonal basis $H_{\kappa}(4)$. In \cite{Shimura1977,Niwa1975}, a correspondence between $\cS_{\kappa}(\Gamma_0(4))$ and $\cS_{2\kappa-1}(\Gamma_0(2))$ is given, which is now known as the Shimura correspondence. For $f\in\cS_{\kappa}(\Gamma_0(4))$, denote by $F:=F_f\in\cS_{2\kappa-1}(\Gamma_0(2))$ the Shimura correspondent of $f.$ The Shimura correspondence sends a Hecke eigenform of $\cS_{\kappa}(\Gamma_0(4))$ to that of $\cS_{2\kappa-1}(\Gamma_0(2))$. We write $\widetilde{T}_n$ for the $n$-th Hecke operator on $\cS_{2\kappa-1}(\Gamma_0(2))$. Then for a Hecke eigenform $F$ of $\cS_{2\kappa-1}(\Gamma_0(2))$, $\widetilde{T}_nF=\lambda_F(n)n^{\kappa-1}F.$ By \cite{Kohnen1982} (or \cite[Theorem 1.1, 1.2]{Purkait2014}), the Hecke operators on the half-integral weight modular forms are multiplicative. Moreover, if $n=p^{2\ell}$, then 
\begin{equation}\label{eq. hecke action on half-integral weight}
T_{p^{2\ell}}f=\Lambda_f(p^{2\ell})f=(\lambda_F(p^{\ell})p^{(\kappa-1)\ell}-\lambda_{F}(p^{\ell-2})p^{\ell(\kappa-1)-1})f.
\end{equation}

 \subsection{The Kernel Function}\label{subsec. kernel function}

The goal of this section is to set up the kernel function that we will examine. In particular, we shall prove an analog of \cite[Theorem 1]{Zagier1977} for half-integers. We begin by setting some notation. For $\gamma\in M_4^+,$ we define the functions
\begin{equation}\label{eq. R function}
R_{\gamma}(z,z')=j_{\gamma}(z)^{2\kappa}(\gamma z+z')^{\kappa},
\end{equation}
and
\begin{equation}
	h_n(z,z')=\sum_{\gamma\in G_4(n)}\frac{1}{R_{\gamma}(z,z')}.
\end{equation}
Notice that, when $n=1,$ $G_4(1)=\Gamma_0(4).$ 
We also set
\begin{equation}\label{eq. kappa constant}
C_{\kappa}=\frac{(-i)^{\kappa}\pi}{2^{\kappa-3}(\kappa-1)}.
\end{equation}

\begin{prop}\label{prop. kernel function}
	Let $H_{\kappa}(4)$ be an orthogonal basis for $\cS_{\kappa}(\Gamma_0(4))$. Let $n\geq1$ be an odd square integer. Then for $z,z'\in\bH$
	\begin{equation}\label{eq. general kernel function}
	\sum_{f\in H_{\kappa}(4)}\frac{(T_nf)(z)\overline{f(z')}}{\|f\|^2}=\frac{n^{\kappa-1}}{C_{\kappa}}h_{n}(z,-\overline{z'}).
	\end{equation}
    In particular, if we set $z=iy_1$ and $z'=iy_2:$
    \begin{equation}\label{eq. kernel function}
	\sum_{f\in H_{\kappa}(4)}\frac{(T_nf)(iy_1)\overline{f(iy_2)}}{\|f\|^2}=\frac{n^{\kappa-1}}{C_{\kappa}}h_{n}(iy_1,iy_2).
	\end{equation}
\end{prop}

\begin{proof}
Our proof begins like that of Zagier \cite[Theorem 1]{Zagier1977}, but diverges when we examine the integral that we denote $I(y)$ as Zagier uses Cauchy's theorem which can only be done in the integral weight case. 

Just like Zagier, it suffices to prove the statement for $n=1.$ As for larger $n$ applying the Hecke operator gives the result. Moreover, it suffices to show that for $f\in\cS_{\kappa}(\Gamma_0(4)),$ we have
\begin{equation}\label{eq. main identity to prove kernel function}
(f* h_1)(z')=\int_{\Gamma_0(4)\backslash\bH}f(z)\overline{h_1(z,-\overline{z'})}(\Im z)^{\kappa}\,d\mu(z)=\overline{C_{\kappa}}f(z').
\end{equation}
By \eqref{eq. automorphy of half-integral weight form}, we have
\[\overline{j_{\gamma}(z)}^{2\kappa}f(z)(\Im z)^{\kappa}=(\Im \gamma z)^{\kappa}f(\gamma z),\]
for $\gamma\in\Gamma_0(4).$ Summing over all matrices $\gamma\in\Gamma_0(4)$ yields
\[f(z)\overline{h_1(z,z')}(\Im z)^{\kappa}=\sum_{\gamma\in\Gamma_0(4)}(\gamma \overline{z}+\overline{z'})^{-\kappa}f(\gamma z)(\Im \gamma z)^{\kappa}.\]
We substitute this into \eqref{eq. main identity to prove kernel function} to obtain
\[(f* h_1)(z')=\int_{\Gamma_0(4)\backslash\bH}f(z)\overline{h_1(z,-\overline{z'})}(\Im z)^{\kappa}\,d\mu(z)=\int_{\Gamma_0(4)\backslash\bH}\sum_{\gamma\in\Gamma_0(4)}(\gamma\overline{z}-z')^{-\kappa}f(\gamma z)(\Im \gamma z)^{\kappa}\,d\mu(z).\]
Then unfolding the integral gives us that 
\begin{equation}\label{eq. middle step to prove the kerenl function}
(f* h_1)(z')=2\int_0^{\infty}\int_{-\infty}^{\infty}\frac{f(x+iy)y^{\kappa-1}}{(x-iy-z')^{\kappa}}\,dx\frac{\,dy}{y}=2\int_0^{\infty}y^{\kappa-1}I(y)\frac{\,dy}{y},
\end{equation}
where
\[I(y)=\int_{-\infty}^{\infty}\frac{f(x+iy)}{(x-iy-z')^{\kappa}}\,dx.\]
We insert the Fourier expansion of $f$ into the integral and obtain
\[I(y)=\sum_{n=1}^{\infty}a_f(n)n^{\frac{\kappa-1}{2}}e^{-2\pi ny}\int_{-\infty}^{\infty}\frac{e(nx)}{(x-iy-z')^{\kappa}}\,dx.\]
Now, by the change of variable $x\mapsto-x$, and multiplying the integral by $(-i)^\kappa$, the integral will be of the form \cite[Equation 3.382(7)]{GradshteynRyzhik2007} with $\nu=\kappa$, $p=2\pi n$, and $\beta=y-iz'$, so we get that 
\[\int_{-\infty}^{\infty}\frac{e(nx)}{(x-iy-z')^{\kappa}}\,dx=2\pi i\frac{e(n(iy+z'))(2\pi in)^{\kappa-1}}{\Gamma(\kappa)}.\]

\noindent Plugging this into $I(y)$ yields
\[I(y)=\frac{(2\pi i)^{\kappa}}{\Gamma(\kappa)}\sum_{n=1}^{\infty}a_f(n)n^{\frac{\kappa-1}{2}}e(nz')e^{-4\pi ny}n^{\kappa-1}.\]
Inserting this into \eqref{eq. middle step to prove the kerenl function}, and we conclude that
\begin{flalign*}
(f* h_1)(z')&=\frac{2(2\pi i)^{\kappa}}{\Gamma(\kappa)}\sum_{n=1}^{\infty}a_f(n)n^{\frac{\kappa-1}{2}}e(nz')n^{\kappa-1}\int_0^{\infty}e^{-4\pi ny}y^{\kappa-1}\frac{\,dy}{y}\\
&=\frac{2(2\pi i)^{\kappa}\Gamma(\kappa-1)}{(4\pi)^{\kappa-1}\Gamma(\kappa)}f(z')=\overline{C_{\kappa}}f(z').
\end{flalign*}
This completes the proof of \eqref{eq. main identity to prove kernel function} and hence Proposition \ref{prop. kernel function}.
\end{proof}

For $\mathbf{s}=(s_1,s_2)\in\bC^2$, we define the function $J(\mathbf{s},n)$. This is the function that we shall study to derive our weighted second moment result. In light of Proposition \ref{prop. kernel function}, we write $J(\mathbf{s},n)$ in the following two different ways:

\begin{equation}
\begin{aligned}\label{relative trace formula function}
J(\mathbf{s},n)&= \int_0^\infty\int_0^\infty\sum_{f\in H_{\kappa}(4)}\frac{(T_nf)(iy_1)\overline{f(iy_2)}}{\|f\|^2}y_1^{s_1+\kappa/2}y_2^{s_2+\kappa/2}\frac{\,dy_1}{y_1}\frac{\,dy_2}{y_2} & =: J_{\mathrm{Spec}}(\mathbf{s},n)\\
&=C_{\kappa}^{-1}n^{\kappa-1}\int_0^{\infty}\int_0^{\infty}\sum_{\gamma\in G_4(n)}\frac{y_1^{s_1+\kappa/2}y_2^{s_2+\kappa/2}}{R_{\gamma}(iy_1,iy_2)}\frac{\,dy_1}{y_1}\frac{\,dy_2}{y_2}
& =: J_{\mathrm{Geo}}(\mathbf{s},n).
\end{aligned}
\end{equation}
For accuracy, we refer to $J_{\mathrm{Spec}}(\textbf{s},n)$ as the \textit{spectral side} of the relative trace formula and $J_{\mathrm{Geo}}(\textbf{s},n)$ as the \textit{geometric side} of the relative trace formula.
\subsection{The Spectral Side}

Let $\Re(s_1),\Re(s_2) \gg 1$. Using the assumption that $f$ is an eigenfunction of the Hecke operator \eqref{eq. hecke action on half-integral weight} and  \eqref{eq. integral representation for L functions}  we get

\begin{equation}\label{eq. spectral side of RTF}
   J_{\mathrm{Spec}}(\mathbf{s},n)=n^{\frac{\kappa-1}{2}}\frac{\Gamma(s_1+\kappa/2)\Gamma(s_2+\kappa/2)}{(2\pi)^{s_1+s_2+\kappa}}\sum_{f\in H_{\kappa}(4)}\frac{\Lambda_f(n)L(1/2+s_1,f)\overline{L(1/2+\overline{s_2},f)}}{\|f\|^2},
\end{equation}
where
\begin{equation}\label{eq. explict hecke eigenvalues}
	\Lambda_f(n)=\prod_{p^{2\ell}\parallel n}(\lambda_F(p^{\ell})-p^{-1}\lambda_F(p^{\ell-2})).
\end{equation}
This is absolutely convergent when $\Re(s_1),\Re(s_2) > \frac{1}{2}$. Moreover, it admits meromorphic continuation to $\mathbb{C}^{2}$ via the analytic continuation of automorphic $L$-functions. In particular, we may take $s_1=s_2=0$ to get the weighted second moment.

\subsection{The Geometric Side}\label{sec, geometric side}
We set
\begin{align*}
B_n&=\left\{\begin{pmatrix}
    a&m\\&d
\end{pmatrix}:\mbox{$ad=n,$ $a,d\in\bZ$, $m\in \bZ$, and $\gcd(a,d,m)=1$}\right\},\\
\overline{B}_n&=B_n/{\pm I_2}=\left\{\begin{pmatrix}
    a&m\\&d
\end{pmatrix}:\mbox{$ad=n,$ $a,d>0,$ $a,d\in\bZ$, $m\in \bZ$, and $\gcd(a,d,m)=1$}\right\},\\
L_n&=\left\{\begin{pmatrix}
    a&\\m&d
\end{pmatrix}:\mbox{$ad=n$, $a,d\in\bZ$, $0\neq m\in 4\bZ$, and $\gcd(a,d,m)=1$}\right\},\\
\overline{L}_n&=L_n/\{\pm I_2\}=\left\{\begin{pmatrix}
    a&\\m&d
\end{pmatrix}:\mbox{$ad=n,$ $a,d>0,$ $a,d\in\bZ$ $0\neq m\in 4\bZ$, and $\gcd(a,d,m)=1$}\right\}.
\end{align*}

With these matrix sets, we prove the following Bruhat decomposition of $G_4(n):$

\begin{lemma}[Bruhat decomposition]\label{lem. bruhat decom}
Let $n$ be an odd square integer. The Bruhat decomposition for $G_4(n)$ is
\begin{align*}
    G_4(n)=B_n\bigsqcup L_n\bigsqcup\left(\bigsqcup_{m\neq0,-n}\bigsqcup_{\substack{ad=n+m\\ a,d\in\bZ}}\bigsqcup_{\substack{bc=m\\ b\in\bZ,c\in4\bZ\\ \gcd(a,b,c,d)=1}}\begin{pmatrix}
        a&b\\c&d
    \end{pmatrix}\right).
\end{align*}
\end{lemma}
\begin{proof}
    
    Recall from Lemma \ref{lem. total set of the Hecke operator} that we have that $G_4(n)=G_4(n)'$. The first component contains matrices of the form $\begin{psmallmatrix}
        *&*\\&*
    \end{psmallmatrix}$. The second component contains matrices of the form $\begin{psmallmatrix}
        *&\\ *&*
    \end{psmallmatrix},$ with bottom left corner being nonzero. The last component contains matrices with each entry being nonzero. Now observe that matrices of the form $\begin{psmallmatrix}
        *&*\\ *&
    \end{psmallmatrix}$ or $\begin{psmallmatrix}
        &*\\ *&*
    \end{psmallmatrix}$ will disappear as $(n,4)=1$. This proves the claim.
\end{proof}

We insert the Bruhat decomposition for $G_{4}(n)$ into \eqref{relative trace formula function} and by the definition of $j_{\gamma}(z)$ for $\gamma=\begin{psmallmatrix}
	a&b\\c&d
\end{psmallmatrix} \in G_{4}(n)$, we obtain the decomposition
\begin{equation}\label{eq. geometric side of RTF}
	J_{\Geo}(\textbf{s},n)=J_{\sm}(\textbf{s},n)+J_{\du}(\textbf{s},n)+J_{\Reg}(\textbf{s},n),
\end{equation}
where
\begin{equation}\label{small orbit of RTF}
   J_{\sm}(\textbf{s},n)= 2C_{\kappa}^{-1}n^{\kappa-1}\int_0^{\infty}\int_0^{\infty}\sum_{\gamma\in \overline{B}_{n}}\frac{y_1^{s_1+\kappa/2}y_2^{s_2+\kappa/2}}{R_{\gamma}(iy_1,iy_2)}\frac{\,dy_1}{y_1}\frac{\,dy_2}{y_2},
\end{equation}
\begin{equation}\label{dual orbit of RTF}
   J_{\du}(\textbf{s},n)= 2C_{\kappa}^{-1}n^{\kappa-1}\int_0^{\infty}\int_0^{\infty}\sum_{\gamma\in \overline{L}_{n}}\frac{y_1^{s_1+\kappa/2}y_2^{s_2+\kappa/2}}{R_{\gamma}(iy_1,iy_2)}\frac{\,dy_1}{y_1}\frac{\,dy_2}{y_2},
\end{equation}
and
\begin{equation}\label{regular orbit of RTF}
   J_{\Reg}(\textbf{s},n)= C_{\kappa}^{-1}n^{\kappa-1}\int_0^{\infty}\int_0^{\infty}\sum_{m\neq0,-n}\sum_{\substack{ad=n+m\\ a,d\in\bZ}}\sum_{\substack{bc=m\\ b\in\bZ,c\in4\bZ\\(a,b,c,d)=1}}\frac{y_1^{s_1+\kappa/2}y_2^{s_2+\kappa/2}}{R_{\gamma}(iy_1,iy_2)}\frac{\,dy_1}{y_1}\frac{\,dy_2}{y_2}.
\end{equation}
\begin{remark}
    Finding the region of absolute convergence for $J_{\Geo}(\textbf{s},n)$ is commonly referred to as the “regularization’’ process, which is a central ingredient of the relative trace formula, see \cite{Yan23a}. In the holomorphic case, however, the regularization is simpler, and $J_{\Geo}(\textbf{s},n)$ admits an analytic continuation to a neighborhood of $\textbf{s}=\textbf{0}.$ For further details, see \cite[Section 5.1]{Wei2025}.  
\end{remark}
We will call $J_{\sm}(\textbf{s},n)$ the \textit{small cell orbital integrals}, $J_{\du}(\textbf{s},n)$ the \textit{dual orbital integrals}, and $J_{\Reg}(\textbf{s},n)$ the \textit{regular orbital integral}. We also define
\[J_{\si}(\textbf{s},n)=J_{\sm}(\textbf{s},n)+J_{\du}(\textbf{s},n),\]
which is the \textit{singular orbital integrals}. Combining this decomposition with \eqref{relative trace formula function}, \eqref{eq. spectral side of RTF}, and \eqref{eq. geometric side of RTF}, and we conclude that  for $\mathbf{s}\in\bC^2$ with $\Re(s_1),\Re(s_2) \gg 1$, we have
\begin{equation}\label{eq. full RTF}
	n^{\frac{\kappa-1}{2}}\frac{\Gamma(s_1+\kappa/2)\Gamma(s_2+\kappa/2)}{(2\pi)^{s_1+s_2+\kappa}}\sum_{f\in H_{\kappa}(4)}\frac{\Lambda_f(n)L(1/2+s_1,f)\overline{L(1/2+\overline{s_2},f)}}{\|f\|^2}=J_{\si}(\mathbf{s},n)+J_{\Reg}(\mathbf{s},n).
\end{equation}
\section{The Singular Orbital Integral $J_{\si}(\textbf{s},n)$}\label{sec. singular}

In this section, we always assume that $n$ is an odd square integer.

\subsection{The Small Cell Orbital Integral $J_{\sm}(\textbf{s},n)$}\label{subsec. small cell}
Before the calculation of the singular orbital integrals, we introduce several useful lemmas. The first lemma is a generalization of the Lipschitz summation formula:

\begin{lemma}\label{lem. possion type lemma}
	Let $z\in\bH$ and $\kappa > 1$. Then for any positive integers $r$ and $s$, we have
	\[\sum_{\substack{m\in\bZ\\ m\equiv r\Mod{s}}}\frac{1}{(z+m)^{\kappa}}=\frac{(-2\pi i)^{\kappa}}{\Gamma(\kappa)s^{\kappa}}\sum_{m=1}^{\infty}m^{\kappa-1}e\left(\frac{rm}{s}\right)e\left(\frac{mz}{s}\right).\]
\end{lemma}
\begin{proof}
    We rewrite the left hand side as
	\[\sum_{\substack{m\in\bZ\\ m\equiv r\Mod{s}}}\frac{1}{(z+m)^{\kappa}}=\sum_{m\in\bZ}\frac{1}{(z+r+ms)^{\kappa}}=\frac{1}{s^{\kappa}}\sum_{m\in\bZ}\frac{1}{\left(\frac{z+r}{s}+m\right)^{\kappa}}.\]
    Applying the Lipschitz summation formula (e.g., \cite[Equation (7)]{RamakrishnanShankhadhar2014}) completes the proof. 
\end{proof}

Recall from \eqref{prop:t_gamma_for_any_matrix} that for any $\gamma=\begin{psmallmatrix}
	a&b\\c&d
\end{psmallmatrix}\in M_4^+,$ there exists a complex number $t_{\gamma}$ satisfying $j_{\gamma}(z)=t_{\gamma}\sqrt{cz+d}$ and such that $|t_{\gamma}| = 1$. The following lemma gives an explicit formula for $t_{\gamma}$ when $\gamma = \gamma_{j}$:
\begin{lemma}\label{lem. t value in the small orbital integral}
	For $\gamma=\gamma^{a,d,m}=\begin{psmallmatrix}
		a&m\\&d
	\end{psmallmatrix}$ satisfying $a,d>0,$ $ad=n$, and $\gcd(a,d,m)=1$, one has
	\[t^{a,d,m}:=t_{\gamma}=\varepsilon_g^{-1}\legendre{-m}{g},\]
	where $g=\gcd(a,d)$ from \S\ref{subsubsec.upper and lower triangualr matrices}.
\end{lemma}

\begin{proof}
	Notice
	\[\begin{pmatrix}
		a&m+\ell d\\&d
	\end{pmatrix}=\begin{pmatrix}
		1&\ell \\&1
	\end{pmatrix}\begin{pmatrix}
		a&m\\&d
	\end{pmatrix}.\]
	So by the cocycle condition, we can reduce the upper right entry modulo $d$ without changing the value of $t^{a,d,m}$. Thus, without loss of generality, it suffices to prove the theorem when $0\leq m<d$. Now we shall apply Lemma \ref{lem. decom lem in linear algebra} where we will let $g_1,r,s,x,y$ define $\gamma_1$ and $\gamma_2$ as in \eqref{eq. explicit smith normal form} and \S \ref{subsubsec.upper and lower triangualr matrices}; in particular, $g_1 = \gcd(m,d)$. From the proof of Lemma \ref{lem. decom lem in linear algebra}, we may further assume $x$ is a prime such that $x\equiv 1\Mod{4}$. By the cocycle condition, we have
	\[t^{a,d,m}=\varepsilon_x^{-1}\varepsilon_s^{-1}\legendre{-4dr}{x}\legendre{4r}{s}=\varepsilon_s^{-1}\legendre{4r}{xs}\legendre{d}{x}.\]

    The determinant of $\gamma_2$ from \eqref{eq. explicit smith normal form} implies $axs-4rg_1=1$. In particular, $as\equiv 1\Mod{4}$ and so $\varepsilon_s=\varepsilon_a$. As $ad=n$ is a square it further follows  that $a\equiv d\Mod{4}$ and thus $\varepsilon_s=\varepsilon_a=\varepsilon_d$.

    Since $ad=n$ is an odd square, we will have that $a\equiv d\Mod{4}$ and so $\legendre{-1}{d}=\legendre{-1}{a}$. Upon writing $4rg_1 = -1+axs$, we see that $axs\equiv 1\Mod{4}$, thus $\legendre{-1}{axs}=1$, and in particular $\legendre{-1}{a}=\legendre{-1}{xs}$. Putting this together gives $\legendre{-1}{d}=\legendre{-1}{a}=\legendre{-1}{xs}=\legendre{-1+axs}{xs}=\legendre{4rg_1}{xs}$. This implies $\legendre{4r}{xs}=\legendre{-1}{d}\legendre{g_1}{xs}$. A similar argument will show $\legendre{g_1}{xs}=\legendre{g_1}{a}$. Then
	\[\legendre{4r}{xs}=\legendre{-1}{d}\legendre{g_1}{xs}=\legendre{-1}{d}\legendre{g_1}{a}.\]
	  As  $x\equiv1\Mod{4},$ we have $\legendre{d}{x}=\legendre{x}{d}.$ Setting $d=d_*(d')^2=p_1\cdots p_s(d')^2$, we can rewrite this identity in the form $\legendre{d}{x}=\legendre{x}{p_1\cdots p_s}$. Recalling that $y$ was chosen so that $\frac{m}{g_1}x+\frac{d}{g_1}y=1$, we have
	\[\legendre{x}{p_1\cdots p_s}\legendre{m/g_1}{p_1\cdots p_s}=\legendre{1-yd/g_1}{p_1\cdots p_s}.\]
	Now as $ad=n$ is an odd square, we must have $p_1\cdots p_s\mid a$. Because $\gcd(a,d,m)=1$, we also have $\ord_{p_1}(g_1)=\cdots\ord_{p_s}(g_1)=0.$ In other words, $p_1\cdots p_s|\frac{d}{g_1}$ which means $\legendre{1-yd/g_1}{p_1\cdots p_s}=1$ and we conclude that
    \[
    \legendre{d}{x}=\legendre{m/g_1}{p_1\cdots p_2}.
    \]
Therefore
	\[t^{a,d,m}=\varepsilon_s^{-1}\legendre{4r}{xs}\legendre{d}{x}=\varepsilon_d^{-1}\legendre{-1}{d}\legendre{g_1}{a}\legendre{m/g_1}{p_1\cdots p_s}=\varepsilon_d^{-1}\legendre{g_1}{a}\legendre{-m/g_1}{p_1\cdots p_s}.\]
	Again recalling that $ad=n$ is a square number, $\ord_p(g)$ is odd if and only if $\ord_p(d)$ is odd if and only if $\ord_p(a)$ is odd. This implies the relations $\varepsilon_d^{-1}=\varepsilon_g^{-1},$ $\legendre{g_1}{a}=\legendre{g_1}{g}$, and $\legendre{-m/g_1}{d}=\legendre{-m/g_1}{g}$. Together, these complete the proof.
	\end{proof}

We now investigate the small cell orbital integral \eqref{small orbit of RTF}. By the definition of $\overline{B}_n,$ we obtain: 
\[ J_{\sm}(\textbf{s},n)= 2C_{\kappa}^{-1}n^{\kappa-1}\int_0^{\infty}\int_0^{\infty}\sum_{\substack{ad=n\\a,d>0}}\sum_{\substack{m\in\bZ\\ \gcd(a,d,m)=1}}\frac{y_1^{s_1+\kappa/2}y_2^{s_2+\kappa/2}(t^{a,d,m})^{-2\kappa}}{(iay_1+idy_2+m)^{\kappa}}\frac{\,dy_1}{y_1}\frac{\,dy_2}{y_2},\]
where $t^{a,d,m}$ is defined as in Lemma \ref{lem. t value in the small orbital integral}. 

\bigskip

The main result of \S \ref{subsec. small cell} is:
\begin{prop}\label{prop. main term 1}
	 Let $n$ be an odd square integer. The small cell orbital integral $J_{\sm}(\textbf{s},n)$ is absolutely convergent in the region
    \[
    \left\{\mathbf{s} = (s_{1},s_{2}) \in \mathbb{C}^{2}:\Re(s_1+s_2)>0, \Re(s_1) >-\frac{\kappa}{2}, \Re(s_2)>-\frac{\kappa}{2}\right\}.
    \]
Moreover, it admits a meromorphic continuation to $\mathbb{C}^2$ given by 
\begin{equation}\label{main term 1}
    \begin{aligned}
	J_{\sm}(\textbf{s},n)&=\frac{2n^{\frac{\kappa}{2}-1}}{i^{\kappa}C_{\kappa}}\frac{\Gamma\left(s_1+\frac{\kappa}{2}\right)\Gamma\left(s_2+\frac{\kappa}{2}\right)}{(2\pi)^{s_1+s_2}\Gamma(\kappa)} \\
    &\phantom{=}\times \sum_{\substack{g^2ad=n\\\gcd(a,d)=1}}L(1+s_1+s_2,\chi_{g_*})g_*^{\frac{1}{2}+s_1+s_2}\frac{\varepsilon_g^{2\kappa+1}\legendre{-1}{g}}{a^{s_1}d^{s_2}g^{s_1+s_2}}\prod_{p|g_0}\left(1-\chi_{g_*}(p)p^{s_1+s_2}\right).
    \end{aligned}
\end{equation}
\end{prop}
\begin{proof}
 Breaking up the sum over $ad=n$ based on $g=\gcd(a,d)$, we can apply Lemma \ref{lem. t value in the small orbital integral} to get
\[J_{\sm}(\textbf{s},n)= 2C_{\kappa}^{-1}n^{\kappa-1}\sum_{\substack{g^2ad=n\\\gcd(a,d)=1}}\varepsilon_g^{2\kappa}\int_0^{\infty}\int_0^{\infty}\sum_{\substack{m\in\bZ\\ \gcd(g,m)=1}}\legendre{-m}{g}\frac{y_1^{s_1+\kappa/2}y_2^{s_2+\kappa/2}}{(iagy_1+idgy_2+m)^{\kappa}}\frac{\,dy_1}{y_1}\frac{\,dy_2}{y_2}.\]
Now we break up the sum over $m\in \mathbb{Z}$ with $\gcd(g,m)=1$ into congruence classes $r$ modulo $g$ with $(r,g)=1$. This gives 
\[J_{\sm}(\textbf{s},n)= 2C_{\kappa}^{-1}n^{\kappa-1}\sum_{\substack{g^2ad=n\\ \gcd(a,d)=1}}\varepsilon_g^{2\kappa}\int_0^{\infty}\int_0^{\infty}\sum_{\substack{r\Mod{g}\\ \gcd(r,g)=1}}\legendre{-r}{g}\sum_{\substack{m\in\bZ \\m\equiv r\Mod{g}}}\frac{y_1^{s_1+\kappa/2}y_2^{s_2+\kappa/2}}{(iagy_1+idgy_2+m)^{\kappa}}\frac{\,dy_1}{y_1}\frac{\,dy_2}{y_2}.\]
Now apply Lemma \ref{lem. possion type lemma} to obtain 
\begin{flalign*}
J_{\sm}(\textbf{s},n)&=\frac{(-2\pi i)^{\kappa}}{\Gamma(\kappa)} \frac{2n^{\kappa-1}}{C_{\kappa}}\sum_{\substack{g^2ad=n\\\gcd(a,d)=1}}\varepsilon_g^{2\kappa}\sum_{\substack{r\Mod{g}\\\gcd(r,g)=1}}\legendre{-r}{g}\sum_{m\geq1}e\left(\frac{mr}{g}\right)\frac{m^{\kappa-1}}{g^{\kappa}}\\
&\phantom{=}\times\int_0^{\infty}e^{-2\pi amy_1}y_1^{s_1+\kappa/2}\frac{\,dy_1}{y_1}\int_0^{\infty}e^{-2\pi dmy_2}y_2^{s_2+\kappa/2}\frac{\,dy_2}{y_2}.	
\end{flalign*} 
Making the change of variable $y_1\mapsto\frac{y_1}{2\pi am}$, $y_2\mapsto\frac{y_2}{2\pi dm},$ 
and recalling the definition of the Gamma function
\[\int_0^{\infty}e^{-y}y^s\frac{\,dy}{y}=\Gamma(s),\]
yields
\[J_{\sm}(\textbf{s},n)=\frac{2n^{\frac{\kappa}{2}-1}}{i^{\kappa}C_{\kappa}}\frac{\Gamma\left(s_1+\frac{\kappa}{2}\right)\Gamma\left(s_2+\frac{\kappa}{2}\right)}{(2\pi)^{s_1+s_2}\Gamma(\kappa)}\sum_{\substack{g^2ad=n\\\gcd(a,d)=1}}\frac{\varepsilon_g^{2\kappa}\legendre{-1}{g}}{a^{s_1}d^{s_2}}\sum_{m\geq1}\frac{c_{\chi_g}(m)}{m^{1+s_1+s_2}},\]
where 
\[c_{\chi}(m)=\sum_{\substack{r\Mod{g}\\\gcd(r,g)=1}}\chi(r)e\left(\frac{rm}{g}\right).\]
Next, we need to investigate the sums $c_{\chi_g}(m),$ which was studied in \cite{BettinBoberBookerConreyLeeMolteniOliverPlattSteiner2018}. Recall the notations in \S \ref{subsubsec.upper and lower triangualr matrices}. By \cite[Lemma 4.11]{BettinBoberBookerConreyLeeMolteniOliverPlattSteiner2018}, $c_{\chi_g}(m)=0$ if $g_2\nmid m,$ and otherwise
\[c_{\chi_g}(mg_2)=g_2\chi_{g_*}(g_0)c_{\chi_{g_*}}(m)c_{g_0}(m),\]  
where $c_{g_0}(m)$ is the Ramanujan sum:
\begin{equation}\label{eq. Ramanujan sum}
c_{g_0}(m)=\sum_{\substack{r\Mod{g_0}\\ \gcd(r,g_0)=1}}e\left(\frac{mr}{g_0}\right)=\sum_{\ell\mid\gcd(m,g_0)}\ell\mu\left(\frac{g_0}{\ell}\right).	
\end{equation}
Notice that $\chi_{g_*}$ is a real primitive character. Then $c_{\chi_{g_*}}(m)=\chi_{g_*}(m)\varepsilon_{g_*}\sqrt{g_{\ast}}.$ This implies
\[\sum_{m\geq1}\frac{c_{\chi_g}(m)}{m^{1+s_1+s_2}}=\frac{1}{g_2^{1+s_1+s_2}}\sum_{m\geq1}\frac{c_{\chi_g}(g_2m)}{m^{1+s_1+s_2}}=\frac{\chi_{g_*}(g_0)\varepsilon_{g_*}\sqrt{g_*}}{g_2^{s_1+s_2}}\sum_{m\geq1}\frac{\chi_{g_*}(m)c_{g_0}(m)}{m^{1+s_1+s_2}}.\]
Combining with \eqref{eq. Ramanujan sum} yields
\begin{align*}
    \sum_{m\geq1}\frac{c_{\chi_g}(m)}{m^{1+s_1+s_2}} &=\frac{\chi_{g_*}(g_0)\varepsilon_{g_*}\sqrt{g_*}}{g_2^{s_1+s_2}}\sum_{m\geq1}\frac{\chi_{g_*}(m)}{m^{1+s_1+s_2}}\sum_{\ell\mid\gcd(m,g_0)}\ell\mu\left(\frac{g_0}{\ell}\right) \\
    &=L(1+s_1+s_2,\chi_{g_*})\frac{\chi_{g_*}(g_0)\mu(g_0)\varepsilon_{g_*}\sqrt{g_*}}{g_2^{s_1+s_2}}\sum_{\ell|g_0}\frac{\mu(\ell)\chi_{g_*}(\ell)}{\ell^{s_1+s_2}},
\end{align*}
where in the second line we have interchanged the order of summation and then made the change of variables $m \mapsto m\ell$. Insert this result into $J_{\sm}(\textbf{s},n)$ to obtain the following:

\begin{align*}
    J_{\sm}(\textbf{s},n)&=\frac{2n^{\frac{\kappa}{2}-1}}{i^{\kappa}C_{\kappa}}\frac{\Gamma\left(s_1+\frac{\kappa}{2}\right)\Gamma\left(s_2+\frac{\kappa}{2}\right)}{(2\pi)^{s_1+s_2}\Gamma(\kappa)} \\
    &\phantom{=}\times\sum_{\substack{g^2ad=n\\\gcd(a,d)=1}}L(1+s_1+s_2,\chi_{g_*})\frac{\varepsilon_g^{2\kappa}\legendre{-1}{g}}{a^{s_1}d^{s_2}}\frac{\chi_{g_*}(g_0)\mu(g_0)\varepsilon_{g_*}\sqrt{g_*}}{g_2^{s_1+s_2}}\sum_{\ell|g_0}\frac{\mu(\ell)\chi_{g_*}(\ell)}{\ell^{s_1+s_2}}.
\end{align*}
Applying the change of variable $\ell\mapsto \frac{g_0}{\ell}$ and noticing that $g_*$ is the squarefree part of $g$, yields
\[\frac{\varepsilon_g^{2\kappa}\legendre{-1}{g}}{a^{s_1}d^{s_2}}\frac{\chi_{g_*}(g_0)\mu(g_0)\varepsilon_{g_*}\sqrt{g_*}}{g_2^{s_1+s_2}}\sum_{\ell|g_0}\frac{\mu(\ell)\chi_{g_*}(\ell)}{\ell^{s_1+s_2}}=g_*^{\frac{1}{2}+s_1+s_2}\frac{\varepsilon_g^{2\kappa+1}\legendre{-1}{g}}{a^{s_1}d^{s_2}g^{s_1+s_2}}\prod_{p|g_0}\left(1-\chi_{g_*}(p)p^{s_1+s_2}\right).\]
Inserting this last expression into $J_{\sm}(\textbf{s},n)$ completes the proof.
\end{proof}

\begin{remark}\label{rem, small orbit n=1}
	When $n=1$, we obtain:
\[
 J_{\sm}(\textbf{s},1)=\frac{2}{i^{\kappa}C_{\kappa}}\frac{\Gamma\left(s_1+\frac{\kappa}{2}\right)\Gamma\left(s_2+\frac{\kappa}{2}\right)}{(2\pi)^{s_1+s_2}\Gamma(\kappa)}\zeta(1+s_1+s_2).
\]
\end{remark}
Notice that $J_{\sm}(\textbf{s},n)$ has a pole at $(s_1,s_2)=(0,0)$ when $\chi_{g_*}$ is the trivial character. This occurs precisely when $g$ is a square. In the following corollary, we separate the terms with and without poles:
\begin{cor}\label{cor. small orbital integral estimation}
 Let $n$ be an odd square integer. Then 
 \[J_{\sm}(\textbf{s},n)=J_{\sm}^1(\textbf{s},n)+J_{\sm}^2(\textbf{s},n),\]
 where $J_{\sm}^1(\textbf{s},n)$ is given by
 \[J_{\sm}^1(\textbf{s},n)=\frac{2n^{\frac{\kappa}{2}-1}}{i^{\kappa}C_{\kappa}}\frac{\Gamma\left(s_1+\frac{\kappa}{2}\right)\Gamma\left(s_2+\frac{\kappa}{2}\right)}{(2\pi)^{s_1+s_2}\Gamma(\kappa)}\zeta(1+s_1+s_2)\sum_{\substack{g^4ad=n\\\gcd(a,d)=1}}\frac{1}{a^{s_1}d^{s_2}g^{2s_1+2s_2}}\prod_{p|g}(1-p^{s_1+s_2}),\]
  and $J_{\sm}^2(\textbf{s},n)$ is entire in $\bC^2$ satisfying
  \[J_{\sm}^2(\mathbf{0},n)\ll_{\varepsilon} \frac{n^{\frac{\kappa}{2}-\frac{3}{4}+\varepsilon}}{|C_{\kappa}|}B\left(\frac{\kappa}{2},\frac{\kappa}{2}\right). \]
\end{cor}
\begin{proof}
	Notice that $L(1+s_1+s_2,\chi_{g_*})$ has a pole at $(s_1,s_2)=(0,0)$ if and only if $\chi_{g_*}=\chi_1$ which occurs if and only if $g$ is a square, and in this case we have $g_*=1$ so that $L(s,\chi_{g_*})=\zeta(s)$. This motivates us to break up $J_{\sm}(\textbf{s},n)$ according to when $g$ is a square or not. Therefore we define
\begin{align*}
	J_{\sm}^1(\textbf{s},n)&=\frac{2n^{\frac{\kappa}{2}-1}}{i^{\kappa}C_{\kappa}}\frac{\Gamma\left(s_1+\frac{\kappa}{2}\right)\Gamma\left(s_2+\frac{\kappa}{2}\right)}{(2\pi)^{s_1+s_2}\Gamma(\kappa)} \\ 
    &\phantom{=}\times \sum_{\substack{g^2ad=n\\\gcd(a,d)=1\\g=\square}}L(1+s_1+s_2,\chi_{{g_*}})g_*^{\frac{1}{2}+s_1+s_2}\frac{\varepsilon_g^{2\kappa+1}\legendre{-1}{g}}{a^{s_1}d^{s_2}g^{s_1+s_2}}\prod_{p|g_0}\left(1-\chi_{g_*}(p)p^{s_1+s_2}\right),
\end{align*}
and
\begin{equation}\label{eq. entire term in the small orbit}
\begin{aligned}
    J_{\sm}^2(\textbf{s},n)&=\frac{2n^{\frac{\kappa}{2}-1}}{i^{\kappa}C_{\kappa}}\frac{\Gamma\left(s_1+\frac{\kappa}{2}\right)\Gamma\left(s_2+\frac{\kappa}{2}\right)}{(2\pi)^{s_1+s_2}\Gamma(\kappa)} \\
    &\phantom{=}\times \sum_{\substack{g^2ad=n\\\gcd(a,d)=1\\g\neq\square}}L(1+s_1+s_2,\chi_{g_*})g_*^{\frac{1}{2}+s_1+s_2}\frac{\varepsilon_g^{2\kappa+1}\legendre{-1}{g}}{a^{s_1}d^{s_2}g^{s_1+s_2}}\prod_{p|g_0}\left(1-\chi_{g_*}(p)p^{s_1+s_2}\right).
\end{aligned}
\end{equation}

Now we derive the formula for $J^1_{\sm}(\textbf{s},n)$. When $g$ is a square, $\varepsilon_g=\legendre{-1}{g}=1$ and $\chi_{g_*}=1.$ Moreover, $g_*=1$, $g_0=\rad(g)$ and $g_2=\frac{g}{\rad(g)}.$ Since in $J_{\sm}^1(\textbf{s},n)$, we have that $g$ must be a square, we substituting $g$ by $g^2$, observe that $g_2$ becomes $\frac{g^2}{\rad(g)}$. Doing this, we obtain
\[J_{\sm}^1(\textbf{s},n)=\frac{2n^{\frac{\kappa}{2}-1}}{i^{\kappa}C_{\kappa}}\frac{\Gamma\left(s_1+\frac{\kappa}{2}\right)\Gamma\left(s_2+\frac{\kappa}{2}\right)}{(2\pi)^{s_1+s_2}\Gamma(\kappa)}\zeta(1+s_1+s_2)\sum_{\substack{g^4ad=n\\\gcd(a,d)=1}}\frac{1}{a^{s_1}d^{s_2}g^{2s_1+2s_2}}\prod_{p|g}(1-p^{s_1+s_2}),\]    
as desired. As $J_{\sm}^1(\textbf{s},n)$ contains all the poles of $J_{\sm}(\textbf{s},n)$, $J_{\sm}^2(\textbf{s},n)$ admits analytic continuation to $\bC^2.$ Now we derive the bound for $J^2_{\sm}(\textbf{0},n)$. Set $(s_1,s_2)=(0,0)$. Observe that $g_*|g$ and $g^2|n,$ together imply $g_*\ll n^{\frac{1}{2}}$. We pick up a power of $n^\varepsilon$ from bounding $\sum_{g^2ad=n}1 \ll_{\varepsilon} n^\varepsilon$. Putting this all together, we conclude that
	 \[J_{\sm}^2(\mathbf{0},n)\ll_{\varepsilon} \frac{n^{\frac{\kappa}{2}-\frac{3}{4}+\varepsilon}}{|C_{\kappa}|}B\left(\frac{\kappa}{2},\frac{\kappa}{2}\right),\]
 which completes the proof.
\end{proof}

\subsection{The Dual Orbital Integral $J_{\du}(\textbf{s},n)$}\label{subsec. dual cell}

Before the calculation of the dual orbital integrals, we introduce the following useful lemmas:
\begin{lemma}\label{lem. special integrals in the dual orbital integral}
	For $\Re(s_1),\Re(s_2)\gg1,$ we set:
	\[J_1(s_1,s_2,\kappa)=\int_0^{\infty}\int_0^{\infty}\frac{y_1^{-s_1+\kappa/2}y_2^{-s_2+\kappa/2}}{(1+iy_1+iy_2)^{\kappa}}\frac{\,dy_1}{y_1}\frac{\,dy_2}{y_2},\]
and
\[J_2(s_1,s_2,\kappa)=\int_0^{\infty}\int_0^{\infty}\frac{y_1^{-s_1+\kappa/2}y_2^{-s_2+\kappa/2}}{(-1+iy_1+iy_2)^{\kappa}}\frac{\,dy_1}{y_1}\frac{\,dy_2}{y_2}.\]
The integrals are convergent when $\Re(s_1)>-\frac{\kappa}{2}+1$ and $\Re(s_2)<\frac{\kappa}{2}.$ Moreover, for $\Re(s_1),\Re(s_2)<\frac{\kappa}{2}$ and $\Re(s_1+s_2)>0,$ we have
\[J_1(s_1,s_2,\kappa)=\frac{e^{\frac{\pi i}{2}(s_1+s_2)}}{i^{\kappa}}\frac{\Gamma(-s_1+\kappa/2)\Gamma(-s_2+\kappa/2)\Gamma(s_1+s_2)}{\Gamma(\kappa)},\]
and
\[J_2(s_1,s_2,\kappa)=\frac{e^{-\frac{\pi i}{2}(s_1+s_2)}}{i^{\kappa}}\frac{\Gamma(-s_1+\kappa/2)\Gamma(-s_2+\kappa/2)\Gamma(s_1+s_2)}{\Gamma(\kappa)}.\]

\end{lemma}

\begin{proof}
	We only prove the claim for $J_1(s_1,s_2,\kappa)$ as the proof for $J_2(s_1,s_2,\kappa)$ is handled similarly. The change of variable $y_2\mapsto \frac{1}{y_2}$ and $y_1\mapsto \frac{y_1}{y_2}$ yields
\[J_1(s_1,s_2,\kappa)=\int_0^{\infty}\int_0^{\infty}\frac{y_2^{s_1+s_2}y_1^{-s_1+\kappa/2}}{(y_2+i(y_1+1))^{\kappa}}\frac{\,dy_2}{y_2}\frac{\,dy_1}{y_1}.\]
Following the method in \cite[Proposition 2.2]{RR05}, the integral is convergent when $\Re(s_1)>-\frac{\kappa}{2}+1$ and $\Re(s_2)<\frac{\kappa}{2}.$

We now examine the inner integral of $y_2.$ Define the function $h(z)=(z+i(y_1+1))^{-\kappa}e^{(s_1+s_2-1)\log z}.$ This is holomorphic when $-\frac{1}{10}<\arg(z)<\frac{\pi}{2}+\frac{1}{10}.$ Let $R>0.$ We consider the sector contour constructed by the line segment connecting $0$ and $R,$ the arc connecting $R$ and $iR$ and the line segment connecting $iR$ and $0.$ By Cauchy's theorem, integrating over this contour will be $0$. Upon taking $R\rightarrow\infty$, the integral over the arc connecting $R$ and $iR$ will go to $0$. Thus, the sum of the integral along the real axis from $0$ to $\infty$ and the integral along the imaginary axis from $i\infty$ to $0$ will be $0$. Using this, we have that:

\[J_1(s_1,s_2,\kappa)=\frac{e^{\frac{\pi i}{2}(s_1+s_2)}}{i^{\kappa}}\int_0^{\infty}y_1^{-s_1+\kappa/2}\int_0^{\infty}\frac{y_2^{s_1+s_2}}{(y_2+y_1+1)^{\kappa}}\frac{\,dy_2}{y_2}\frac{\,dy_1}{y_1}.\]
Then the change of variable $y_2\mapsto y_2(y_1+1)$ implies
\[J_1(s_1,s_2,\kappa)=\frac{e^{\frac{\pi i}{2}(s_1+s_2)}}{i^{\kappa}}\int_0^{\infty}y_1^{-s_1+\kappa/2}(y_1+1)^{s_1+s_2-\kappa}\int_0^{\infty}\frac{y_2^{s_1+s_2}}{(y_2+1)^{\kappa}}\frac{\,dy_2}{y_2}\frac{\,dy_1}{y_1}.\]
Recall that, for $\Re(z_1),\Re(z_2)>0,$ we have
\begin{equation}\label{B function}
    \frac{\Gamma(z_1)\Gamma(z_2)}{\Gamma(z_1+z_2)}=B(z_1,z_2)=\int_0^{\infty}\frac{t^{z_1}}{(t+1)^{z_1+z_2}}\frac{\,dt}{t}.
 \end{equation}
This implies
\[J_1(s_1,s_2,\kappa)=\frac{e^{\frac{\pi i}{2}(s_1+s_2)}}{i^{\kappa}}\frac{\Gamma(-s_1+\kappa/2)\Gamma(-s_2+\kappa/2)\Gamma(s_1+s_2)}{\Gamma(\kappa)},\]
 when $\Re(s_1),\Re(s_2)<\frac{\kappa}{2}$ and $\Re(s_1+s_2)>0$. 
\end{proof}

The next lemma gives the explicit value for $t_{\gamma}$ when $\gamma\in \overline{L}_n.$

\begin{lemma}\label{lem. t value in the dual orbital integral}
	Let $n$ be an odd square integer. Let $\gamma:=\gamma_{a,d,m}=\begin{pmatrix}
		a&\\m&d
	\end{pmatrix}$ be a matrix satisfying $a,d>0,$ $ad=n$, $m\equiv0\Mod{4}$, and $\gcd(a,d,m)=1$. Let $g=\gcd(a,d).$ Then
	\[t_{a,d,m}:=t_{\gamma}=\varepsilon_{g}^{-1}\legendre{m}{g}.\]
\end{lemma}
\begin{proof}
	For $\gamma=\begin{pmatrix}
		a&\\m&d
	\end{pmatrix},$ we recall from  \S \ref{subsubsec.upper and lower triangualr matrices} that $h=\gcd(a,m)$ and $h_1=\gcd\left(h,\frac{ad}{h}\right)$, so we can find $x,y\in\bZ$ such that $\frac{a}{h}x+\frac{m}{h}y=1.$ Furthermore, we can assume that $y$ is a prime and $\gcd(y,mn)=1.$ Thus, we can write
	\[\begin{pmatrix}
		a&\\m&d
	\end{pmatrix}=\begin{pmatrix}
		a/h&-y\\m/h&x
	\end{pmatrix}\begin{pmatrix}
		h&yd\\&ad/h
	\end{pmatrix}.\]
	Notice that as $ad=n$ is odd so is $h$; thus $\begin{pmatrix}
		a/h&-y\\m/h&x
	\end{pmatrix}\in\Gamma_0(4).$ If we apply the cocycle relation, we see that 
    \[
    t_{a,d,m}=\varepsilon_x^{-1}\legendre{m/h}{x} t^{h,\frac{ad}{h},yd}.
    \]
    Now we claim that $\varepsilon_x^{-1}\legendre{m/h}{x} =\varepsilon_{a/h}^{-1}\legendre{m/h}{a/h}$. Indeed, $\frac{a}{h}x\equiv 1\Mod{4}$ because $\frac{a}{h}x=1-\frac{m}{h}y$. Hence $\varepsilon_x=\varepsilon_{a/h}$ and $\legendre{m/h}{a/h}\legendre{m/h}{x}=\legendre{m/h}{ax/h}=\legendre{m/h}{1-my/h}=1$. These two facts together imply $\varepsilon_x^{-1}\legendre{m/h}{x} =\varepsilon_{a/h}^{-1}\legendre{m/h}{a/h}$. If we use this and apply Lemma \ref{lem. t value in the small orbital integral}, we obtain 
    \[t_{a,d,m}=\varepsilon_{a/h}^{-1}\varepsilon_{h_1}^{-1}\legendre{m/h}{a/h}\legendre{-yd}{h_1}.\]
    Now we shall write $h=h_*(h')^2$ and $d=d_*(d')^2$ with the goal of showing that $\varepsilon_{a/h}^{-1}\legendre{m/h}{a/h}=\varepsilon_{h_*d_*}^{-1}\legendre{m/h}{h_*d_*}$ and $\varepsilon_{h_1}^{-1}\legendre{-yd}{h_1}=\varepsilon_{h_*}^{-1}\legendre{-yd}{h_*}$. We will do this by studying prime by prime. In particular, for the first claim it suffices to show that $\ord_p(a/h)$ is odd if and only if $p\mid h_*d_*$, and for the second claim it suffices to show that $\ord_p(h_1)$ is odd if and only if $p\mid h_*$. For the first claim, we shall have two cases depending on if for a prime $p$ if $p\mid h$ or $p\nmid h$. These two cases will correspond to the $p\mid h_*$ and $p\mid d_*$ conditions. 

    Thus, suppose $p$ is a prime such that $\ord_p(h)\geq 1$. Then $\gcd(a,d,m)=1$ implies $\ord_p(d)=0$ and that $\ord_p(a)=\ord_p(n)$. In particular, $\ord_p(a)=\ord_p(n)$ is even because $n$ is a square. Now $\ord_p(a/h)$ is odd if and only if $\ord_p(h)$ is odd which is equivalent to $p\mid h_*$.

    Now suppose that $p$ is a prime such that $p\nmid h$, but $p\mid d$. We wish to show that $\ord_p(a/h)$ is odd if and only if $p\mid d_*$. Then we will have that $\ord_p(a/h)=\ord_p(a)$. Now as $ad=n$ is a square, we have that $\ord_p(a)$ is odd if and only if $\ord_p(d)$ is odd which is equivalent to $p\mid d_*$. These two cases allow us to prove the first claim; as for the second claim, assume that $p\mid h_1$. We have two cases as $h_1=\gcd\left(h,\frac{a}{h}\right)$. If $\ord_p(h_1)=\ord_p(h)$, then it is clear that $\ord_p(h_1)$ is odd if and only if $p\mid h_*$. Now if $\ord_p(h_1)=\ord_p(a/h)$, then as $\ord_p(h)\geq 1$, we may use the first case above to conclude that $\ord_p(h_1)$ is odd if and only if $p\mid h_*$. This finishes the proof of the second claim, putting this together yields:

	\[t_{a,d,m}=\varepsilon_{h_*d_*}^{-1}\varepsilon_{h_*}^{-1}\legendre{m/h}{h_*d_*}\legendre{-yd}{h_*}.\]
	Notice that $\ord_{p}(a)>\ord_p(h)$ if $p|h_*$, which implies:
	\[\legendre{y}{h_*}\legendre{m/h}{h_*}=\legendre{1-xa/h}{h_*}=1\]
	and
	\[t_{a,d,m}=\varepsilon_{h_*d_*}^{-1}\varepsilon_{h_*}^{-1}\legendre{-d}{h_*}\legendre{m/h}{d_*}=\varepsilon_{h_*d_*}^{-1}\varepsilon_{h_*}^{-1}\legendre{-d_*}{h_*}\legendre{m/h}{d_*}.\]
	
Finally, notice that $d_*$ is the square-free part of $d$ and $ad$ is square. This implies that $d_*$ is also the square-free part of $g=\gcd(a,d).$ Then notice that $h_*$  is the square-free part of $h$ and $\gcd(m,g)=1.$ Then we can replace $h_*$ by $h$ and $d_*$ by $g$ in $t_{a,d,m},$ which yields:
\[t_{a,d,m}=\varepsilon_{hd}^{-1}\varepsilon_{h}^{-1}\legendre{-1}{h}\legendre{g}{h}\legendre{m/h}{g}.\]
By the quadratic reciprocity $\legendre{g}{h}\legendre{h}{g}=(-1)^{\frac{h-1}{2}\cdot\frac{g-1}{2}}$ and the fact
\[\varepsilon_{hg}(-1)^{\frac{h-1}{2}\cdot\frac{g-1}{2}}=\varepsilon_h\varepsilon_g,\]
we obtain:
\[t_{a,d,m}=\varepsilon_{g}^{-1}\varepsilon_{h}^{-2}\legendre{-1}{h}\legendre{h}{g}\legendre{m/h}{g}.\]
Then apply $\varepsilon_{h}^{-2}\legendre{-1}{h}=1$ and we complete the proof.
\end{proof}

Assume the notations in \S \ref{subsubsec.upper and lower triangualr matrices}, and the main result in \S \ref{subsec. dual cell} is:
\begin{prop}\label{prop. dual cell}
    Let $n$ be an odd square integer. The dual orbital integral $J_{\du}(\textbf{s},n)$ is absolutely convergent in the region
    \[
    \left\{\mathbf{s} = (s_{1},s_{2}) \in \mathbb{C}^{2}:\Re(s_1+s_2)>1, -\frac{\kappa}{2}+1<\Re(s_1)<\frac{\kappa}{2}, \Re(s_2)<\frac{\kappa}{2}\right\}.
    \]
Moreover, it admits a meromorphic continuation to $\mathbb{C}^2$ given explicitly by 
\begin{equation}\label{main term 3}
\begin{split}
	J_{\du}(\textbf{s},n)
	&=\frac{2(2\pi)^{s_1+s_2}n^{\frac{\kappa}{2}-1}}{i^{\kappa}C_{\kappa}4^{s_1+s_2}}\frac{\Gamma(-s_1+\kappa/2)\Gamma(-s_2+\kappa/2)}{\Gamma(\kappa)}\\
	&\phantom{=}\times \sum_{\substack{adg^2=n\\ \gcd(a,d)=1}}L(1-s_1-s_2,\chi_{g_*})g_*^{\frac{1}{2}-s_1-s_2}d^{s_1}a^{s_2}g^{s_1+s_2}\varepsilon_g^{2\kappa+1}\legendre{-1}{g}\prod_{p|g_0}\left(1-\frac{\chi_{g_*}(p)}{p^{s_1+s_2}}\right).
\end{split}
\end{equation}
\end{prop}

\begin{proof}
Recall \eqref{dual orbit of RTF} which says
\[ J_{\du}(\textbf{s},n)= 2C_{\kappa}^{-1}n^{\kappa-1}\int_0^{\infty}\int_0^{\infty}\sum_{\gamma\in \overline{L}_{n}}\frac{y_1^{s_1+\kappa/2}y_2^{s_2+\kappa/2}}{R_{\gamma}(iy_1,iy_2)}\frac{\,dy_1}{y_1}\frac{\,dy_2}{y_2}.\]
By the definition of $\overline{L}_n$ and $R_{\gamma}(iy_1,iy_2),$ this is
\[J_{\du}(\textbf{s},n)= 2C_{\kappa}^{-1}n^{\kappa-1}\int_0^{\infty}\int_0^{\infty}\sum_{\substack{ad=n\\ a,d>0}}\sum_{\substack{m\in\bZ\\ m\neq0,4|m\\ \gcd(a,d,m)=1}}\frac{y_1^{s_1+\kappa/2}y_2^{s_2+\kappa/2}t_{a,d,m}^{-2\kappa}}{(-my_1y_2+i(dy_2+ay_1))^{\kappa}}\frac{\,dy_1}{y_1}\frac{\,dy_2}{y_2}.\]
We separate the terms by $m$ positive or negative. For the negative terms, we perform the change of variable $m\mapsto -m$ giving
\begin{flalign*}
	J_{\du}(\textbf{s},n)&=2C_{\kappa}^{-1}n^{\kappa-1}\int_0^{\infty}\int_0^{\infty}\sum_{\substack{ad=n\\ a,d>0}}\sum_{\substack{m\geq1,4|m\\ \gcd(a,d,m)=1}}\frac{y_1^{s_1+\kappa/2}y_2^{s_2+\kappa/2}t_{a,d,-m}^{-2\kappa}}{m^{\kappa}\left(y_1y_2+i\left(\frac{dy_2}{m}+\frac{ay_1}{m}\right)\right)^{\kappa}}\frac{\,dy_1}{y_1}\frac{\,dy_2}{y_2} \\
    &\phantom{=}+2C_{\kappa}^{-1}n^{\kappa-1}\int_0^{\infty}\int_0^{\infty}\sum_{\substack{ad=n\\ a,d>0}}\sum_{\substack{m\geq1,4|m\\ \gcd(a,d,m)=1}}\frac{y_1^{s_1+\kappa/2}y_2^{s_2+\kappa/2}t_{a,d,m}^{-2\kappa}}{m^{\kappa}\left(-y_1y_2+i\left(\frac{dy_2}{m}+\frac{ay_1}{m}\right)\right)^{\kappa}}\frac{\,dy_1}{y_1}\frac{\,dy_2}{y_2}.
\end{flalign*}

We then factor out $y_1y_2$ and apply the change of variable $y_1\mapsto\frac{d}{my_1}$ and $y_2\mapsto \frac{a}{my_2}$ to get
\begin{flalign*}
	J_{\du}(\textbf{s},n)&=2C_{\kappa}^{-1}n^{\frac{\kappa}{2}-1}\int_0^{\infty}\int_0^{\infty}\frac{y_1^{-s_1+\kappa/2}y_2^{-s_2+\kappa/2}}{(1+i(y_1+y_2))^{\kappa}}\frac{\,dy_1}{y_1}\frac{\,dy_2}{y_2}\sum_{\substack{ad=n\\ a,d>0}}\sum_{\substack{m\geq1,4|m\\ \gcd(a,d,m)=1}}\frac{t_{a,d,-m}^{-2\kappa}d^{s_1}a^{s_2}}{m^{s_1+s_2}} \\
    &\phantom{=}+2C_{\kappa}^{-1}n^{\frac{\kappa}{2}-1}\int_0^{\infty}\int_0^{\infty}\frac{y_1^{-s_1+\kappa/2}y_2^{-s_2+\kappa/2}}{(-1+i(y_1+y_2))^{\kappa}}\frac{\,dy_1}{y_1}\frac{\,dy_2}{y_2}\sum_{\substack{ad=n\\ a,d>0}}\sum_{\substack{m\geq1,4|m\\ \gcd(a,d,m)=1}}\frac{t_{a,d,m}^{-2\kappa}d^{s_1}a^{s_2}}{m^{s_1+s_2}}.
\end{flalign*}
The remaining two integrals are $J_1(s_1,s_2,\kappa)$ and $J_2(s_1,s_2,\kappa)$ respectively as given in Lemma \ref{lem. special integrals in the dual orbital integral}. It follows that
\begin{flalign*}
	J_{\du}(\textbf{s},n)&=\frac{2n^{\frac{\kappa}{2}-1}}{i^{\kappa}C_{\kappa}}\frac{\Gamma(-s_1+\kappa/2)\Gamma(-s_2+\kappa/2)\Gamma(s_1+s_2)}{\Gamma(\kappa)}\\
	&\phantom{=}\times\left(e^{\frac{\pi i}{2}(s_1+s_2)}\sum_{\substack{ad=n\\ a,d>0}}\sum_{\substack{m\geq1,4|m\\ \gcd(a,d,m)=1}}\frac{t_{a,d,-m}^{-2\kappa}d^{s_1}a^{s_2}}{m^{s_1+s_2}}+e^{-\frac{\pi i}{2}(s_1+s_2)}\sum_{\substack{ad=n\\ a,d>0}}\sum_{\substack{m\geq1,4|m\\ \gcd(a,d,m)=1}}\frac{t_{a,d,m}^{-2\kappa}d^{s_1}a^{s_2}}{m^{s_1+s_2}}\right).
\end{flalign*}
Next, we investigate the series:
\[I^{\pm}(s_1,s_2;n):=\sum_{\substack{ad=n\\ a,d>0}}\sum_{\substack{m\geq1,4|m\\ \gcd(a,d,m)=1}}\frac{t_{a,d,\mp m}^{-2\kappa}d^{s_1}a^{s_2}}{m^{s_1+s_2}}.\]
It suffices to only study $I^{+}(s_1,s_2;n)$ as the argument for $I^{-}(s_1,s_2;n)$ is similar. Set $g=\gcd(a,d)$ and substitute $a$ (resp. $d$) by $ag$ (resp. $dg$) to obtain
\[I^+(s_1,s_2;n)=\sum_{\substack{adg^2=n\\ \gcd(a,d)=1}}\sum_{\substack{m\geq1\\ \gcd(m,g)=1}}\frac{d^{s_1}a^{s_2}g^{s_1+s_2}\varepsilon_g^{2\kappa}\legendre{-m}{g}}{4^{s_1+s_2}m^{s_1+s_2}}.\]
Next, we set $g=g_*(g')^2=p_1\cdots p_s(g')^2.$ Then $\chi_g$ is induced from the character $\chi_{g_*}$ with conductor $g_{*}.$ It follows that
\[\sum_{\substack{m\geq1\\ \gcd(m,g)=1}}\frac{\legendre{m}{g}}{m^{s_1+s_2}}=L(s_1+s_2,\chi_{g_*})\prod_{p|g_0}\left(1-\frac{\chi_{g_*}(p)}{p^{s_1+s_2}}\right),\]
and hence
\[I^+(s_1,s_2;n)=\frac{1}{4^{s_1+s_2}}\sum_{\substack{adg^2=n\\ \gcd(a,d)=1}}L(s_1+s_2,\chi_{g_*})d^{s_1}a^{s_2}g^{s_1+s_2}\varepsilon_g^{2\kappa}\legendre{-1}{g}\prod_{p|g_0}\left(1-\frac{\chi_{g_*}(p)}{p^{s_1+s_2}}\right).\]
A similar argument will show
\[I^-(s_1,s_2;n)=\frac{1}{4^{s_1+s_2}}\sum_{\substack{adg^2=n\\ \gcd(a,d)=1}}L(s_1+s_2,\chi_{g_*})d^{s_1}a^{s_2}g^{s_1+s_2}\varepsilon_g^{2\kappa}\prod_{p|g_0}\left(1-\frac{\chi_{g_*}(p)}{p^{s_1+s_2}}\right).\]
Inserting $I^{\pm}(s_1,s_2;n)$ into $J_{\du}(\textbf{s},n)$ yields
\begin{flalign*}
	J_{\du}(\textbf{s},n)&=\frac{2n^{\frac{\kappa}{2}-1}}{i^{\kappa}C_{\kappa}}\frac{\Gamma(-s_1+\kappa/2)\Gamma(-s_2+\kappa/2)}{4^{s_1+s_2}\Gamma(\kappa)}\sum_{\substack{adg^2=n\\ \gcd(a,d)=1}}d^{s_1}a^{s_2}g^{s_1+s_2}\varepsilon_g^{2\kappa}\legendre{-1}{g}\prod_{p|g_0}\left(1-\frac{\chi_{g_*}(p)}{p^{s_1+s_2}}\right)\\
	&\hspace{40mm}\times\Gamma(s_1+s_2)\left(e^{\frac{\pi i}{2}(s_1+s_2)}+\legendre{-1}{g}e^{-\frac{\pi i}{2}(s_1+s_2)}\right)L(s_1+s_2,\chi_{g_*})\\
	&=\frac{2n^{\frac{\kappa}{2}-1}}{i^{\kappa}C_{\kappa}}\frac{\Gamma(-s_1+\kappa/2)\Gamma(-s_2+\kappa/2)}{4^{s_1+s_2}\Gamma(\kappa)}\sum_{\substack{adg^2=n\\ \gcd(a,d)=1}}d^{s_1}a^{s_2}g^{s_1+s_2}\varepsilon_g^{2\kappa+1}\legendre{-1}{g}\prod_{p|g_0}\left(1-\frac{\chi_{g_*}(p)}{p^{s_1+s_2}}\right)\\
	&\hspace{40mm}\times\Gamma(s_1+s_2)2\cos\left(\frac{\pi}{2}(s_1+s_2-\delta)\right)L(s_1+s_2,\chi_{g_*}),
\end{flalign*}
where the second equality is due to the fact $e^{\frac{\pi i}{2}(s_1+s_2)}+\legendre{-1}{g}e^{-\frac{\pi i}{2}(s_1+s_2)}=2\varepsilon_g\cos\left(\frac{\pi}{2}(s_1+s_2-\delta)\right),$ with $\delta=\frac{1-\chi_{g_*}(-1)}{2}.$ As $\chi_{g_*}$ is a real primitive character, the functional equation associated to its Dirichlet $L$-function is
\[L(1-s,\chi_{g_*})=\pi^{-s}2^{1-s}g_*^{s-\frac{1}{2}}\Gamma(s)\cos\left(\frac{\pi (s-\delta)}{2}\right)L(s,\chi_{g_*}).\]
Inserting this into $J_{\du}(\textbf{s},n)$ allows us to conclude
\begin{flalign*}
	J_{\du}(\textbf{s},n)
	&=\frac{2(2\pi)^{s_1+s_2}n^{\frac{\kappa}{2}-1}}{i^{\kappa}C_{\kappa}4^{s_1+s_2}}\frac{\Gamma(-s_1+\kappa/2)\Gamma(-s_2+\kappa/2)}{\Gamma(\kappa)}\\
	&\phantom{=}\times \sum_{\substack{adg^2=n\\ \gcd(a,d)=1}}L(1-s_1-s_2,\chi_{g_*})g_*^{\frac{1}{2}-s_1-s_2}d^{s_1}a^{s_2}g^{s_1+s_2}\varepsilon_g^{2\kappa+1}\legendre{-1}{g}\prod_{p|g_0}\left(1-\frac{\chi_{g_*}(p)}{p^{s_1+s_2}}\right).
\end{flalign*}

\end{proof}
\begin{remark}\label{rem, dual orbit n=1}
	When $n=1$, we obtain:
\[
 J_{\du}(\textbf{s},1)=\frac{2(2\pi)^{s_1+s_2}}{i^{\kappa}C_{\kappa}4^{s_1+s_2}}\frac{\Gamma(-s_1+\kappa/2)\Gamma(-s_2+\kappa/2)}{\Gamma(\kappa)}\zeta(1-s_1-s_2)
\]
\end{remark}

Similar to Corollary \ref{cor. small orbital integral estimation}, we separate the sum into those terms which have a pole at $(0,0)$ and those which are holomorphic. In particular we prove the following:
\begin{cor}\label{cor. dual orbital integral estimation}
 Let $n$ be an odd square integer. Then 
 \[J_{\du}(\textbf{s},n)=J_{\du}^1(\textbf{s},n)+J_{\du}^2(\textbf{s},n),\]
 where $J_{\du}^1(\textbf{s},n)$ is given by
 \[J_{\du}^1(\textbf{s},n)=\frac{2n^{\frac{\kappa}{2}-1}}{i^{\kappa}C_{\kappa}}\frac{\pi^{s_1+s_2}\Gamma\left(-s_1+\frac{\kappa}{2}\right)\Gamma\left(-s_2+\frac{\kappa}{2}\right)}{2^{s_1+s_2}\Gamma(\kappa)}\zeta(1-s_1-s_2)\sum_{\substack{g^4ad=n\\\gcd(a,d)=1}}{a^{s_1}d^{s_2}g^{2s_1+2s_2}}\prod_{p|g}\left(1-\frac{1}{p^{s_1+s_2}}\right),\]
  and $J_{\du}^2(\textbf{s},n)$ is entire in $\bC^2$ satisfying
  \[J_{\du}^2(\mathbf{0},n)\ll_{\varepsilon} \frac{n^{\frac{\kappa}{2}-\frac{3}{4}+\varepsilon}}{|C_{\kappa}|}B\left(\frac{\kappa}{2},\frac{\kappa}{2}\right). \]
\end{cor}
\begin{proof}
The proof is identical to that of Corollary \ref{cor. small orbital integral estimation} and we only sketch the proof. Let $g$ be an integer and denote by $g_*$ the squarefree part of $g.$ We set
\begin{flalign*}
	J_{\du}^1(\textbf{s},n)
	&=\frac{2(2\pi)^{s_1+s_2}n^{\frac{\kappa}{2}-1}}{i^{\kappa}C_{\kappa}4^{s_1+s_2}}\frac{\Gamma(-s_1+\kappa/2)\Gamma(-s_2+\kappa/2)}{\Gamma(\kappa)}\\
	&\phantom{=}\times \sum_{\substack{adg^2=n\\ \gcd(a,d)=1\\g=\square}}L(1-s_1-s_2,\chi_{g_*})g_*^{\frac{1}{2}-s_1-s_2}d^{s_1}a^{s_2}g^{s_1+s_2}\varepsilon_g^{2\kappa+1}\legendre{-1}{g}\prod_{p|g_0}\left(1-\frac{\chi_{g_*}(p)}{p^{s_1+s_2}}\right).
\end{flalign*}
and
\begin{equation}\label{eq. entire term in the dual orbit}
\begin{split}
J_{\du}^2(\textbf{s},n)
	&=\frac{2(2\pi)^{s_1+s_2}n^{\frac{\kappa}{2}-1}}{i^{\kappa}C_{\kappa}4^{s_1+s_2}}\frac{\Gamma(-s_1+\kappa/2)\Gamma(-s_2+\kappa/2)}{\Gamma(\kappa)}\\
	&\phantom{=}\times \sum_{\substack{adg^2=n\\ \gcd(a,d)=1\\g\neq\square}}L(1-s_1-s_2,\chi_{g_*})g_*^{\frac{1}{2}-s_1-s_2}d^{s_1}a^{s_2}g^{s_1+s_2}\varepsilon_g^{2\kappa+1}\legendre{-1}{g}\prod_{p|g_0}\left(1-\frac{\chi_{g_*}(p)}{p^{s_1+s_2}}\right).
\end{split}	
\end{equation}
\end{proof}

\subsection{The Singular Orbital Integrals: A Summary}\label{subsec. summary of singular}
Recall the definition of $J_{\si}(\textbf{s},n)$, Corollary  \ref{cor. small orbital integral estimation}, and Corollary \ref{cor. dual orbital integral estimation}. We write
\begin{flalign*}
J_{\si}(\textbf{s},n)=&J_{\sm}^2(\textbf{s},n)+J_{\du}^2(\textbf{s},n)+\frac{2n^{\frac{\kappa}{2}-1}}{i^{\kappa}C_{\kappa}2^{s_1+s_2}\Gamma(\kappa)}\left(K(s_1,s_2;n)+K(-s_1,-s_2;n)\right),
\end{flalign*}
where
\[K(s_1,s_2;n)=\frac{\Gamma(s_1+\kappa/2)\Gamma(s_2+\kappa/2)\zeta(1+s_1+s_2)}{\pi^{s_1+s_2}}\sum_{\substack{g^4ad=n\\\gcd(a,d)=1}}\frac{1}{a^{s_1}d^{s_2}g^{2s_1+2s_2}}\prod_{p|g}(1-p^{s_1+s_2}),\]
and $J_{\sm}^2(\textbf{s},n)$ (resp. $J_{\du}^2(\textbf{s},n)$) is defined in \eqref{eq. entire term in the small orbit} (resp. \eqref{eq. entire term in the dual orbit}). Now $J_{\si}(\textbf{s},n)$ admits meromorphic continuation to $\bC^2$ by Corollary \ref{cor. small orbital integral estimation} and Corollary \ref{cor. dual orbital integral estimation}. Moreover, $J_{\si}(\textbf{s},n)$ has a singularity at $(s_1,s_2)=(0,0).$ In the following, we will show that this singularity is removable. To this end, we note that $K(s_1,s_2;n)$ satisfies the conditions in \cite[Lemma 2.5.2]{ConreyFarmerKeatingRubinsteinSnaith2005}, which implies the following: 

\begin{prop}\label{prop. main term for the general case}
Let $n$ be an odd square integer. Then for $\textbf{s}\in \bC^2,$
\begin{align}
    \begin{split}\label{eq. singular orbital integral general form}
	J_{\si}(\textbf{s},n)&=J_{\sm}^2(\textbf{s},n)+J_{\du}^2(\textbf{s},n)\\
    &\phantom{=}-\frac{2n^{\frac{\kappa}{2}-1}}{i^{\kappa}C_{\kappa}2^{s_1+s_2}\Gamma(\kappa)}\frac{1}{(2\pi i)^2}\oint\displaylimits_{|z_1|=1}\oint\displaylimits_{|z_2|=1}\frac{(z_1^2-z_2^2)^2(z_1z_2+s_1s_2)K(z_1,z_2;n)}{\prod_{i=1}^2\prod_{j=1}^2(z_i-s_j)(z_i+s_j)}\,dz_1\,dz_2.
    \end{split}
\end{align}
where $J_{\sm}^2(\textbf{s},n)$ (resp. $J_{\du}^2(\textbf{s},n)$) is defined in \eqref{eq. entire term in the small orbit} (resp. \eqref{eq. entire term in the dual orbit}).
\end{prop}

From Proposition \ref{prop. main term for the general case} we derive the following corollary:
\begin{cor}\label{cor. estimation of the singular orbital integral}
	Let $n$ be an odd square integer. Then for any $\varepsilon>0,$
	\[J_{\si}(\textbf{0},n)\ll_{\varepsilon} n^{\frac{\kappa}{2}-\frac{3}{4}+\varepsilon}\frac{\kappa^{\varepsilon}B(\kappa/2,\kappa/2)}{|C_{\kappa}|}.\]	
\end{cor}
\begin{proof}
Apply Proposition \ref{prop. main term for the general case}, with $\mathbf{s}=\mathbf{0}$ to see that 
\[J_{\si}(\textbf{0},n)=J_{\sm}^2(\textbf{0},n)+J_{\du}^2(\textbf{0},n)-\frac{2n^{\frac{\kappa}{2}-1}}{i^{\kappa}C_{\kappa}\Gamma(\kappa)}\frac{1}{(2\pi i)^2}\oint\displaylimits_{|z_1|=1}\oint\displaylimits_{|z_2|=1}\frac{(z_1^2-z_2^2)^2}{z_1^3z_2^3}K(z_1,z_2;n)\,dz_1\,dz_2.	\]

	By Corollary \ref{cor. small orbital integral estimation} and Corollary \ref{cor. dual orbital integral estimation}, we obtain
	\[J_{\sm}^2(\textbf{0},n)+J_{\du}^2(\textbf{0},n)\ll_{\varepsilon}  n^{\frac{\kappa}{2}-\frac{3}{4}+\varepsilon}\frac{B(\kappa/2,\kappa/2)}{|C_{\kappa}|}.\]
	 For the last term, apply the residue theorem and the estimate
	\[\frac{\Gamma'\left(\frac{\kappa}{2}\right)}{\Gamma\left(\frac{\kappa}{2}\right)}\ll \log \kappa\ll_{\varepsilon} \kappa^{\varepsilon}.\]
	This implies that the last term is $O_{\varepsilon}\left(n^{\frac{\kappa}{2}-1+\varepsilon}\frac{\kappa^{\varepsilon}B(\kappa/2,\kappa/2)}{|C_{\kappa}|}\right)$ which is an admissible error in comparison to the bound obtained for the first two terms.
\end{proof}

\begin{remark}\label{rem. singular orbital integral when n=1}
	When $n=1,$ we have that $J_{\sm}^2(\textbf{0},1)$ and $J_{\du}^2(\textbf{0},1)$ vanish since in these terms we must have that $n=a=d=g=1$, and the terms defining these sums are over $g$ not a square. Now a similar argument to Remark \ref{rem, small orbit n=1} and Remark \ref{rem, dual orbit n=1} will show that 
	\begin{equation}\label{eq, closed formula for singular orbital integral}
J_{\si}(\textbf{0},1)=-\frac{2}{i^{\kappa}C_{\kappa}\Gamma(\kappa)}\frac{1}{(2\pi i)^2}\oint\displaylimits_{|z_1|=1}\oint\displaylimits_{|z_2|=1}\frac{(z_1^2-z_2^2)^2}{(z_1z_2)^3}\frac{\Gamma(z_1+\kappa/2)\Gamma(z_2+\kappa/2)}{\pi^{z_1+z_2}}\zeta(1+z_1+z_2)\,dz_1\,dz_2.
\end{equation}
Applying the Laurent expansion for $\zeta(s)$ (at $s=1$) and the residue theorem, we can show that
\[J_{\si}(\textbf{0},1)=\frac{4\Gamma\left(\frac{\kappa}{2}\right)^2}{i^{\kappa}C_{\kappa}\Gamma(\kappa)}\left(\frac{\Gamma'\left(\frac{\kappa}{2}\right)}{\Gamma\left(\frac{\kappa}{2}\right)}+O(1)\right)=\frac{4\Gamma\left(\frac{\kappa}{2}\right)^2}{i^{\kappa}C_{\kappa}\Gamma(\kappa)}\left(\log\kappa+O\left(1+\frac{1}{\kappa}\right)\right).\]
\end{remark}

\section{The Regular Orbital Integrals $J_{\Reg}(\textbf{s},n)$}\label{sec. regular}
In this section, let $n$ be an odd square integer. We study $J_{\mathrm{Reg}}(\mathbf{s},n)$. For any integer $m\neq 0,-n$, let $\mathcal{M}_m$ denote the set of matrices
\[
    \mathcal{M}_m = \left\{ \begin{pmatrix} a&b\\c&d \end{pmatrix} \in M_2(\mathbb{Z}) \;:\; ad=n+m, \, bc=m, 4\mid c, \, \gcd(a,b,c,d)=1 \right\}.
\]
We define $\Omega = \bigsqcup_{m \neq 0, -n} \mathcal{M}_m$; this is precisely the last union of matrices in Lemma \ref{lem. bruhat decom}. We further decompose $\Omega$ into the disjoint union $\Omega_1 \sqcup \Omega_2 \sqcup \Omega_3$, where
\begin{align*}
    \Omega_1 &= \bigsqcup_{m\geq 1} \mathcal{M}_m, \\
    \Omega_2 &= \bigsqcup_{m\leq -(n+1)} \mathcal{M}_m, \\
    \Omega_3 &= \bigsqcup_{-n+1 \leq m \leq -1} \mathcal{M}_m.
\end{align*}
Then set $\overline{\Omega}_j=\Omega_j/\{\pm I\}$. We define
\[J_{\Reg}^j(\textbf{s},n)=\frac{n^{\kappa-1}}{C_{\kappa}}\int_0^{\infty}\int_0^{\infty}\sum_{\gamma\in\Omega_j}\frac{y_1^{s_1+\kappa/2}y_2^{s_2+\kappa/2}}{R_{\gamma}(iy_1,iy_2)}\frac{\,dy_1}{y_1}\frac{\,dy_2}{y_2}=\frac{2n^{\kappa-1}}{C_{\kappa}}\int_0^{\infty}\int_0^{\infty}\sum_{\gamma\in\overline{\Omega}_j}\frac{y_1^{s_1+\kappa/2}y_2^{s_2+\kappa/2}}{R_{\gamma}(iy_1,iy_2)}\frac{\,dy_1}{y_1}\frac{\,dy_2}{y_2}.\]
The disjoint union of matrices implies
\begin{equation}\label{eq. decom of regular orbital integral}
	J_{\Reg}(\textbf{s},n)=J_{\Reg}^1(\textbf{s},n)+J_{\Reg}^2(\textbf{s},n)+J_{\Reg}^3(\textbf{s},n).
\end{equation}
The main result of this section is the following:
\begin{prop}\label{prop, the error term in RTF}
	We adopt the notation established above. The regular orbital integral $J_{\Reg}(\textbf{s},n)$ converges absolutely in the region
    \[
    \left\{\mathbf{s} = (s_{1},s_{2}) \in \mathbb{C}^{2}:|\Re(s_1)|<\frac{\kappa}{2}-1,
|\Re(s_2)|<\frac{\kappa}{2}-1. \right\}.
    \]
In particular, if we set $\mathbf{s}=\mathbf{0}$ then
    $$J_{\Reg}(\textbf{0}, n)=J_{\Reg}^1(\textbf{0},n)+J_{\Reg}^2(\textbf{0},n)+J_{\Reg}^3(\textbf{0},n),$$ with
\begin{align*}
J_{\Reg}^1(\textbf{0},n)&=\frac{2n^{\kappa-1}B(\kappa/2,\kappa/2)^2}{i^{2\kappa}C_{\kappa}}\sum_{m\geq1}\frac{F(\kappa/2,\kappa/2;\kappa;n/(m+n))}{(m+n)^{\kappa/2}}\sum_{\substack{ad=n+m\\ a,d\geq1}}\sum_{\substack{bc=m\\b,c\geq1,4\mid c\\\gcd(a,b,c,d)=1}}(t_{\gamma}^{-2\kappa}i^{\kappa}+t_{\widehat{\gamma}}^{-2\kappa}i^{-\kappa}),\\
J_{\Reg}^2(\textbf{0},n)&=\frac{2n^{\kappa-1}B(\kappa/2,\kappa/2)^2}{i^{2\kappa}C_{\kappa}}\sum_{m\geq1}\frac{F(\kappa/2,\kappa/2;\kappa;n/(m+n))}{(m+n)^{\kappa/2}}\sum_{\substack{ad=m\\ a,d\geq1}}\sum_{\substack{bc=m+n\\b,c\geq1,4|c\\\gcd(a,b,c,d)=1}}(t_{\overline{\gamma}}^{-2\kappa}+t_{\underline{\gamma}}^{-2\kappa}),\\
J_{\Reg}^3(\textbf{0},n)&=\frac{2n^{\kappa-1}B(\kappa/2,\kappa/2)^2}{i^{2\kappa}C_{\kappa}}\sum_{1\leq m\leq n-1}\frac{F(\kappa/2,\kappa/2;\kappa;n/m)}{m^{\kappa/2}}\sum_{\substack{ad=n-m\\a,d\geq1}}\sum_{\substack{bc=m\\b,c\geq1,4|c\\\gcd(a,b,c,d)=1}}(t_{{}^{\iota}{\gamma}}^{-2\kappa}+t_{{}{\gamma}^{\iota}}^{-2\kappa}).
\end{align*}
\end{prop}

\begin{proof}
	The region of absolute convergence follows from \cite[Theorem 2.1]{RR05}. The second part of the proposition is a direct corollary of Lemma \ref{lem, 1st error}, Lemma \ref{lem, 2nd error} and Lemma \ref{lem, 3rd error}.
\end{proof}

We analyze $J_{\Reg}^j(\textbf{s},n)$ separately in the subsequent sections. As a preliminary step, we make a reduction using the following lemma:
\begin{lemma}\label{lemma, prep for roi}
	Let $\gamma\in G_4(n).$ Then for $y_1,y_2>0$ and $\kappa\in\bR,$
	\[\gamma(iy_1)^{\kappa}\left(\frac{iy_2}{\gamma(iy_1)}+1\right)^{\kappa}=(iy_2+\gamma(iy_1))^{\kappa}.\]
\end{lemma}
\begin{proof}
	We first check the arguments of each complex number. We can show that $\arg(\gamma(iy_1))\in(0,\pi)$ and $\arg(iy_2+\gamma(iy_1))\in(0,\pi).$ This will force 
	\[\arg(\gamma(iy_1))+\arg\left(\frac{iy_2}{\gamma(iy_1)}+1\right)\in(-\pi,\pi).\]
	Applying Lemma \ref{lemma. complex power st} completes the proof.
\end{proof}
To simplify notation in the subsequent sections, we introduce some notation. For $\mathbf{s}=(s_1,s_2)\in\bC^2,$  we will set
$$\mathcal{B}(s_2):=\mathcal{B}(s_2;n,\kappa)=\frac{2n^{\kappa-1}B(s_2+\kappa/2,-s_2+\kappa/2)}{i^{s_2+\kappa/2}C_{\kappa}}.$$
By Lemma \ref{lemma, prep for roi} and the definition of $R_\gamma$ in \eqref{eq. R function}, we obtain
\[J_{\Reg}^j(\textbf{s},n)=\frac{2n^{\kappa-1}}{C_{\kappa}}\sum_{\gamma\in\overline{\Omega}_j}t_{\gamma}^{-2\kappa}\int_0^{\infty}\frac{y_1^{s_1+\kappa/2}}{j(\gamma,iy_1)^{\kappa}\gamma(iy_1)^{\kappa}}\int_0^{\infty}\frac{y_2^{s_2+\kappa/2}}{\left(\frac{i}{\gamma(iy_1)}y_2+1\right)^{\kappa}}\frac{\,dy_2}{y_2}\frac{\,dy_1}{y_1}.\]
Here $\gamma=\begin{psmallmatrix}
	a&b\\c&d
\end{psmallmatrix}\in G_4(n)$ and $t_{\gamma}=\frac{j_{\gamma}(z)}{(cz+d)^{\frac{1}{2}}}.$ We apply \cite[Equation 3.194(3)]{GradshteynRyzhik2007} with $\mu=s_2+\kappa/2$ to rewrite the integral in the form
\begin{equation}\label{eq. first reduction for roi}
J_{\Reg}^j(\textbf{s},n)=\mathcal{B}(s_2)\sum_{\gamma\in\overline{\Omega}_j}t_{\gamma}^{-2\kappa}\int_0^{\infty}\frac{y_1^{s_1+\kappa/2}}{j(\gamma,iy_1)^{\kappa}\gamma(iy_1)^{-s_2+\kappa/2}}\frac{\,dy_1}{y_1}.
\end{equation}

\begin{remark}\label{rem: matrix-discussion-sec4}We will now prove the decompositions for $J_{\Reg}^j(\textbf{s},n)$. While the details are different in each case, the strategy remains the same. We consider the sum over the matrices in $\Omega_j$; these are matrices of the form $\begin{psmallmatrix}
    a & b\\ c & d
\end{psmallmatrix}\in M_2(\bZ)$. Our goal is to rewrite the corresponding sums for $a,b,c,d>0$. Furthermore, for each $\Omega_j$, we determine the signs of the matrix entries. This is the motivation of the notation in \S \ref{subsubsec: transformation-matrices} as these will be the corresponding locations of the negative signs in the matrices when we reindex our sums to be over positive integers. 
\end{remark}

\subsection{The Calculation of $J_{\Reg}^1(\textbf{s},n)$}\label{subsec. reg 1} 
First, we consider the set $\Omega_1$. For positive integers $a,b,c,d$ we define the matrices 
$$\gamma=\begin{pmatrix}
	a&b\\c&d
\end{pmatrix} \quad and \quad \widehat{\gamma}=\begin{pmatrix}
	-a&b\\c&-d
\end{pmatrix}.$$ These matrices belong to $G_4(n)$ and the multipliers $t_\gamma=\frac{j_\gamma(z)}{(cz+d)^{1/2}}$ and $t_{\widehat{\gamma}}$ are well-defined. In this section, we prove the following lemma:
\begin{lemma}\label{lem, 1st error}The orbital integral 
	$J_{\Reg}^1(\textbf{s},n)$ converges absolutely in the region
    \[
    \left\{\mathbf{s} = (s_{1},s_{2}) \in \mathbb{C}^{2}:|\Re(s_1)|<\frac{\kappa}{2}-1, \quad
|\Re(s_2)|<\frac{\kappa}{2}-1 \right\}.
    \]
In particular, if we set $\mathbf{s}=\mathbf{0}$ then
$$
J_{\Reg}^1(\textbf{0},n)=\frac{2n^{\kappa-1}B(\kappa/2,\kappa/2)^2}{i^{2\kappa}C_{\kappa}}\sum_{m\geq1}\frac{F(\kappa/2,\kappa/2;\kappa;n/(m+n))}{(m+n)^{\kappa/2}}\sum_{\substack{ad=n+m\\ a,d\geq1}}\sum_{\substack{bc=m\\b,c\geq1,4\mid c\\\gcd(a,b,c,d)=1}}(t_{\gamma}^{-2\kappa}i^{\kappa}+t_{\widehat{\gamma}}^{-2\kappa}i^{-\kappa}).
$$
\end{lemma}
\begin{proof}
Recall that for matrices in $\Omega_1$ we have $ad > 0$. We decompose the expression $J_{\Reg}^1(\textbf{s},n)$ into parts where $a, d > 0$ (associated with $\gamma$) and into parts with $a, d < 0$ (associated with $\widehat{\gamma}$). Using equation \eqref{eq. first reduction for roi} yields 
\begin{align*}
	J_{\Reg}^1(\textbf{s},n)
	&=\mathcal{B}(s_2)\sum_{m\geq1}\sum_{\substack{ad=n+m\\a,d\geq1}}\sum_{\substack{bc=m\\b,c\geq1,4|c\\\gcd(a,b,c,d)=1}}t_{\gamma}^{-2\kappa}\int_0^{\infty}\frac{y_1^{s_1+\kappa/2}}{(icy_1+d)^{\kappa}\left(\frac{iay_1+b}{icy_1+d}\right)^{-s_2+\kappa/2}}\frac{\,dy_1}{y_1}\\
	&\phantom{=}+\mathcal{B}(s_2)\sum_{m\geq1}\sum_{\substack{ad=n+m\\a,d \geq 1}}\sum_{\substack{bc=m\\b,c\geq1,4|c\\\gcd(a,b,c,d)=1}}t_{\widehat{\gamma}}^{-2\kappa}\int_0^{\infty}\frac{y_1^{s_1+\kappa/2}}{(icy_1-d)^{\kappa}\left(\frac{-iay_1+b}{icy_1-d}\right)^{-s_2+\kappa/2}}\frac{\,dy_1}{y_1}\\
	&=:J_{\Reg}^{1,1}(\textbf{s},n)+J_{\Reg}^{1,2}(\textbf{s},n).
\end{align*}

\noindent\textbf{Evaluation of $J_{\mathrm{Reg}}^{1,1}(\textbf{s},n)$.}
We use Lemma \ref{lemma. complex power st} to handle complex powers and rewrite
$$(icy_1+d)^{s_2+\kappa/2}(iay_1+b)^{-s_2+\kappa/2}=i^\kappa c^{s_2+\kappa/2}a^{-s_2+\kappa/2}\Big(y_1+\frac{d}{ic}\Big)^{s_2+\kappa/2}\Big(y_1+\frac{b}{ia}\Big)^{-s_2+\kappa/2}.$$

Expressing the resulting integral in terms of hypergeometric functions, using \cite[Equation 3.197(1)]{GradshteynRyzhik2007} with $\nu=s_1+\kappa/2,$ $\mu=s_2+\kappa/2$, $\rho=-s_2+\kappa/2$, $\beta=\frac{d}{ic}$, $\gamma=\frac{b}{ia}$ for $-\frac{\kappa}{2}<\Re(s_1)<\frac{\kappa}{2}$, we obtain
\begin{align*}
  J_{\Reg}^{1,1}(\textbf{s},n)=&\frac{\mathcal{B}(s_2)B(s_1+\kappa/2,-s_1+\kappa/2)}{i^{s_1+\kappa/2}}\sum_{m\geq1}\frac{F(s_2+\kappa/2,s_1+\kappa/2;\kappa;n/(n+m))}{(m+n)^{\kappa/2}}\\
  &\times\sum_{\substack{ad=n+m\\a,d\geq1}}\sum_{\substack{bc=m\\b,c\geq1,4|c\\\gcd(a,b,c,d)=1}}t_{\gamma}^{-2\kappa}b^{s_1+s_2}a^{-s_1}d^{-s_2}.
\end{align*}
\noindent\textbf{Evaluating $ J_{\Reg}^{1,2}(\textbf{s},n)$ }
We proceed in a similar way for $J_{\Reg}^{1,2}(\textbf{s},n)$. Notably, we pick up the phase factor $e^{2\pi i (-s_2+\kappa/2)}$ when we apply Lemma \ref{lemma. complex power st} to factor $(\frac{-iay_1+b}{icy_1-d})^{-s_2+\kappa/2}$. Then 
\begin{align*}
J_{\Reg}^{1,2}(\textbf{s},n)=&\frac{\mathcal{B}(s_2)}{e^{2\pi i(-s_2+\kappa/2)}}\sum_{m\geq1}\sum_{\substack{ad=n+m\\a,d \geq1}}\sum_{\substack{bc=m\\b,c\geq1,4\mid c\\\gcd(a,b,c,d)=1}}t_{\widehat{\gamma}}^{-2\kappa}\int_0^{\infty}\frac{y_1^{s_1+\kappa/2}}{(icy_1-d)^{s_2+\kappa/2}(-iay_1+b)^{-s_2+\kappa/2}}\frac{\,dy_1}{y_1}\\
=&\frac{\mathcal{B}(s_2)}{i^{\kappa}e^{\pi i(-s_2+\kappa)}}\sum_{m\geq1}\sum_{\substack{ad=n+m\\a,d\geq 1}}\sum_{\substack{bc=m\\b,c\geq1,4|c\\\gcd(a,b,c,d)=1}}\frac{t_{\widehat{\gamma}}^{-2\kappa}}{c^{s_2+\kappa/2}a^{-s_2+\kappa/2}}\\
&\times\int_0^{\infty}y_1^{s_1+\kappa/2}\left(y_1+\frac{id}{c}\right)^{-(s_2+\kappa/2)}\left(y_1+\frac{ib}{a}\right)^{-(-s_2+\kappa/2)}\frac{\,dy_1}{y_1}.
\end{align*}
Invoking the integral relation \cite[Equation 3.197(1)]{GradshteynRyzhik2007} with $\nu=s_1+\kappa/2,$ $\mu=s_2+\kappa/2$, $\rho=-s_2+\kappa/2$, $\beta=\frac{id}{c}$ and $\gamma=\frac{ib}{a}$ for $-\frac{\kappa}{2}<\Re(s_1)<\frac{\kappa}{2}$, gives
\begin{align*}
  J_{\Reg}^{1,2}(\textbf{s},n)=&\frac{\mathcal{B}(s_2)B(s_1+\kappa/2,-s_1+\kappa/2)}{i^{-s_1-2s_2+5\kappa/2}}\sum_{m\geq1}\frac{F(s_2+\kappa/2,s_1+\kappa/2;\kappa;n/(m+n))}{(m+n)^{\kappa/2}}\\
  &\times\sum_{\substack{ad=n+m\\a,d\geq 1}}\sum_{\substack{bc=m\\b,c\geq1,4|c\\\gcd(a,b,c,d)=1}}t_{\widehat{\gamma}}^{-2\kappa}a^{-s_1}d^{-s_2}b^{s_1+s_2}.
\end{align*}
Finally, we substitute the results for $ J_{\Reg}^{1,1}(\textbf{s},n)$ and $ J_{\Reg}^{1,2}(\textbf{s},n)$ into $ J_{\Reg}^{1}(\textbf{s},n)$, and we set $\textbf{s}=\textbf{0}$ to complete the proof.
\end{proof}

\subsection{The Calculation of $J_{\Reg}^2(\textbf{s},n)$}\label{subsec. reg 2}
Next, we consider the set $\Omega_2.$ To simplify computations, we introduce the following matrices $$\overline{\gamma}=\begin{pmatrix}
	-a&-b\\c&d
\end{pmatrix}, \quad \text{ and } \underline{\gamma}=\begin{pmatrix}
	a&-b\\c&-d
\end{pmatrix}.$$ Since $4\mid c$, $\overline{\gamma}, \underline{\gamma} \in G_4(n)$, which implies that the multipliers $t_{\overline{\gamma}}$ and $t_{\underline{\gamma}}$ are well defined.  

\begin{lemma}\label{lem, 2nd error} The orbital integral 
$J_{\Reg}^2(\textbf{s},n)$ converges absolutely in the region
\[
    \left\{\mathbf{s} = (s_{1},s_{2}) \in \mathbb{C}^{2}:|\Re(s_1)|<\frac{\kappa}{2}-1, \quad
|\Re(s_2)|<\frac{\kappa}{2}-1 \right\}.
    \]
In particular, if we set $\mathbf{s}=\mathbf{0}$ then
\begin{flalign*}
		J_{\Reg}^2(\textbf{0},n)&=\frac{2n^{\kappa-1}B(\kappa/2,\kappa/2)^2}{i^{2\kappa}C_{\kappa}}\sum_{m\geq1}\frac{F(\kappa/2,\kappa/2;\kappa;n/(m+n))}{(m+n)^{\kappa/2}}\sum_{\substack{ad=m\\ a,d\geq1}}\sum_{\substack{bc=m+n\\b,c\geq1,4|c\\\gcd(a,b,c,d)=1}}(t_{\overline{\gamma}}^{-2\kappa} +t_{\underline{\gamma}}^{-2\kappa}).
	\end{flalign*}
\end{lemma}
\begin{proof}
	From the definition of $\Omega_2$ and expression \eqref{eq. first reduction for roi}, we have
	\[J_{\Reg}^2(\textbf{s},n)=\mathcal{B}(s_2)\sum_{m\leq-(n+1)}\sum_{\substack{ad=n+m\\a,d\in\bZ}}\sum_{\substack{bc=m\\c\geq1,4|c,b<0\\\gcd(a,b,c,d)=1}}t_{\gamma}^{-2\kappa}\int_0^{\infty}\frac{y_1^{s_1+\kappa/2}}{j(\gamma,iy_1)^{\kappa}\gamma(iy_1)^{-s_2+\kappa/2}}\frac{\,dy_1}{y_1}.\]
    We apply the change in variable $m\mapsto-m-n$, which results in the new summation range $m \geq 1$  and the conditions $ad=-m, bc=-(m+n).$ The condition $ad=-m$ implies that the integers $a,d$ have opposite signs and so we separate the summation into two cases:
        \begin{enumerate}
        \item $d > 0 \implies a < 0$. We transform variables: $a \mapsto -a, b \mapsto -b$ (associated matrix $\overline{\gamma}$).
        \item $d < 0 \implies a > 0$. We transform variables: $d \mapsto -d, b \mapsto -b$ (associated matrix $\underline{\gamma}$).
    \end{enumerate}

    This decomposition yields

\begin{align*}
	J_{\Reg}^2(\textbf{s},n)
	=&\mathcal{B}(s_2)\sum_{m\geq1}\sum_{\substack{ad=m\\a,d\geq1}}\sum_{\substack{bc=m+n\\b,c\geq1,4|c\\\gcd(a,b,c,d)=1}}t_{\overline{\gamma}}^{-2\kappa}\int_0^{\infty}\frac{y_1^{s_1+\kappa/2}}{(icy_1+d)^{\kappa}\left(\frac{-iay_1-b}{icy_1+d}\right)^{-s_2+\kappa/2}}\frac{\,dy_1}{y_1}\\
	&+\mathcal{B}(s_2)\sum_{m\geq1}\sum_{\substack{ad=m\\a,d\geq1}}\sum_{\substack{bc=m+n\\b,c\geq1,4|c\\\gcd(a,b,c,d)=1}}t_{\underline{\gamma}}^{-2\kappa}\int_0^{\infty}\frac{y_1^{s_1+\kappa/2}}{(icy_1-d)^{\kappa}\left(\frac{iay_1-b}{icy_1-d}\right)^{-s_2+\kappa/2}}\frac{\,dy_1}{y_1}\\
	=&:J_{\Reg}^{2,1}(\textbf{s},n)+J_{\Reg}^{2,2}(\textbf{s},n).
\end{align*}	
\noindent\textbf{Evaluation of $J_{\Reg}^{2,1}(\textbf{s},n)$:}
Similarly to the computations in the prior section, we use Lemma \ref{lemma. complex power st} to handle the factorization of complex powers and obtain
\begin{align*}
J_{\Reg}^{2,1}(\textbf{s},n)
=&\frac{\mathcal{B}(s_2)}{e^{2\pi i(-s_2+\kappa/2)}}\sum_{m\geq1}\sum_{\substack{ad=m\\a,d\geq1}}\sum_{\substack{bc=m+n\\b,c\geq1,4|c\\\gcd(a,b,c,d)=1}}\frac{t_{\overline{\gamma}}^{-2\kappa}}{(ic)^{s_2+\kappa/2}(-ia)^{-s_2+\kappa/2}}\\
&\times\int_0^{\infty}y_1^{s_1+\kappa/2}\left(y_1+\frac{d}{ic}\right)^{-(s_2+\kappa/2)}\left(y_1+\frac{b}{ia}\right)^{-(-s_2+\kappa/2)}\frac{\,dy_1}{y_1}.
\end{align*}

Applying the integral relation \cite[Equation 3.197(1)]{GradshteynRyzhik2007} with $\nu=s_1+\kappa/2,$ $\mu=-s_2+\kappa/2$, $\rho=s_2+\kappa/2$, $\beta=\frac{b}{ia}$ and $\gamma=\frac{d}{ic}$ for $-\frac{\kappa}{2}<\Re(s_1)<\frac{\kappa}{2}$) implies
\begin{flalign*}
	 J_{\Reg}^{2,1}(\textbf{s},n)&=\frac{\mathcal{B}(s_2)B(s_1+\kappa/2,-s_1+\kappa/2)}{i^{s_1-2s_2+3\kappa/2}}\sum_{m\geq1}\frac{F(-s_2+\kappa/2,s_1+\kappa/2;\kappa;n/(m+n))}{(m+n)^{\kappa/2}}\\
  &\phantom{=}\times\sum_{\substack{ad=m\\a,d\geq1}}\sum_{\substack{bc=m+n\\b,c\geq1,4\mid c\\\gcd(a,b,c,d)=1}}t_{\overline{\gamma}}^{-2\kappa}d^{s_1-s_2}c^{-s_1}b^{s_2}.
\end{flalign*}

\noindent\textbf{Evaluation of $J_{\Reg}^{2,2}(\textbf{s},n):$} We again use Lemma \ref{lemma. complex power st} and rewrite
\begin{align*}
J_{\Reg}^{2,2}(\textbf{s},n)
=&\mathcal{B}(s_2)\sum_{m\geq1}\sum_{\substack{ad=m\\a,d\geq1}}\sum_{\substack{bc=m+n\\b,c\geq1,4|c\\\gcd(a,b,c,d)=1}}\frac{t_{\underline{\gamma}}^{-2\kappa}}{(ic)^{s_2+\kappa/2}(ia)^{-s_2+\kappa/2}}\\
&\times\int_0^{\infty}y_1^{s_1+\kappa/2}\left(y_1+\frac{id}{c}\right)^{-(s_2+\kappa/2)}\left(y_1+\frac{ib}{a}\right)^{-(-s_2+\kappa/2)}\frac{\,dy_1}{y_1}.
\end{align*}

Applying the relation \cite[Equation 3.197(1)]{GradshteynRyzhik2007} with $\nu=s_1+\kappa/2,$ $\mu=-s_2+\kappa/2$, $\rho=s_2+\kappa/2$, $\beta=\frac{ib}{a}$ and $\gamma=\frac{id}{c}$ for $-\frac{\kappa}{2}<\Re(s_1)<\frac{\kappa}{2}$) yields
\begin{flalign*}
	 J_{\Reg}^{2,2}(\textbf{s},n)&=\frac{\mathcal{B}(s_2)B(s_1+\kappa/2,-s_1+\kappa/2)}{i^{-s_1+3\kappa/2}}\sum_{m\geq1}\frac{F(-s_2+\kappa/2,s_1+\kappa/2;\kappa;n/(m+n))}{(m+n)^{\kappa/2}}\\
  &\phantom{=}\times\sum_{\substack{ad=m\\a,d\geq1}}\sum_{\substack{bc=m+n\\b,c\geq1,4|c\\\gcd(a,b,c,d)=1}}t_{\underline{\gamma}}^{-2\kappa}d^{s_1-s_2}c^{-s_1}b^{s_2}.
\end{flalign*}

Substituting $ J_{\Reg}^{2,1}(\textbf{s},n)$ and $ J_{\Reg}^{2,2}(\textbf{s},n)$ into $ J_{\Reg}^{2}(\textbf{s},n)$, and setting $\textbf{s}=\textbf{0}$ proves the claim.
\end{proof}

\subsection{The Calculation of $J_{\Reg}^3(\textbf{s},n)$}\label{subsec. reg 3} 
Finally, we treat the case of $\Omega_3$. We introduce the matrices $${}^{\iota}{\gamma}=\begin{pmatrix}
	a&-b\\c&d
\end{pmatrix} \quad \text{ and } \quad  \gamma^{\iota}=\begin{pmatrix}
	-a&-b\\c&-d
\end{pmatrix},$$ belonging to $G_4(n)$. As in the prior sections this implies that the multipliers $t_{{{}^{\iota}\gamma}}$ and $t_{\gamma^{\iota}}$ are well defined.  

\begin{lemma}\label{lem, 3rd error} The orbital integral 
	$J_{\Reg}^3(\textbf{s},n)$ converges absolutely in the region
\[
    \left\{\mathbf{s} = (s_{1},s_{2}) \in \mathbb{C}^{2}:|\Re(s_1)|<\frac{\kappa}{2}-1, \quad
|\Re(s_2)|<\frac{\kappa}{2}-1 \right\}.
    \]
In particular, if we set $\mathbf{s}=\mathbf{0}$ then
	\begin{flalign*}
		J_{\Reg}^3(\textbf{0},n)&=\frac{2n^{\kappa-1}B(\kappa/2,\kappa/2)^2}{i^{2\kappa}C_{\kappa}}\sum_{1\leq m\leq n-1}\frac{F(\kappa/2,\kappa/2;\kappa;n/m)}{m^{\kappa/2}}\sum_{\substack{ad=n-m\\a,d\geq1}}\sum_{\substack{bc=m\\b,c\geq1,4|c\\\gcd(a,b,c,d)=1}}(t_{{}^{\iota}{\gamma}}^{-2\kappa} +t_{{}{\gamma}^{\iota}}^{-2\kappa}).
	\end{flalign*}
\end{lemma}
\begin{proof}
	From the definition of $\Omega_3$ and expression \eqref{eq. first reduction for roi}, we obtain
	\[J_{\Reg}^3(\textbf{s},n)=\mathcal{B}(s_2)\sum_{-(n-1)\leq m\leq -1}\sum_{\substack{ad=n+m\\a,d\in\bZ}}\sum_{\substack{bc=m\\c\geq1,4|c,b<0\\\gcd(a,b,c,d)=1}}t_{\gamma}^{-2\kappa}\int_0^{\infty}\frac{y_1^{s_1+\kappa/2}}{j(\gamma,iy_1)^{\kappa}\gamma(iy_1)^{-s_2+\kappa/2}}\frac{\,dy_1}{y_1}.\]
	Changing the variable $m\mapsto-m$ and $b\mapsto -b$ yields the new summation range $1 \leq m \leq n-1$ and conditions $ad=n-m$, $b\geq 1$. We divide the summation again into two cases, depending on the sign of $d$:
    \begin{enumerate}
        \item $d > 0 \implies a > 0$ (associated matrix ${}^{\iota}{\gamma}$).
        \item $d < 0 \implies a < 0$. We map $a \mapsto -a, d \mapsto -d$ (associated matrix $\gamma^{\iota}$).
    \end{enumerate}
    This decomposition yields
	\begin{align*}
	J_{\Reg}^3(\textbf{s},n)
	=&\mathcal{B}(s_2)\sum_{1\leq m\leq n-1}\sum_{\substack{ad=n-m\\a,d\geq1}}\sum_{\substack{bc=m\\b,c\geq1,4|c\\\gcd(a,b,c,d)=1}}t_{{}^{\iota}\gamma}^{-2\kappa}\int_0^{\infty}\frac{y_1^{s_1+\kappa/2}}{(icy_1+d)^{\kappa}\left(\frac{iay_1-b}{icy_1+d}\right)^{-s_2+\kappa/2}}\frac{\,dy_1}{y_1}\\
	&+\mathcal{B}(s_2)\sum_{1\leq m\leq n-1}\sum_{\substack{ad=n-m\\a,d\geq1}}\sum_{\substack{bc=m\\b,c\geq1,4|c\\\gcd(a,b,c,d)=1}}t_{\gamma^{\iota}}^{-2\kappa}\int_0^{\infty}\frac{y_1^{s_1+\kappa/2}}{(icy_1-d)^{\kappa}\left(\frac{-iay_1-b}{icy_1-d}\right)^{-s_2+\kappa/2}}\frac{\,dy_1}{y_1}\\
	=:& J_{\Reg}^{3,1}(\textbf{s},n)+J_{\Reg}^{3,2}(\textbf{s},n).
\end{align*}
\noindent\textbf{Evaluation of $J_{\Reg}^{3,1}(\textbf{s},n)$:}
Using Lemma \ref{lemma. complex power st}, we obtain
\begin{align*}
J_{\Reg}^{3,1}(\textbf{s},n)=&\mathcal{B}(s_2)\sum_{1\leq m\leq n-1}\sum_{\substack{ad=n-m\\a,d\geq1}}\sum_{\substack{bc=m\\b,c\geq1,4|c\\\gcd(a,b,c,d)=1}}\frac{t_{{}^{\iota}{\gamma}}^{-2\kappa}}{(ic)^{s_2+\kappa/2}(ia)^{-s_2+\kappa/2}}\\
&\times\int_0^{\infty}y_1^{s_1+\kappa/2}\left(y_1+\frac{d}{ic}\right)^{-(s_2+\kappa/2)}\left(y_1+\frac{ib}{a}\right)^{-(-s_2+\kappa/2)}\frac{\,dy_1}{y_1}.
\end{align*}
Invoking the relation \cite[Equation 3.197(1)]{GradshteynRyzhik2007} with $\nu=s_1+\kappa/2,$ $\mu=-s_2+\kappa/2$, $\rho=s_2+\kappa/2$, $\beta=\frac{ib}{a}$ and $\gamma=\frac{d}{ic}$ for $-\frac{\kappa}{2}<\Re(s_1)<\frac{\kappa}{2}$, implies
\begin{flalign*}
	 J_{\Reg}^{3,1}(\textbf{s},n)&=\frac{\mathcal{B}(s_2)B(s_1+\kappa/2,-s_1+\kappa/2)}{i^{s_1-2s_2+3\kappa/2}}\sum_{1\leq m\leq n-1}\frac{F(-s_2+\kappa/2,s_1+\kappa/2;\kappa;n/m)}{m^{\kappa/2}}\\
  &\phantom{=}\times\sum_{\substack{ad=n-m\\a,d\geq1}}\sum_{\substack{bc=m\\b,c\geq1,4|c\\\gcd(a,b,c,d)=1}}t_{{}^{\iota}{\gamma}}^{-2\kappa}d^{s_1-s_2}c^{-s_1}b^{s_2}.
\end{flalign*}
\noindent\textbf{Evaluation of $J_{\Reg}^{3,2}(\textbf{s},n)$:}
Finally, we investigate $J_{\Reg}^{3,2}(\textbf{s},n)$. We proceed with a similar computation as in previous sections, using Lemma \ref{lemma. complex power st} and obtain
\begin{align*}
J_{\Reg}^{3,2}(\textbf{s},n)=&\frac{\mathcal{B}(s_2)}{e^{2\pi i(-s_2+\kappa/2)}}\sum_{1\leq m\leq n-1}\sum_{\substack{ad=n-m\\a,d\geq1}}\sum_{\substack{bc=m\\b,c\geq1,4|c\\\gcd(a,b,c,d)=1}}t_{{\gamma}^{\iota}}^{-2\kappa}\int_0^{\infty}\frac{y_1^{s_1+\kappa/2}}{(icy_1-d)^{s_2+\kappa/2}(-iay_1-b)^{-s_2+\kappa/2}}\frac{\,dy_1}{y_1}\\
=&\frac{\mathcal{B}(s_2)}{i^{-4s_2+2\kappa}}\sum_{m\geq1}\sum_{\substack{ad=m\\a,d\geq1}}\sum_{\substack{bc=m+n\\b,c\geq1,4|c\\\gcd(a,b,c,d)=1}}\frac{t_{\underline{\gamma}}^{-2\kappa}}{(ic)^{s_2+\kappa/2}(-ia)^{-s_2+\kappa/2}}\\
&\times\int_0^{\infty}y_1^{s_1+\kappa/2}\left(y_1+\frac{id}{c}\right)^{-(s_2+\kappa/2)}\left(y_1+\frac{b}{ia}\right)^{-(-s_2+\kappa/2)}\frac{\,dy_1}{y_1}.
\end{align*}	
Applying \cite[Equation 3.197(1)] {GradshteynRyzhik2007} with $\nu=s_1+\kappa/2,$ $\mu=-s_2+\kappa/2$, $\rho=s_2+\kappa/2$, $\beta=\frac{b}{ia}$ and $\gamma=\frac{id}{c}$ for $-\frac{\kappa}{2}<\Re(s_1)<\frac{\kappa}{2}$ gives
\begin{align*}
	 J_{\Reg}^{3,2}(\textbf{s},n)=&\frac{\mathcal{B}(s_2)B(s_1+\kappa/2,-s_1+\kappa/2)}{i^{-s_1+3\kappa/2}}\sum_{1\leq m\leq n-1}\frac{F(-s_2+\kappa/2,s_1+\kappa/2;\kappa;n/m)}{m^{\kappa/2}}\\
  &\times\sum_{\substack{ad=n-m\\a,d\geq1}}\sum_{\substack{bc=m\\b,c\geq1,4|c\\\gcd(a,b,c,d)=1}}t_{{}{\gamma}^{\iota}}^{-2\kappa}d^{s_1-s_2}c^{-s_1}b^{s_2}.
\end{align*}
We conclude the proof by substituting $ J_{\Reg}^{3,1}(\textbf{s},n)$ and $ J_{\Reg}^{3,2}(\textbf{s},n)$ into $ J_{\Reg}^{3}(\textbf{s},n)$ and setting $\textbf{s}=\textbf{0}$.
\end{proof}

\subsection{Estimation of $J_{\Reg}
(\mathbf{0},n)$}\label{subsec. majorization}
In this section, we establish an upper bound for the regular orbital integral. In particular, our primary aim is to prove the following result:
\begin{prop}\label{prop. est of error general}
	Let $n$ be an odd square integer. For any $\varepsilon > 0$, the regular orbital integral satisfies:
	\[J_{\Reg}(\textbf{0},n)\ll_\varepsilon \frac{n^{\frac{\kappa}{2}+\varepsilon}B(\kappa/2,\kappa/2)}{|C_{\kappa}|\kappa^{\frac{1}{2}}}.\]
\end{prop}
\begin{proof}
	Notice that 
	\[J_{\Reg}(\textbf{s},n)=J_{\Reg}^1(\textbf{s},n)+J_{\Reg}^2(\textbf{s},n)+J_{\Reg}^3(\textbf{s},n).\]
	The result follows immediately from the bounds on the individual components given in Lemma \ref{lem. 1 2 term in error} and Lemma \ref{lem. 3 term in error}. 
\end{proof}

The main contribution to the error term will arise from $J_{\Reg}^3(\textbf{0},n)$, and so the following, rather crude, bound for $J^1_{\Reg}(\textbf{0},n)+J^2_{\Reg}(\textbf{0},n)$ will suffice for our purposes:

\begin{lemma}\label{lem. 1 2 term in error}
	Let $J_{\Reg}^1(\textbf{s},n)$ and  $J_{\Reg}^2(\textbf{s},n)$ be defined as in Lemma \ref{lem, 1st error} and Lemma \ref{lem, 2nd error} respectively. Then for any $\varepsilon > 0,$
	\[J_{\Reg}^1(\textbf{0},n)+J_{\Reg}^2(\textbf{0},n)\ll_\varepsilon \frac{n^{\frac{\kappa}{2}+\varepsilon}B(\kappa/2,\kappa/2)}{|C_{\kappa}|\kappa^{\frac{1}{2}}}.\]
\end{lemma}

\begin{proof}
We restrict our attention to bounding  $J_{\Reg}^1(\textbf{0},n)$, the case of $J_{\Reg}^2(\textbf{0},n)$ being treated similarly. Since $|t_{\gamma}|=1$ for $\gamma\in G_4(n)$, and
$$\sum_{\substack{ad=m\\ a,d\geq1}}\sum_{\substack{bc=m+n\\b,c\geq1,4|c\\\gcd(a,b,c,d)=1}}1 \ll_\varepsilon m^\varepsilon,$$ we have
  \begin{equation}\label{equ:preliminary_hypergeo_J_reg_estimate}
  J_{\Reg}^1(\textbf{0},n)\ll_\varepsilon \frac{n^{\kappa-1+\varepsilon}B(\kappa/2,\kappa/2)^2}{|C_{\kappa}|}\sum_{m\geq1}\frac{m^{\varepsilon}}{(m+n)^{\kappa/2}}|F(\kappa/2,\kappa/2;\kappa;n/(m+n))|.\end{equation}
Next we relate the hypergeometric function to the Legendre function of the second kind, $Q_\nu(z).$ 
Applying the transformation formulas \cite[9.134(1) and 8.820(2)]{GradshteynRyzhik2007} and the duplication formula for the Gamma function, we obtain
	\begin{equation}\label{eq. hypergeometric function and legrende function}
		F\left(\frac{\kappa}{2},\frac{\kappa}{2};\kappa;1-z\right)=\frac{2(1-z)^{-\frac{\kappa}{2}}}{B\left(\frac{\kappa}{2},\frac{\kappa}{2}\right)}Q_{\frac{\kappa}{2}-1}\left(\frac{1+z}{1-z}\right).
	\end{equation} 
 Setting $z=\frac{m}{m+n}$ and substituting relation \eqref{eq. hypergeometric function and legrende function} into inequality \eqref{equ:preliminary_hypergeo_J_reg_estimate} yields
\begin{equation}\label{eq. J1 by legerdene}
J_{\Reg}^1(\textbf{0},n)\ll_\varepsilon \frac{n^{\frac{\kappa}{2}-1+\varepsilon}B(\kappa/2,\kappa/2)}{|C_{\kappa}|}\sum_{m\geq1}m^{\varepsilon}\left|Q_{\frac{\kappa}{2}-1}\left(\frac{2m}{n}+1\right)\right|.
\end{equation}
To estimate the Legendre function we use the following integral representation given in \cite[Equation 8.712]{GradshteynRyzhik2007}:
\[Q_{\frac{\kappa}{2}-1}\left(\frac{2m}{n}+1\right)= \frac{1}{2^{\frac{\kappa}{2}}}\int_{-1}^1\frac{(1-t^2)^{\frac{\kappa}{2}-1}}{\left(\frac{2m}{n}+1-t\right)^{\frac{\kappa}{2}}}\,dt.\]

We now split the analysis of this integral into two ranges.

\textbf{Case 1: $m \leq 4n$.} We have

	 \[\int_{-1}^1\frac{(1-t^2)^{\frac{\kappa}{2}-1}}{\left(\frac{2m}{n}+1-t\right)^{\frac{\kappa}{2}}}\,dt\ll \frac{n}{m}\int_{-1}^1\frac{(1-t^2)^{\frac{\kappa}{2}-1}}{\left(1-t\right)^{\frac{\kappa}{2}-1}}\,dt\ll \frac{n}{m}\frac{2^{\frac{\kappa}{2}}}{\kappa}.\]
	 Consequently,
	 \[Q_{\frac{\kappa}{2}-1}\left(\frac{2m}{n}+1\right)\ll \frac{n}{m\kappa}.\]
     
\textbf{Case 2: $m > 4n$.}
	 Suppose that $2^{\ell}n< m\leq 2^{\ell+1}n$ with $\ell\geq2.$ Then
	 \[\int_{-1}^1\frac{(1-t^2)^{\frac{\kappa}{2}-1}}{\left(\frac{2m}{n}+1-t\right)^{\frac{\kappa}{2}}}\,dt\ll \frac{1}{2^{\frac{\kappa}{2}\ell}}\int_{-1}^{1}(1-t^2)^{\frac{\kappa}{2}-1}\,dt.\]
	 This integral is even. So after the change of variable $t\mapsto\sqrt{t},$ together with an application of Stirling's formula, we obtain
	 \[\int_{-1}^{1}(1-t^2)^{\frac{\kappa}{2}-1}\,dt=\int_0^1 (1-t)^{\frac{\kappa}{2}-1}t^{-\frac{1}{2}}\,dt=\frac{\Gamma\left(\frac{\kappa}{2}\right)\Gamma\left(\frac{1}{2}\right)}{\Gamma\left(\frac{\kappa+1}{2}\right)}\ll \frac{1}{\kappa^{1/2}}.\]
	 When $2^\ell n < m \leq 2^{\ell+1}n$, the above computations result in
	 \[Q_{\frac{\kappa}{2}-1}\left(\frac{2m}{n}+1\right)\ll \frac{1}{2^{\frac{\kappa}{2}(\ell+1)}}\frac{1}{\kappa}.\]

	 Next, we sum over dyadic intervals. This gives
	 \begin{flalign*}
	 	\sum_{m\geq1}m^{\varepsilon}\left|Q_{\frac{\kappa}{2}-1}\left(\frac{2m}{n}+1\right)\right|&=\sum_{m\leq 4n}m^{\varepsilon}\left|Q_{\frac{\kappa}{2}-1}\left(\frac{2m}{n}+1\right)\right|+\sum_{\ell\geq2}\quad\sum_{2^{\ell}n< m\leq 2^{\ell}n}m^{\varepsilon}\left|Q_{\frac{\kappa}{2}-1}\left(\frac{2m}{n}+1\right)\right|\\
	 	&\ll_{\varepsilon}\frac{1}{\kappa}\sum_{m\leq 4n}\frac{n}{m^{1-\varepsilon}}+\sum_{\ell\geq2}\quad\sum_{2^{\ell}n< m\leq 2^{\ell+1}n}m^{\varepsilon}\frac{1}{2^{\frac{\kappa}{2}(\ell+1)}}\frac{1}{\kappa^{1/2}}\\
	 	&\ll_{\varepsilon} \frac{n^{1+\varepsilon}}{\kappa^{\frac{1}{2}}}.
	 \end{flalign*}	 
	Upon inserting this result into inequality \eqref{eq. J1 by legerdene}, we conclude
	\[J_{\Reg}^1(\textbf{0},n)\ll_{\varepsilon} \frac{n^{\frac{\kappa}{2}+\varepsilon}B(\kappa/2,\kappa/2)}{|C_{\kappa}|\kappa^{\frac{1}{2}}}.\]
\end{proof}

\begin{remark}\label{rem. n=1 remark for regular}
    When $n=1,$ we claim that
    \[J_{\Reg}(\textbf{0},1)=J_{\Reg}^1(\textbf{0},1)+J_{\Reg}^2(\textbf{0},1)\ll \frac{B(\kappa/2,\kappa/2)}{|C_{\kappa}|2^{\frac{\kappa}{2}}}.\]
    By Proposition \ref{prop, the error term in RTF}, we have $J_{\Reg}^3(\textbf{0},1)=0$. It suffices to study $J_{\Reg}^1(\textbf{0},1).$ Notice that we have the inequality: 
\[2m+1-t\geq 2|1-t|,\]
when $t\in[-1,1]$. Then we have
	 \[\int_{-1}^1\frac{(1-t^2)^{\frac{\kappa}{2}-1}}{\left(2m+1-t\right)^{\frac{\kappa}{2}}}\,dt\ll \int_{-1}^1\frac{(1-t^2)^{\frac{\kappa}{2}-1}}{2^{\frac{\kappa}{2}}\left(1-t\right)^{\frac{\kappa}{2}-1}}\,dt\ll \frac{1}{\kappa}.\]
	 This implies
	 \[Q_{\frac{\kappa}{2}-1}\left(2m+1\right)\ll \frac{1}{2^{\frac{\kappa}{2}}\kappa},\]
  and consequently, 
  	 \begin{flalign*}
	 \sum_{m\geq1}m^{\varepsilon}\left|Q_{\frac{\kappa}{2}-1}\left(2m+1\right)\right|&=\sum_{m\leq 4}m^{\varepsilon}\left|Q_{\frac{\kappa}{2}-1}\left(2m+1\right)\right|+\sum_{\ell\geq2}\quad\sum_{2^{\ell}< m\leq 2^{\ell}}m^{\varepsilon}\left|Q_{\frac{\kappa}{2}-1}\left(2m+1\right)\right|\ll\frac{1}{2^{\frac{\kappa}{2}}}.
	 \end{flalign*}	 
     Insert it into $J_{\Reg}^1(\mathbf{0},1)$ and we prove the claim.
\end{remark}

\bigskip

Finally, we investigate $J_{\Reg}^3(\mathbf{0},n).$
\begin{lemma}\label{lem. 3 term in error}
	Let $J_{\Reg}^3(\textbf{s},n)$ be defined as in Lemma \ref{lem, 3rd error}. Then for any $\varepsilon > 0$,
	\[J_{\Reg}^3(\textbf{0},n)\ll_\varepsilon \frac{n^{\frac{\kappa}{2}+\varepsilon}B(\kappa/2,\kappa/2)}{|C_{\kappa}|\kappa^{\frac{1}{2}}}.\]
\end{lemma}
\begin{proof}
First, we notice again that $|t_{\gamma}|=1$ for $\gamma\in G_4(n),$ so that 
  \[J_{\Reg}^3(\textbf{0},n)\ll_{\varepsilon} \frac{2n^{\kappa-1}B(\kappa/2,\kappa/2)^2}{i^{2\kappa}C_{\kappa}}\sum_{1\leq m\leq n-1}\frac{m^{\varepsilon}}{m^{\kappa/2}}|F(\kappa/2,\kappa/2;\kappa;n/m)|.\]
  As before, we use \eqref{eq. hypergeometric function and legrende function} and obtain
\begin{equation}\label{eq. J3 by legerdene}
J_{\Reg}^3(\textbf{0},n)\ll_{\varepsilon} \frac{n^{\frac{\kappa}{2}-1+\varepsilon}B(\kappa/2,\kappa/2)}{|C_{\kappa}|}\sum_{1\leq m\leq n-1}m^{\varepsilon}\left|Q_{\frac{\kappa}{2}-1}\left(\frac{2m}{n}-1\right)\right|.
\end{equation} 
Since $1\leq m\leq n-1$, we have $\frac{2m}{n}-1\in (-1,1).$ Applying the bound \cite[Equation 8.724(2)]{GradshteynRyzhik2007} to the Legendre function yields
\[Q_{\frac{\kappa}{2}-1}\left(\frac{2m}{n}-1\right)\ll \frac{1}{\sqrt{\kappa}}\frac{1}{(m/n)^{1/4}(1-m/n)^{1/4}}.\]
Then we obtain
\begin{flalign*}
	\sum_{1\leq m\leq n-1}m^{\varepsilon}\left|Q_{\frac{\kappa}{2}-1}\left(\frac{2m}{n}-1\right)\right|&=\sum_{1\leq m\leq n/2}m^{\varepsilon}\left|Q_{\frac{\kappa}{2}-1}\left(\frac{2m}{n}-1\right)\right|+\sum_{n/2<m\leq n-1}m^{\varepsilon}\left|Q_{\frac{\kappa}{2}-1}\left(\frac{2m}{n}-1\right)\right|\\
	&\ll_{\varepsilon} \frac{n^{\varepsilon}}{\sqrt{\kappa}}\sum_{1\leq m\leq n/2}\left(\frac{n}{m}\right)^{\frac{1}{4}}+\frac{n^{\varepsilon}}{\sqrt{\kappa}}\sum_{n/2< m\leq n}\left(\frac{n}{n-m}\right)^{\frac{1}{4}}\ll_{\varepsilon}\frac{n^{1+\varepsilon}}{\sqrt{\kappa}}.
\end{flalign*}
 Upon inserting this bound into \eqref{eq. J3 by legerdene} we conclude
 \[J_{\Reg}^3(\textbf{0},n)\ll_{\varepsilon} \frac{n^{\frac{\kappa}{2}+\varepsilon}B(\kappa/2,\kappa/2)}{|C_{\kappa}|\kappa^{\frac{1}{2}}}.\]
\end{proof}

\section{Proof of Theorems}\label{sec. proof}

\subsection{Proof of Theorem \ref{thm. second moment closed formula}}\label{subsec. proof A}
In this section, we prove Theorem \ref{thm. second moment closed formula}.
\begin{proof}
	Applying  \eqref{eq. full RTF} with $\mathbf{s}=\mathbf{0}$ we obtain
	\[n^{\frac{\kappa-1}{2}}\frac{\Gamma(\kappa/2)^2}{(2\pi)^{\kappa}}\sum_{f\in H_{\kappa}(4)}\frac{\Lambda_f(n)L(1/2,f)^2}{\|f\|^2}=J_{\si}(\textbf{0},n)+J_{\Reg}(\textbf{0},n).\]
The singular orbit integral is explicitly given in Proposition \ref{prop. main term for the general case}. $J_{\Reg}(\textbf{0},n)$ is explicitly given in Proposition \ref{prop, the error term in RTF}, and the bound for $J_{\Reg}(\textbf{0},n)$ is given by Proposition \ref{prop. est of error general}. Combining these results completes the proof.
\end{proof}

When $n=1,$ we can prove a stronger error term: 
\begin{prop}\label{prop. second moment n=1}
	When $n=1,$ 
	\[\frac{\Gamma(\kappa/2)^2}{(2\pi)^{\kappa}}\sum_{f\in H_{\kappa}(4)}\frac{L(1/2,f)^2}{\|f\|^2}=\frac{4\Gamma\left(\frac{\kappa}{2}\right)^2}{i^{\kappa}C_{\kappa}\Gamma(\kappa)}\left(\log\kappa+O\left(1+\frac{1}{\kappa}\right)\right)\]
\end{prop}
\begin{proof}
	The proof is identical to that of Theorem \ref{thm. second moment closed formula} except it uses Remark \ref{rem. singular orbital integral when n=1} and Remark \ref{rem. n=1 remark for regular} instead.
\end{proof}

\subsection{Proof of Theorem \ref{thm. subconvexity theorem}}\label{subsec. proof C}
In this section, we prove Theorem \ref{thm. subconvexity theorem}. We first require the following lemma:

\begin{lemma}\label{lem. petersson norm lemma}
	Let $f\in H_{\kappa}^+(4)$ be a Hecke eigenform in the Kohnen plus space. Assume that $f$ satisfies the normalization \eqref{eq. normalization of forms}. Then
	\[\frac{1}{\|f\|^2}=\frac{(4\pi)^{\kappa}}{2\pi\Gamma(\kappa)L(1,\sym^2F)},\]
	where $F:=F_f\in\cS_{2\kappa-1}(\SL_2(\bZ))$ is the Shimura correspondent of $f$ and $L(s,\sym^2F)$ is the symmetric square $L$-function of $F.$
\end{lemma}
\begin{proof}
	By \cite[Theorem B]{LuoRamakrishnan1997}, there exists a fundamental discriminant $D$ satisfying $(-1)^{\kappa-\frac{1}{2}}D>0$ and $a_f(|D|)\neq0.$ We then apply the explicit Kohnen-Zagier formula \cite[Theorem 1]{KohnenZagier1981} to obtain
	\[\frac{a_f(|D|)^2}{\|f\|^2}=\frac{(4\pi)^{\kappa-1}}{\Gamma(\kappa)}\frac{2\pi^2L\left(\frac{1}{2},F\times\chi_D\right)}{L(1,\sym^2F)}.\]
	Inserting \eqref{eq. normalization of forms} into this identity completes the proof.
\end{proof}

Let $m$ be an odd integer and $n=m^2.$ By Theorem \ref{thm. second moment closed formula}, we have
\[(m^2)^{\frac{\kappa-1}{2}}\frac{\Gamma(\kappa/2)^2}{(2\pi)^{\kappa}}\sum_{f\in H_{\kappa}(4)}\frac{\Lambda_f(m^2)|L(1/2,f)|^2}{\|f\|^2}=J_{\si}(\textbf{0},m^2)+O_{\varepsilon}\left(\frac{(m^2)^{\frac{\kappa}{2}+\varepsilon}B(\kappa/2,\kappa/2)}{|C_{\kappa}|\kappa^{\frac{1}{2}}}\right).\]
Applying Corollary \ref{cor. estimation of the singular orbital integral} and Proposition \ref{prop, the error term in RTF} we obtain
\[(m^2)^{\frac{\kappa-1}{2}}\frac{\Gamma(\kappa/2)^2}{(2\pi)^{\kappa}}\sum_{f\in H_{\kappa}(4)}\frac{\Lambda_f(m^2)|L(1/2,f)|^2}{\|f\|^2}\ll_{\varepsilon} (m^2)^{\frac{\kappa}{2}-\frac{3}{4}+\varepsilon}\frac{\kappa^{\varepsilon}B(\kappa/2,\kappa/2)}{|C_{\kappa}|}+\frac{(m^2)^{\frac{\kappa}{2}+\varepsilon}B(\kappa/2,\kappa/2)}{|C_{\kappa}|\kappa^{\frac{1}{2}}}.\] 
Inserting the definition of $C_{\kappa}$ together with \eqref{eq. explict hecke eigenvalues} and \eqref{eq. kappa constant} yields
\begin{equation}\label{eq. explicit bound in subconvexity}
\sum_{f\in H_{\kappa}(4)}\frac{\lambda_F(m)|L(1/2,f)|^2}{\|f\|^2}\ll_{\varepsilon}\frac{(4\pi)^{\kappa}}{\Gamma(\kappa)}\left(\frac{\kappa^{1+\varepsilon}}{m^{\frac{1}{2}}}+m^{1+\varepsilon}\kappa^{\frac{1}{2}+\varepsilon}\right).	
\end{equation}

We can now prove Theorem \ref{thm. subconvexity theorem}:
\begin{proof}
Let $f_0\in H_{\kappa}^+(4),$ a Hecke eigenform in the Kohnen plus space and let $F_0:=F_{f_0}\in \cS_{2\kappa-1}(\SL_2(\bZ))$ be the Shimura correspondent of $f_0.$ Also let $p$ be an odd prime. The Hecke relation $\lambda_{F_0}(p)^2=\lambda_{F_0}(p^2)+1$ implies
\[|\lambda_{F_0}(p^{\alpha_{p}})|\geq\frac{1}{4},\]
for some $\alpha_{p}\in\{1,2\}.$ If both $|\lambda_{F_0}(p)|,|\lambda_{F_0}(p^2)|\geq\frac{1}{4},$ we set $\alpha_p=1.$ Let $X \gg 1$ and set 
\[\mathcal{P}=\left\{p^{\alpha_{p}}:p\in[X,2X],\alpha_p\in\{1,2\}\hspace{2mm}\mbox{$p$ is a prime},\hspace{2mm}|\lambda_{F_0}(p^{\alpha_{p}})|\geq\frac{1}{4}\right\}.\]
This implies that if $\ell=p^{\alpha_p}\in\mathcal{P}$ then $\ell\leq 4X^2$ and $|\lambda_{F_0}(\ell)|\geq\frac{1}{4}.$ Let $\{c_{\ell}\}_{\ell\in \cP}$ be a complex sequence satisfying $|c_{\ell}|=1$ for $\ell\in\cP.$ Consider the sum
\[\cJ=\sum_{f\in H_{\kappa}(4)}\frac{|L(1/2,f)|^2}{\|f\|^2}\left|\sum_{\ell\in\cP}c_{\ell}\lambda_F(\ell)\right|^2,\]
where we take $c_{\ell}=\frac{\lambda_{F_0}(\ell)}{|\lambda_{F_0}(\ell)|}$ which is the sign of $\lambda_{F_0}(\ell).$ The positivity of summands in $\cJ$ and that $|\lambda_{F_0}(\ell)|\geq\frac{1}{4}$ together yield the lower bound
\[\left(\frac{X}{\log X}\right)^2\frac{L(1/2,f_0)^2}{\|f_0\|^2}\ll \cJ.\]
On the other hand, open the square in $\cJ$ and interchange the order of summation to obtain
\[\cJ=\sum_{\ell\in\cP}|c_{\ell}|^2\sum_{f\in H_{\kappa}(4)}\frac{\lambda_F(\ell)^2|L(1/2,f)|^2}{\|f\|^2}+\sum_{\ell_1\neq\ell_2\in\cP}c_{\ell_1}\overline{c_{\ell_2}}\sum_{f\in H_{\kappa}(4)}\frac{\lambda_F(\ell_1\ell_2)|L(1/2,f)|^2}{\|f\|^2}.\]
Applying the aforementioned Hecke relation and \eqref{eq. explicit bound in subconvexity} yields the upper bound

\[\cJ\ll_{\varepsilon}\frac{(4\pi)^{\kappa}}{\Gamma(\kappa)}\left(X\left(\frac{\kappa^{1+\varepsilon}}{X^{\frac{1}{2}}}+X^2\kappa^{\frac{1}{2}+\varepsilon}\right)+X^2\left(\frac{\kappa^{1+\varepsilon}}{X}+X^4\kappa^{\frac{1}{2}+\varepsilon}\right)\right).\]
Combining the lower and upper bound for $\cJ$ gives
\[\frac{L(1/2,f_0)^2}{\|f_0\|^2}\ll_{\varepsilon}\frac{(4\pi)^{\kappa}}{\Gamma(\kappa)}\left( \frac{\kappa^{1+\varepsilon}}{X}+X^4\kappa^{\frac{1}{2}+\varepsilon}\right).\]
Now apply Lemma \ref{lem. petersson norm lemma} and set $X=\kappa^{\frac{1}{10}}$ to obtain
\[L\left(\frac{1}{2},f_0\right)\ll_{\varepsilon} \kappa^{\frac{9}{20}+\varepsilon}=(\kappa^{2})^{\frac{1}{4}-\frac{1}{40}+\varepsilon}.\]
This completes the proof of Theorem \ref{thm. subconvexity theorem}.
\end{proof}

\subsection{Proof of Theorem \ref{thm. nonvanishing theorem}}\label{subsec. proof B} Before the proof of Theorem \ref{thm. nonvanishing theorem}, we prove a result on the first moment of central values. This result is extensively studied in \cite{RamakrishnanShankhadhar2014,KohnenRaji2017} and we only sketch the proof.
\begin{prop}\label{prop. first moment}
	Let $D$ be a fixed fundamental discriminant satisfying $(-1)^{\kappa-\frac{1}{2}}D>0.$ Then
	\[\sum_{f\in H_{\kappa}^+(4)}\frac{a_f(|D|)L\left(\frac{1}{2},f\right)}{L(1,\sym^2F)}\gg \kappa.\]
\end{prop}
\begin{proof}
 By \cite[Equation (18),(21)]{KohnenRaji2017}, we obtain
	\[\frac{(2\pi)^{\frac{\kappa}{2}}\Gamma\left(\frac{\kappa}{2}\right)\Gamma(\kappa-1)}{2(4\pi)^{\kappa-1}}\sum_{f\in H_{\kappa}^+(4)}\frac{L\left(\frac{1}{2},f\right)}{\|f\|^2}f(z)=R_{\frac{\kappa}{2},\kappa-1}|\textrm{pr}(z),\]
where $|\textrm{pr}$ is the projection operator defined in \cite[Equation (4)]{KohnenRaji2017}. We note that our normalization differs slightly from that used in \cite{KohnenRaji2017}. Taking the $|D|$-th Fourier coefficient on each side and applying \cite[Lemma 1]{KohnenRaji2017} gives
	\begin{equation}\label{eq. rough first moment}
		\frac{(2\pi)^{\frac{\kappa}{2}}\Gamma\left(\frac{\kappa}{2}\right)\Gamma(\kappa-1)}{2(4\pi)^{\kappa-1}}|D|^{\frac{\kappa-1}{2}}\sum_{f\in H_{\kappa}^+(4)}\frac{a_f(|D|)L\left(\frac{1}{2},f\right)}{\|f\|^2}=\int_{iC}^{iC+1}R_{\frac{\kappa}{2},\kappa-1}|\textrm{pr}(z)e(-|D|z)\,dz.
	\end{equation}
	Then by \cite[Lemma 1]{KohnenRaji2017} and the bound
	\[\frac{\Gamma\left(\frac{\kappa}{2}\right)^2}{\Gamma(\kappa)}{}_1F_1\left(\frac{\kappa}{2},\kappa;z\right)\ll \frac{1}{2^{\kappa}},\]
	(see \cite[Equation 9.211(2)]{GradshteynRyzhik2007}), we conclude that
	\[\frac{(2\pi)^{\frac{\kappa}{2}}\Gamma\left(\frac{\kappa}{2}\right)\Gamma(\kappa-1)}{2(4\pi)^{\kappa-1}}\sum_{f\in H_{\kappa}^+(4)}\frac{a_f(|D|)L\left(\frac{1}{2},f\right)}{\|f\|^2}\gg (2\pi)^{\frac{\kappa}{2}}\Gamma\left(\frac{\kappa}{2}\right).\]
	Notice that $D$ is fixed. Then inserting Lemma \ref{lem. petersson norm lemma} we complete the proof.
\end{proof}

We introduce the indicator function $\mathbf{1}_{a\neq b}$, which takes the value $1$ if $a\neq b$ and $0$ if $a=b$. In particular, we will use H\"older's inequality with $\mathbf{1}_{L(1/2,f)a_f(|D|)\neq 0}$. Furthermore, by the Kohnen-Zagier formula, we have
\begin{equation}\label{eq. two indicator fucntions}
  \mathbf{1}_{L(1/2,f)a_f(|D|)\neq 0}=\mathbf{1}_{L(1/2,f)L(1/2,F\times\chi_D)\neq 0}.  
\end{equation}
This will allow us to detect simultaneous non-vanishing of the $L$-values. Finally, we prove Theorem \ref{thm. nonvanishing theorem}:
\begin{proof}
	We apply Proposition \ref{prop. first moment} and H\"older's inequality, and obtain 
	\begin{equation}\label{eq. result after holder}
	\begin{split}
		\kappa\ll\sum_{f\in H_{\kappa}^+(4)}\frac{a_f(|D|)L\left(\frac{1}{2},f\right)}{L(1,\sym^2F)}&\leq \left(\sum_{f\in H_{\kappa}^+(4)}\frac{\mathbf{1}_{L(1/2,f)a_f(|D|)\neq 0}}{L(1,\sym^2F)}\right)^{\frac{1}{4}}\\		
	&\hspace{15mm}\times\left(\sum_{f\in H_{\kappa}^+(4)}\frac{a_f(|D|)^4}{L(1,\sym^2F)}\right)^{\frac{1}{4}}\left(\sum_{f\in H_{\kappa}^+(4)}\frac{L\left(\frac{1}{2},f\right)^2}{L(1,\sym^2F)}\right)^{\frac{1}{2}}.
	\end{split}
	\end{equation}
For the second term, we apply \eqref{eq. normalization of forms} so that it becomes the second moment of $L\left(\frac{1}{2},F\times\chi_D\right)$. Then we have
\[\sum_{f\in H_{\kappa}^+(4)}\frac{a_f(|D|)^4}{L(1,\sym^2F)}\ll \sum_{F\in H_{2\kappa-1}(1)}\frac{L\left(\frac{1}{2},F\times\chi_D\right)^2}{L(1,\sym^2F)}\ll\kappa\log\kappa.\]
Here $H_{2\kappa-1}(1)$ is the set of Hecke eigenforms of (integral) weight $2\kappa-1$ and level $1.$ For $D=1,$ the proof of the second moment can be found in \cite{BalkanovaFrolenkov21}. Since the fundamental discriminant $D$ is fixed, the proof is identical to that of the $D=1$ case. For the last term, we apply Proposition \ref{prop. second moment n=1}, together with Lemma \ref{lem. petersson norm lemma} and \eqref{eq. kappa constant} to obtain
\[\sum_{f\in H_{\kappa}^+(4)}\frac{L\left(\frac{1}{2},f\right)^2}{L(1,\sym^2F)}\ll \frac{\Gamma(\kappa)}{(4\pi)^{\kappa}}\sum_{f\in H_{\kappa}(4)}\frac{\left|L\left(\frac{1}{2},f\right)\right|^2}{\|f\|^2}\ll \kappa\log\kappa.\]	
Inserting these estimates into \eqref{eq. result after holder} gives
\[\sum_{f\in H_{\kappa}^+(4)}\frac{\mathbf{1}_{L(1/2,f)a_f(|D|)\neq 0}}{L(1,\sym^2F)}\gg \frac{\kappa}{(\log \kappa)^3}.\]
Applying H\"older's inequality,  we obtain
\[\frac{\kappa}{(\log \kappa)^3}\ll \left(\sum_{F\in H_{2\kappa-1}(1)}\frac{1}{L(1,\sym^2F)^n}\right)^{\frac{1}{n}}\left(\sum_{f\in H_{\kappa}^+(4)}\mathbf{1}_{L(1/2,f)a_f(|D|)\neq 0}\right)^{1-\frac{1}{n}},\]
for any fixed $n\gg1$. Following the method in \cite[Section 5]{Luo2001} (see also \cite[Proposition 6.1]{LauWu2006}), we can show
\[\sum_{F\in H_{2\kappa-1}(1)}\frac{1}{L(1,\sym^2F)^n}\ll \kappa.\]
Together with \eqref{eq. two indicator fucntions}, this implies
\[\#\{f\in H_{\kappa}^+(4):L(1/2,f)L(1/2,F\times\chi_D)\neq0\}\gg\frac{\kappa}{(\log\kappa)^{\frac{3n}{n-1}}}.\]
Choosing $n$ large enough completes the proof of Theorem \ref{thm. nonvanishing theorem}.
\end{proof}

\bibliographystyle{alpha}	
\bibliography{ASTC.bib}

@book {AnalyticNumberTheory,
    AUTHOR = {Iwaniec, Henryk and Kowalski, Emmanuel},
     TITLE = {Analytic number theory},
    SERIES = {American Mathematical Society Colloquium Publications},
    VOLUME = {53},
 PUBLISHER = {American Mathematical Society, Providence, RI},
      YEAR = {2004},
     PAGES = {xii+615},
      ISBN = {0-8218-3633-1},
   MRCLASS = {11-02 (11Fxx 11Lxx 11Mxx 11Nxx)},
  MRNUMBER = {2061214},
MRREVIEWER = {K.\ Soundararajan},
       DOI = {10.1090/coll/053},
       URL = {https://doi.org/10.1090/coll/053},
}

@article {BalkanovaFrolenkov21,
    AUTHOR = {Balkanova, Olga and Frolenkov, Dmitry},
     TITLE = {Moments of {$L$}-functions and the {L}iouville-{G}reen method},
   JOURNAL = {J. Eur. Math. Soc. (JEMS)},
  FJOURNAL = {Journal of the European Mathematical Society (JEMS)},
    VOLUME = {23},
      YEAR = {2021},
    NUMBER = {4},
     PAGES = {1333--1380},
      ISSN = {1435-9855,1435-9863},
   MRCLASS = {11F67 (11F12 34M60)},
  MRNUMBER = {4228282},
MRREVIEWER = {Dominic\ A.\ Lanphier},
       DOI = {10.4171/jems/1035},
       URL = {https://doi.org/10.4171/jems/1035},
}

@article {BettinBoberBookerConreyLeeMolteniOliverPlattSteiner2018,
    AUTHOR = {Bettin, Sandro and Bober, Jonathan W. and Booker, Andrew R.
              and Conrey, Brian and Lee, Min and Molteni, Giuseppe and
              Oliver, Thomas and Platt, David J. and Steiner, Raphael S.},
     TITLE = {A conjectural extension of {H}ecke's converse theorem},
   JOURNAL = {Ramanujan J.},
  FJOURNAL = {Ramanujan Journal. An International Journal Devoted to the
              Areas of Mathematics Influenced by Ramanujan},
    VOLUME = {47},
      YEAR = {2018},
    NUMBER = {3},
     PAGES = {659--684},
      ISSN = {1382-4090,1572-9303},
   MRCLASS = {11F11 (11F06 11F66)},
  MRNUMBER = {3874812},
MRREVIEWER = {Dominic\ A.\ Lanphier},
       DOI = {10.1007/s11139-017-9953-y},
       URL = {https://doi.org/10.1007/s11139-017-9953-y},
}

@article {Blomer2011,
    AUTHOR = {Blomer, Valentin},
     TITLE = {Subconvexity for a double {D}irichlet series},
   JOURNAL = {Compos. Math.},
  FJOURNAL = {Compositio Mathematica},
    VOLUME = {147},
      YEAR = {2011},
    NUMBER = {2},
     PAGES = {355--374},
      ISSN = {0010-437X,1570-5846},
   MRCLASS = {11M32 (11L40 11M41)},
  MRNUMBER = {2776608},
MRREVIEWER = {Angel\ V.\ Kumchev},
       DOI = {10.1112/S0010437X10004926},
       URL = {https://doi.org/10.1112/S0010437X10004926},
}

@article {ConreyFarmerKeatingRubinsteinSnaith2005,
    AUTHOR = {Conrey, J. B. and Farmer, D. W. and Keating, J. P. and
              Rubinstein, M. O. and Snaith, N. C.},
     TITLE = {Integral moments of {$L$}-functions},
   JOURNAL = {Proc. London Math. Soc. (3)},
  FJOURNAL = {Proceedings of the London Mathematical Society. Third Series},
    VOLUME = {91},
      YEAR = {2005},
    NUMBER = {1},
     PAGES = {33--104},
      ISSN = {0024-6115,1460-244X},
   MRCLASS = {11M26},
  MRNUMBER = {2149530},
MRREVIEWER = {K.\ Soundararajan},
       DOI = {10.1112/S0024611504015175},
       URL = {https://doi.org/10.1112/S0024611504015175},
}

@article {ConreyGhosh2006,
    AUTHOR = {Conrey, J. Brian and Ghosh, Amit},
     TITLE = {Remarks on the generalized {L}indel\"of hypothesis},
   JOURNAL = {Funct. Approx. Comment. Math.},
  FJOURNAL = {Uniwersytet im. Adama Mickiewicza w Poznaniu. Wydzia\l\
              Matematyki i Informatyki. Functiones et Approximatio
              Commentarii Mathematici},
    VOLUME = {36},
      YEAR = {2006},
     PAGES = {71--78},
      ISSN = {0208-6573,2080-9433},
   MRCLASS = {11M41 (11M26)},
  MRNUMBER = {2296639},
MRREVIEWER = {Antanas\ Laurin\v cikas},
       DOI = {10.7169/facm/1229616442},
       URL = {https://doi.org/10.7169/facm/1229616442},
}

@article {Duke1988,
    AUTHOR = {Duke, W.},
     TITLE = {Hyperbolic distribution problems and half-integral weight
              {M}aass forms},
   JOURNAL = {Invent. Math.},
  FJOURNAL = {Inventiones Mathematicae},
    VOLUME = {92},
      YEAR = {1988},
    NUMBER = {1},
     PAGES = {73--90},
      ISSN = {0020-9910,1432-1297},
   MRCLASS = {11F11 (11E32 11F30 11F37)},
  MRNUMBER = {931205},
MRREVIEWER = {Mark\ Sheingorn},
       DOI = {10.1007/BF01393993},
       URL = {https://doi.org/10.1007/BF01393993},
}

@article {DunnZaharescu2025,
    AUTHOR = {Dunn, Alexander and Zaharescu, Alexandru},
     TITLE = {The twisted second moment of modular half-integral weight
              {$L$}-functions},
   JOURNAL = {J. Eur. Math. Soc. (JEMS)},
  FJOURNAL = {Journal of the European Mathematical Society (JEMS)},
    VOLUME = {27},
      YEAR = {2025},
    NUMBER = {1},
     PAGES = {1--69},
      ISSN = {1435-9855,1435-9863},
   MRCLASS = {11F66 (11F72 11L05 11L07 11L15)},
  MRNUMBER = {4859565},
MRREVIEWER = {Liyang\ Yang},
       DOI = {10.4171/jems/1523},
       URL = {https://doi.org/10.4171/jems/1523},
}

@article {FeigonWhitehouse2009,
    AUTHOR = {Feigon, Brooke and Whitehouse, David},
     TITLE = {Averages of central {$L$}-values of {H}ilbert modular forms
              with an application to subconvexity},
   JOURNAL = {Duke Math. J.},
  FJOURNAL = {Duke Mathematical Journal},
    VOLUME = {149},
      YEAR = {2009},
    NUMBER = {2},
     PAGES = {347--410},
      ISSN = {0012-7094,1547-7398},
   MRCLASS = {11F72 (11F67 11F70)},
  MRNUMBER = {2541706},
MRREVIEWER = {Gergely\ Harcos},
       DOI = {10.1215/00127094-2009-041},
       URL = {https://doi.org/10.1215/00127094-2009-041},
}

@book {GradshteynRyzhik2007,
    AUTHOR = {Gradshteyn, I. S. and Ryzhik, I. M.},
     TITLE = {Table of integrals, series, and products},
   EDITION = {Seventh},
      NOTE = {Translated from the Russian,
              Translation edited and with a preface by Alan Jeffrey and
              Daniel Zwillinger,
              With one CD-ROM (Windows, Macintosh and UNIX)},
 PUBLISHER = {Elsevier/Academic Press, Amsterdam},
      YEAR = {2007},
     PAGES = {xlviii+1171},
      ISBN = {978-0-12-373637-6; 0-12-373637-4},
   MRCLASS = {00A22 (33-00 65-00 65A05)},
  MRNUMBER = {2360010},
}

@article {IwaniecSarnak1995,
    AUTHOR = {Iwaniec, H. and Sarnak, P.},
     TITLE = {{$L^\infty$} norms of eigenfunctions of arithmetic surfaces},
   JOURNAL = {Ann. of Math. (2)},
  FJOURNAL = {Annals of Mathematics. Second Series},
    VOLUME = {141},
      YEAR = {1995},
    NUMBER = {2},
     PAGES = {301--320},
      ISSN = {0003-486X,1939-8980},
   MRCLASS = {11F72 (11F37 58G25 81Q50)},
  MRNUMBER = {1324136},
MRREVIEWER = {Jens\ Bolte},
       DOI = {10.2307/2118522},
       URL = {https://doi.org/10.2307/2118522},
}

@article {Kiral2015,
    AUTHOR = {K\i ral, Eren Mehmet},
     TITLE = {Subconvexity for half integral weight {$L$}-functions},
   JOURNAL = {Math. Z.},
  FJOURNAL = {Mathematische Zeitschrift},
    VOLUME = {281},
      YEAR = {2015},
    NUMBER = {3-4},
     PAGES = {689--722},
      ISSN = {0025-5874,1432-1823},
   MRCLASS = {11M41 (11F37 11F72)},
  MRNUMBER = {3421637},
MRREVIEWER = {Michael\ M.\ Schein},
       DOI = {10.1007/s00209-015-1504-x},
       URL = {https://doi.org/10.1007/s00209-015-1504-x},
}

@article {Kohnen1982,
    AUTHOR = {Kohnen, Winfried},
     TITLE = {Newforms of half-integral weight},
   JOURNAL = {J. Reine Angew. Math.},
  FJOURNAL = {Journal f\"ur die Reine und Angewandte Mathematik. [Crelle's
              Journal]},
    VOLUME = {333},
      YEAR = {1982},
     PAGES = {32--72},
      ISSN = {0075-4102,1435-5345},
   MRCLASS = {10D12},
  MRNUMBER = {660784},
MRREVIEWER = {Wen-Ch'ing\ Winnie\ Li},
       DOI = {10.1515/crll.1982.333.32},
       URL = {https://doi.org/10.1515/crll.1982.333.32},
}

@article {KohnenRaji2017,
    AUTHOR = {Kohnen, Winfried and Raji, Wissam},
     TITLE = {Non-vanishing of {L}-functions associated to cusp forms of
              half-integral weight in the plus space},
   JOURNAL = {Res. Number Theory},
  FJOURNAL = {Research in Number Theory},
    VOLUME = {3},
      YEAR = {2017},
     PAGES = {Paper No. 6, 8},
      ISSN = {2522-0160,2363-9555},
   MRCLASS = {11F37 (11F25)},
  MRNUMBER = {3633401},
MRREVIEWER = {Shaul\ Zemel},
       DOI = {10.1007/s40993-017-0072-z},
       URL = {https://doi.org/10.1007/s40993-017-0072-z},
}

@article {KohnenZagier1981,
    AUTHOR = {Kohnen, W. and Zagier, D.},
     TITLE = {Values of {$L$}-series of modular forms at the center of the
              critical strip},
   JOURNAL = {Invent. Math.},
  FJOURNAL = {Inventiones Mathematicae},
    VOLUME = {64},
      YEAR = {1981},
    NUMBER = {2},
     PAGES = {175--198},
      ISSN = {0020-9910,1432-1297},
   MRCLASS = {10D24 (10D12)},
  MRNUMBER = {629468},
MRREVIEWER = {Marie-France\ Vign\'eras},
       DOI = {10.1007/BF01389166},
       URL = {https://doi-org.revproxy.brown.edu/10.1007/BF01389166},
}

@article {LauWu2006,
    AUTHOR = {Lau, Yuk-Kam and Wu, Jie},
     TITLE = {A density theorem on automorphic {$L$}-functions and some
              applications},
   JOURNAL = {Trans. Amer. Math. Soc.},
  FJOURNAL = {Transactions of the American Mathematical Society},
    VOLUME = {358},
      YEAR = {2006},
    NUMBER = {1},
     PAGES = {441--472},
      ISSN = {0002-9947,1088-6850},
   MRCLASS = {11F67 (11F30)},
  MRNUMBER = {2171241},
MRREVIEWER = {Emmanuel\ P.\ Royer},
       DOI = {10.1090/S0002-9947-05-03774-8},
       URL = {https://doi.org/10.1090/S0002-9947-05-03774-8},
}

@article {Luo2001,
    AUTHOR = {Luo, Wenzhi},
     TITLE = {Nonvanishing of {$L$}-values and the {W}eyl law},
   JOURNAL = {Ann. of Math. (2)},
  FJOURNAL = {Annals of Mathematics. Second Series},
    VOLUME = {154},
      YEAR = {2001},
    NUMBER = {2},
     PAGES = {477--502},
      ISSN = {0003-486X,1939-8980},
   MRCLASS = {11M26 (11F66 11F67)},
  MRNUMBER = {1865978},
MRREVIEWER = {Kazuyuki\ Hatada},
       DOI = {10.2307/3062104},
       URL = {https://doi.org/10.2307/3062104},
}

@article {LuoRamakrishnan1997,
    AUTHOR = {Luo, Wenzhi and Ramakrishnan, Dinakar},
     TITLE = {Determination of modular forms by twists of critical
              {$L$}-values},
   JOURNAL = {Invent. Math.},
  FJOURNAL = {Inventiones Mathematicae},
    VOLUME = {130},
      YEAR = {1997},
    NUMBER = {2},
     PAGES = {371--398},
      ISSN = {0020-9910,1432-1297},
   MRCLASS = {11F67 (11F11 11F37 11F66)},
  MRNUMBER = {1474162},
MRREVIEWER = {Jiandong\ Guo},
       DOI = {10.1007/s002220050189},
       URL = {https://doi.org/10.1007/s002220050189},
}

@incollection {Michel2007,
    AUTHOR = {Michel, Philippe},
     TITLE = {Analytic number theory and families of automorphic
              {$L$}-functions},
 BOOKTITLE = {Automorphic forms and applications},
    SERIES = {IAS/Park City Math. Ser.},
    VOLUME = {12},
     PAGES = {181--295},
 PUBLISHER = {Amer. Math. Soc., Providence, RI},
      YEAR = {2007},
      ISBN = {978-0-8218-2873-1},
   MRCLASS = {11F66 (11M26 11M41)},
  MRNUMBER = {2331346},
MRREVIEWER = {Emmanuel\ P.\ Royer},
       DOI = {10.1090/pcms/012/05},
       URL = {https://doi-org.revproxy.brown.edu/10.1090/pcms/012/05},
}

@article {Niwa1975,
    AUTHOR = {Niwa, Shinji},
     TITLE = {Modular forms of half integral weight and the integral of
              certain theta-functions},
   JOURNAL = {Nagoya Math. J.},
  FJOURNAL = {Nagoya Mathematical Journal},
    VOLUME = {56},
      YEAR = {1975},
     PAGES = {147--161},
      ISSN = {0027-7630,2152-6842},
   MRCLASS = {10D15},
  MRNUMBER = {364106},
MRREVIEWER = {I.\ Piatetski-Shapiro},
       URL = {http://projecteuclid.org/euclid.nmj/1118795278},
}

@book {Peng2001,
    AUTHOR = {Peng, Zhuangzhuang},
     TITLE = {Zeros and central values of automorphic {L}-functions},
      NOTE = {Thesis (Ph.D.)--Princeton University},
 PUBLISHER = {ProQuest LLC, Ann Arbor, MI},
      YEAR = {2001},
     PAGES = {73},
      ISBN = {978-0493-16728-2},
   MRCLASS = {99-05},
  MRNUMBER = {2701928},
       URL =
              {http://gateway.proquest.com/openurl?url_ver=Z39.88-2004&rft_val_fmt=info:ofi/fmt:kev:mtx:dissertation&res_dat=xri:pqdiss&rft_dat=xri:pqdiss:3007225},
}

@article {Purkait2014,
    AUTHOR = {Purkait, Soma},
     TITLE = {Hecke operators in half-integral weight},
   JOURNAL = {J. Th\'eor. Nombres Bordeaux},
  FJOURNAL = {Journal de Th\'eorie des Nombres de Bordeaux},
    VOLUME = {26},
      YEAR = {2014},
    NUMBER = {1},
     PAGES = {233--251},
      ISSN = {1246-7405,2118-8572},
   MRCLASS = {11F37 (11F11)},
  MRNUMBER = {3232773},
MRREVIEWER = {\.Ilker\ \.Inam},
       DOI = {10.5802/jtnb.865},
       URL = {https://doi.org/10.5802/jtnb.865},
}

@article {RudnickSarnak1994,
    AUTHOR = {Rudnick, Ze\'ev and Sarnak, Peter},
     TITLE = {The behaviour of eigenstates of arithmetic hyperbolic
              manifolds},
   JOURNAL = {Comm. Math. Phys.},
  FJOURNAL = {Communications in Mathematical Physics},
    VOLUME = {161},
      YEAR = {1994},
    NUMBER = {1},
     PAGES = {195--213},
      ISSN = {0010-3616,1432-0916},
   MRCLASS = {11F72 (11F32 11F37 81Q50)},
  MRNUMBER = {1266075},
MRREVIEWER = {Jens\ Bolte},
       URL = {http://projecteuclid.org/euclid.cmp/1104269797},
}

@incollection {RamakrishnanShankhadhar2014,
    AUTHOR = {Ramakrishnan, B. and Shankhadhar, Karam Deo},
     TITLE = {Nonvanishing of {$L$}-functions associated to cusp forms of
              half-integral weight},
 BOOKTITLE = {Automorphic forms},
    SERIES = {Springer Proc. Math. Stat.},
    VOLUME = {115},
     PAGES = {223--231},
 PUBLISHER = {Springer, Cham},
      YEAR = {2014},
      ISBN = {978-3-319-11352-4; 978-3-319-11351-7},
   MRCLASS = {11F37 (11F25 11F30 11F66)},
  MRNUMBER = {3297499},
MRREVIEWER = {Wen-Wei\ Li},
       DOI = {10.1007/978-3-319-11352-4\_16},
       URL = {https://doi.org/10.1007/978-3-319-11352-4_16},
}

@article {RR05,
    AUTHOR = {Ramakrishnan, Dinakar and Rogawski, Jonathan},
     TITLE = {Average values of modular {$L$}-series via the relative trace
              formula},
   JOURNAL = {Pure Appl. Math. Q.},
  FJOURNAL = {Pure and Applied Mathematics Quarterly},
    VOLUME = {1},
      YEAR = {2005},
    NUMBER = {4},
     PAGES = {701--735},
      ISSN = {1558-8599,1558-8602},
   MRCLASS = {11M06 (11F72)},
  MRNUMBER = {2200997},
MRREVIEWER = {Emmanuel\ P.\ Royer},
       DOI = {10.4310/PAMQ.2005.v1.n4.a1},
       URL = {https://doi.org/10.4310/PAMQ.2005.v1.n4.a1},
}

@article{Shimura1977,
    AUTHOR = {Shimura, Goro},
     TITLE = {On the periods of modular forms},
   JOURNAL = {Math. Ann.},
  FJOURNAL = {Mathematische Annalen},
    VOLUME = {229},
      YEAR = {1977},
    NUMBER = {3},
     PAGES = {211--221},
      ISSN = {0025-5831,1432-1807},
   MRCLASS = {10D15},
  MRNUMBER = {463119},
MRREVIEWER = {K.-B.\ Gundlach},
       DOI = {10.1007/BF01391466},
       URL = {https://doi.org/10.1007/BF01391466},
}

@article {Young2017,
    AUTHOR = {Young, Matthew P.},
     TITLE = {Weyl-type hybrid subconvexity bounds for twisted
              {$L$}-functions and {H}eegner points on shrinking sets},
   JOURNAL = {J. Eur. Math. Soc. (JEMS)},
  FJOURNAL = {Journal of the European Mathematical Society (JEMS)},
    VOLUME = {19},
      YEAR = {2017},
    NUMBER = {5},
     PAGES = {1545--1576},
      ISSN = {1435-9855,1435-9863},
   MRCLASS = {11M41 (11F67 11M06)},
  MRNUMBER = {3635360},
MRREVIEWER = {D.\ R.\ Heath-Brown},
       DOI = {10.4171/JEMS/699},
       URL = {https://doi.org/10.4171/JEMS/699},
}

@article{Wei2025,
      title={A note on a classical relative trace formula}, 
      author={Zhining Wei},
      journal={arXiv preprint arXiv:2502.18593},
      year={2025}
}

@article{Yan23a,
	title={Relative Trace Formula and {$L$}-functions for {$\mathrm{GL}(n+1)\times \mathrm{GL}(n)$}},
	author={Yang, Liyang},
  journal={arXiv preprint arXiv:2303.02225},
  year={2023}
  }

@article{WeiYangZhao2024,
      title={Relative Trace Formula and Uniform non-vanishing of Central {$L$}-values of {H}ilbert Modular Forms}, 
      author={Zhining Wei and Liyang Yang and Shifan Zhao},
      journal={arXiv preprint arXiv:2410.09593},
    year={2024}
}

@incollection {Zagier1977,
    AUTHOR = {Zagier, D.},
     TITLE = {{T}he {E}ichler-{S}elberg trace formula on
              {${\rm SL}\sb{2}({\bf Z})$}, {A}ppendix to \it{{I}ntroduction to modular
              forms}, by {S}. {L}ang, {S}pringer, {B}erlin, 1976; errata in {L}ecture {N}otes in {M}ath., {V}ol. 627, {S}pringer-{V}erlag, 171-173, 1977.},
      ISBN = {3-540-08530-0},
   MRCLASS = {10D15},
  MRNUMBER = {480354},
MRREVIEWER = {A.\ N.\ Andrianov},
}

@article {BlomerEpstein,
    AUTHOR = {Blomer, Valentin},
     TITLE = {Epstein zeta-functions, subconvexity, and the purity
              conjecture},
   JOURNAL = {J. Inst. Math. Jussieu},
  FJOURNAL = {Journal of the Institute of Mathematics of Jussieu. JIMJ.
              Journal de l'Institut de Math\'ematiques de Jussieu},
    VOLUME = {19},
      YEAR = {2020},
    NUMBER = {2},
     PAGES = {581--596},
      ISSN = {1474-7480,1475-3030},
   MRCLASS = {11M36 (11E45 11F68 58J50)},
  MRNUMBER = {4079154},
MRREVIEWER = {Spencer\ Leslie},
       DOI = {10.1017/s1474748018000142},
       URL = {https://doi.org/10.1017/s1474748018000142},
}

@article{yoshida1995calculations,
  title={On calculations of zeros of various $ L $-functions},
  author={Yoshida, Hiroyuki},
  journal={Journal of Mathematics of Kyoto University},
  volume={35},
  number={4},
  pages={663--696},
  year={1995},
  publisher={Duke University Press}
}
\end{document}